\documentclass{amsart}

\usepackage{curves}
\usepackage{epic}
\usepackage{eepic}
\usepackage{amsmath}
\usepackage{amscd}
\usepackage{amsfonts}
\usepackage{amssymb}
\usepackage{latexsym}
\usepackage{pifont}
\usepackage{array}

\newtheorem{thm}{Theorem}[section]
\newtheorem{pro}[thm]{Proposition}
\newtheorem{lem}[thm]{Lemma}
\newtheorem{cor}[thm]{Corollary}
\newtheorem{thm&def}[thm]{Theorem \& Definition}

\theoremstyle{definition}
\newtheorem{defi}[thm]{Definition}
\newtheorem{deflem}[thm]{Definition \& Lemma}

\newtheorem{exa}[thm]{Example}

\theoremstyle{remark}
\newtheorem{rmk}[thm]{Remark}

\numberwithin{equation}{section}\newcommand{\ncm}{\newcommand}

\ncm{\Cat}{\mathsf{Cat}}
\ncm{\CAT}{\mathsf{CAT}}
\ncm{\ob}{\operatorname{ob}}
\ncm{\Nat}{\operatorname{Nat}}
\ncm{\Set}{\mathsf{Set}}
\ncm{\Ab}{\mathsf{Ab}}
\ncm{\Add}{\mathsf{Add}}
\ncm{\Fun}{\mathsf{Fun}}
\ncm{\Mnd}{\mathsf{Mnd}}
\ncm{\MonCat}{\mathsf{MonCat}}
\ncm{\MonFun}{\mathsf{MonFun}}
\ncm{\M}{\mathcal{M}}
\ncm{\Elt}{\mathsf{Elt}}
\ncm{\proj}{\mathsf{proj}}

\ncm{\T}{\mathsf{T}}
\ncm{\Q}{\mathsf{Q}}
\ncm{\GT}{\mathcal{T}}

\ncm{\A}{\mathcal{A}}
\ncm{\B}{\mathcal{B}}
\ncm{\C}{\mathcal{C}}
\ncm{\D}{\mathcal{D}}
\ncm{\E}{\mathcal{E}}
\ncm{\F}{\mathcal{F}}
\ncm{\G}{\mathcal{G}}
\ncm{\Ha}{\mathcal{H}}
\ncm{\I}{\mathcal{I}}
\ncm{\Ipre}{\hat I}
\ncm{\J}{\mathcal{J}}
\ncm{\K}{\mathcal{K}}
\ncm{\Ll}{\mathcal{L}}
\ncm{\Pee}{\mathcal{P}}
\ncm{\R}{\mathcal{R}}
\ncm{\U}{\mathcal{U}}
\ncm{\V}{\mathcal{V}}
\ncm{\Y}{\mathcal{Y}}
\ncm{\Z}{\mathcal{Z}}

\ncm{\dom}{\operatorname{dom}}
\ncm{\codom}{\operatorname{codom}}
\ncm{\End}{\operatorname{End}}
\ncm{\Aut}{\operatorname{Aut}}
\ncm{\Hom}{\operatorname{Hom}}
\ncm{\kernel}{\operatorname{ker}}
\ncm{\Ker}{\operatorname{Ker}}
\ncm{\Ima}{\operatorname{Im}}
\ncm{\coker}{\operatorname{coker}}
\ncm{\Coker}{\operatorname{Coker}}
\ncm{\im}{\operatorname{im}}
\ncm{\coim}{\operatorname{coim}}
\ncm{\id}{\operatorname{id}}
\ncm{\Center}{\operatorname{Center}}
\ncm{\colim}{\operatorname{colim}}

\ncm{\ci}{\circ}
\ncm{\bu}{\bullet}
\ncm{\bo}{\,\Box\,}
\ncm{\coten}[1]{\underset{\ssst #1}{\bo}}

\ncm{\ot}{\otimes}
\ncm{\x}{\times}
\ncm{\obar}[1]{\,\underset{\scriptscriptstyle #1}{\bar{\otimes}}\,}
\ncm{\xover}[1]{\underset{\scriptscriptstyle #1}{\times}}
\ncm{\am}[1]{\underset{\scriptscriptstyle #1}{\ot}}
\ncm{\amo}[1]{\underset{\scriptstyle #1}{\ot}}
\ncm{\mash}{\Pisymbol{psy}{35}}
\ncm{\mashed}[1]{\underset{\scriptscriptstyle #1}{\Pisymbol{psy}{35}}}
\ncm{\cross}[1]{\underset{\scriptstyle #1}{\rtimes}}
\ncm{\rarr}[1]{\stackrel{#1}{\longrightarrow}}
\ncm{\larr}[1]{\stackrel{#1}{\longleftarrow}}
\ncm{\mapsot}{\leftarrow\!\!\!\raisebox{1pt}{$\scriptscriptstyle |$}}
\ncm{\oR}{\am{R}}
\ncm{\oL}{\am{L}}
\ncm{\cop}{\Delta}
\ncm{\minT}{^{(-1)}}
\ncm{\nulT}{^{(0)}}
\ncm{\oneT}{^{(1)}}
\ncm{\twoT}{^{(2)}}
\ncm{\threeT}{^{(3)}}
\ncm{\oneB}{_{(1)}}
\ncm{\twoB}{_{(2)}}
\ncm{\threeB}{_{(3)}}
\ncm{\eps}{\varepsilon}
\ncm{\op}{\mathrm{op}}
\ncm{\coop}{\mathrm{coop}}
\ncm{\rev}{\mathrm{rev}}
\ncm{\bra}{\langle}
\ncm{\ket}{\rangle}
\ncm{\co}{\text{co-}}

\ncm{\asso}{\mathbf{a}}
\ncm{\luni}{\mathbf{l}}
\ncm{\runi}{\mathbf{r}}
\ncm{\iso}{\stackrel{\sim}{\rightarrow}}
\ncm{\iiso}{\rarr{\sim}}
\ncm{\ract}{\triangleleft}
\ncm{\lact}{\triangleright}
\ncm{\under}{\mbox{\,\rm\_}\,}

\ncm{\adjoint}{\dashv}
\ncm{\into}{\hookrightarrow}
\ncm{\fgp}{\mathrm{fgp}}
\ncm{\adj}{\dashv}
\ncm{\sst}{\scriptstyle}
\ncm{\ssst}{\scriptscriptstyle}

\ncm{\eqby}[1]{\stackrel{(\ref{#1})}{=}}
\ncm{\lef}{{\ssst <}}
\ncm{\righ}{{\ssst >}}

\ncm{\ZZ}{\mathbb{Z}}
\ncm{\GG}{\mathbf{G}}
\ncm{\FF}{\mathbb{F}}
\ncm{\NN}{\mathbb{N}}

\ncm{\philong}{\boldsymbol{\phi}}
\ncm{\longf}{\mathbf{F}}
\ncm{\longF}{\boldsymbol{\F}}
\ncm{\longg}{\mathbf{G}}
\ncm{\longG}{\boldsymbol{\G}}
\ncm{\longN}{\mathbf{N}}

\ncm{\bimf}{F}
\ncm{\bimg}{G}
\ncm{\bimF}{\F}
\ncm{\bimG}{\G}
\ncm{\bimQ}{Q}
\ncm{\bimK}{\K}
\ncm{\bimk}{K}
\ncm{\bimL}{\Ll}
\ncm{\bimeta}{\eta}
\ncm{\bimeps}{\eps}
\ncm{\bimcop}{\cop}
\ncm{\bimdelta}{\delta}
\ncm{\bimN}{N}
\ncm{\NQ}{N}
\ncm{\bimi}{i}
\ncm{\bimnu}{\nu}
\ncm{\bimtheta}{\theta}
\ncm{\bimT}{\T}
\ncm{\bimmu}{\mu}
\ncm{\bimt}{T}
\ncm{\GroT}{\mathcal{T}}
\ncm{\bimS}{S}

\ncm{\pisharp}{\sharp}
\ncm{\fsh}{F^\pisharp}
\ncm{\Fsh}{\F^\pisharp}
\ncm{\Gsh}{\G^\pisharp}
\ncm{\gsh}{G^\pisharp}
\ncm{\Qsh}{{Q^\pisharp}} 
\ncm{\Ks}{\K^\pisharp}
\ncm{\ksh}{K^\pisharp}
\ncm{\Ls}{\Ll^\pisharp}
\ncm{\epsh}{\eps^\pisharp}
\ncm{\etash}{\eta^\pisharp}
\ncm{\copsh}{\cop^\pisharp}
\ncm{\phish}{\phi}
\ncm{\psh}{\psi}
\ncm{\deltash}{\delta^\pisharp}
\ncm{\Nsh}{N^\pisharp}
\ncm{\ish}{i^\pisharp}
\ncm{\nush}{\nu^\pisharp}
\ncm{\thetash}{\theta^\pisharp}
\ncm{\Tsh}{\T^\pisharp}
\ncm{\mush}{\mu^\pisharp}
\ncm{\tsh}{T^\pisharp}
\ncm{\Ssh}{S^\pisharp}

\ncm{\Ush}{\bra U,\nu_U^{-1}\ket}
\ncm{\Vsh}{\bra V,\nu_V^{-1}\ket}
\ncm{\Wsh}{\bra W,\nu_W^{-1}\ket}
\ncm{\Zsh}{\bra Z,\nu_Z^{-1}\ket}

\ncm{\Pre}{\hat\C}
\ncm{\She}{\Pre_\T}

\ncm{\ellr}{\overrightarrow{\ell}}
\ncm{\elll}{\overleftarrow{\ell}}
\ncm{\eLr}{\overrightarrow{L}}
\ncm{\eLl}{\overleftarrow{L}}

\ncm{\Coend}{\operatorname{Coend}}
\ncm{\ev}{\operatorname{ev}}
\ncm{\db}{\operatorname{db}}
\ncm{\Zee}{\mathsf{Z}}

\ncm{\YI}{{\hat I}}

\ncm{\Phung}{Ph\`ung H{\^o} Hai\ }

\ncm{\parallelpair}{
\parbox{43pt}{
\begin{picture}(43,8)
\put(3,6){\vector(1,0){37}}
\put(3,2){\vector(1,0){37}}
\end{picture}
}}
\ncm{\pair}[2]{\overset{#1}{\underset{#2}{\parallelpair}}}

\ncm{\longrightarrowtail}{
\parbox{40pt}{
\begin{picture}(40,8)
\put(6,4){\line(-1,1){1.5}}
\put(6,4){\line(-1,-1){1.5}}
\put(6,4){\vector(1,0){31}}
\end{picture}
}}

\ncm{\longmono}[1]{\overset{#1}{\longrightarrowtail}}

\ncm{\longerrightarrowtail}{
\parbox{60pt}{
\begin{picture}(60,8)
\put(6,4){\line(-1,1){1.5}}
\put(6,4){\line(-1,-1){1.5}}
\put(6,4){\vector(1,0){51}}
\end{picture}
}}

\ncm{\longrarr}[1]{
\overset{#1}{
\parbox{40pt}{
\begin{picture}(40,8)
\put(3,4){\vector(1,0){34}}
\end{picture}
}}}

\ncm{\longerrarr}[1]{
\overset{#1}{
\parbox{70pt}{
\begin{picture}(70,8)
\put(3,3){\vector(1,0){64}}
\end{picture}
}}}

\ncm{\antiparallelpair}{
\parbox{23pt}{
\begin{picture}(23,4)
\put(3,3){\vector(1,0){17}}
\put(20,1){\vector(-1,0){17}}
\end{picture}
}}

\ncm{\dualpair}[2]{\overset{#1}{\underset{#2}{\antiparallelpair}}}

\ncm{\epi}[1]{\overset{#1}{\twoheadrightarrow}}

\ncm{\longepi}[1]{
\overset{#1}{
\parbox{40pt}{
\begin{picture}(40,8)
\put(2,4){\vector(1,0){32}} \put(37,4){\vector(1,0){0}}
\end{picture}
}}}

\begin{document}

\title[Fiber functors, sheaves and bialgebroids]{Fiber functors, monoidal sites and\\
Tannaka duality for bialgebroids$^{(*)}$}

\author{K. Szlach\'anyi}

\begin{abstract}
What are the fiber functors on small additive monoidal categories $\C$ which are not abelian? 
We give an answer which leads to a new Tannaka duality theorem for bialgebroids generalizing
earlier results by Ph\`ung H{\^o} Hai. The construction reveals a sheaf theoretic interpretation in so far as
the reconstructed bialgebroid $H$ has comodule category equivalent to the category of $\GroT$-sheaves
w.r.t. a monoidal Grothendieck topology on $\C$. We also prove an existence theorem for fiber functors
on small additive monoidal categories with bounded fusion and weak kernels. 
For certain autonomous categories a generalized Ulbrich Theorem can be formulated which relates fiber 
functors to Hopf algebroid Galois extensions. 

\end{abstract}

\thanks{$^{(*)}$Preliminary version. Comments are welcome: \texttt{szlach@rmki.kfki.hu}\\
\indent Supported by the Hungarian Scientific Research Fund, OTKA K-68195}
\maketitle

\tableofcontents

\section{Introduction}

Let $\C$ be a small monoidal $\Ab$-category and $F:\C\to\,_R\M_R$ an additive strong monoidal functor
to the category of bimodules over some ring $R$. Non-commutative Tannaka duality consists of constructing a 
quantum groupoid or, more generally, a monoidal comonad $\bimQ$ and a universal factorization of $F$ through the forgetful functor
$\M^\bimQ\to\,_R\M_R$ of the category of $\bimQ$-comodules. 
\begin{equation} \label{tria-1}
\parbox{130pt}{
\setlength{\unitlength}{0.3mm}
\begin{picture}(130,85)
\put(0,5){$\C$}
\put(16,8){\vector(1,0){60}} \put(45,13){$F$}
\put(82,5){$_R\M_R$}
\put(89,72){$\M^\bimQ$}
\put(15,17){\vector(4,3){65}} \put(40,47){$K$}
\put(95,65){\vector(0,-1){47}} \put(99,43){$\F^\bimQ$}
\end{picture}
}
\end{equation}
This `Tannaka construction' can be done for quite general $\C$ and $F$. 
The difficult part of the problem is to find conditions that 
guarantee that $K$ establishes a monoidal equivalence of $\C$ with a distinguished full subcategory 
of $\M^\bimQ$ which is usually a subcategory $\M^\bimQ_f$ of comodules obeying a finiteness condition.
This equivalence is usually referred to as the Representation Theorem. 

If we replace $_R\M_R$ with the category $\M_k$ of modules over the commutative ring $k$ we are in the framework
of `commutative' Tannaka dualities of Saavedra-Rivano \cite{Saavedra-Rivano}, Deligne and Milne \cite{Deligne-Milne}
and Ulbrich \cite{Ulbrich}, see also \cite{Schauenburg: Thesis} and \cite{Jo-St}.
In \cite{Day96} Day considers arbitrary closed symmetric monoidal categories as targets of the fiber functor. 
McCrudden in \cite{McCrudden} generalizes this further by considering braided target categories.
In their celebrated works Deligne \cite{Deligne} and Doplicher and Roberts \cite{DR} went beyond Tannaka duality in that they proved existence and uniqueness of the fiber functor, establishing in this way an 
`abstract duality' theorem. 

The need of a non-braided target category such as $_R\M_R$ emerged in the 90's when, motivated by several areas of mathematics and physics, various authors proposed groupoid-like generalizations of the notion of Hopf algebra
\cite{Maltsiniotis}, 
\cite{Hayashi}, \cite{BSz96}, \cite{Lu}, 
\cite{Schauenburg: ddqg}, \cite{Sz: EM}, \cite{Day-Street}, \cite{B-Sz}. The most fundamental among them is the
notion of bialgebroid which had been actually invented much earlier by Takeuchi under the name $\x_R$-bialgebras
\cite{Takeuchi}. These `quantum groupoids' $H$ all share in the property that their comodule categories $\M^H$ have strong monoidal forgetful functors to some bimodule category $_R\M_R$ where $R$, the non-commutative base ring of $H$, generalizes the algebra of functions on the objects of a groupoid.
Tannaka duality, resp. reconstruction theorems in this non-commutative setting has been worked out by Bruguières \cite{Brugi}, Hayashi \cite{Hayashi: Tannaka}, \Phung \cite{Phung}, and recently by Pfeiffer \cite{Pfeif} and McCurdy \cite{McCurdy}.

In all the above  mentioned Tannaka dualities, either over a commutative base or not, the hypotheses on $\C$ include abelianness. Moreover, the fiber functors $F$ are faithful exact strong monoidal with values in the subcategory $_R\M_R^\fgp$ of right dual objects in $_R\M_R$, i.e., $FC$ is finitely generated and projective as right $R$-module for all $C\in\ob\C$. This means that $\M^Q_f$ is chosen to be $\M^Q_\fgp$ which is the full subcategory of $\M^Q$ the underlying bimodules of which are in $_R\M_R^\fgp$. There can be arguments for choosing different $\M^Q_f$ but in the present paper we insist to this tradition. However, 
abelianness of $\C$ will be relaxed for the following reason. For generic $R$ the category $_R\M_R^\fgp$ is not abelian, not all morphisms have kernels or cokernels. Therefore if we once arrive to a Representation Theorem
stating the equivalence $\C\simeq\M^Q_\fgp$ then this will imply severe restrictions on the quantum groupoid
or comonad $Q$. Similar reason lead Bruguières to introduce `semitransitive corings' in his paper \cite{Brugi}.
The Representation Theorem of \Phung \cite[Corollary 2.2.5]{Phung} also uses `semitransitive bialgebroids'
for this reason. 

The motivation of the present paper was to derive Tannaka duality for as general bialgebroids as possible without assuming that $\C$ is abelian. This involved, unfortunately, that we did not know at the beginning what properties
to postulate for a fiber functor. This is why the definition of fiber functor awaits until Section \ref{sec: rep}.
It includes all the faithful exact functors if $\C$ is abelian. 

Diagram (\ref{tria-1}) is reminiscent to the Eilenberg-Moore situation for comonads. If we replace $\C$ with a `larger' category $\Pre$ and assume that $\F:\Pre\to\,_R\M_R$ is a strong monoidal and left adjoint functor then
there is monoidal comonad $\bimQ$ and a universal factorization
\begin{equation}
\parbox{130pt}{
\setlength{\unitlength}{0.3mm}
\begin{picture}(130,85)
\put(0,72){$\Pre$}
\put(15,66){\vector(4,-3){65}} \put(38,33){$\F$}
\put(82,5){$_R\M_R$}
\put(89,72){$\M^\bimQ$}
\put(17,75){\vector(1,0){60}} \put(42,79){$\K$}
\put(95,65){\vector(0,-1){47}} \put(99,43){$\F^\bimQ$}
\end{picture}
}
\end{equation}
where $\K$ is the comparison functor. The analogue of the Representation Theorem is the statement that if $\F$
is comonadic then $\bimK$ is an equivalence. 

If we choose $\Pre$ to be the category $\Add(\C^\op,\Ab)$ of additive presheaves over $\C$ then left Kan extension provides a  connection between the two approaches. 
\begin{equation}
\parbox{130pt}{
\setlength{\unitlength}{0.3mm}
\begin{picture}(130,85)
\put(0,5){$\C$}
\put(16,8){\vector(1,0){60}} \put(42,13){$F$}
\put(0,72){$\Pre$}
\put(15,66){\vector(4,-3){65}} \put(45,46){$\F$}
\put(82,5){$_R\M_R$}
\put(3,16){\vector(0,1){50}} \put(-7,35){$Y$}
\end{picture}
}
\end{equation}
where $\F$ is the left Kan extension of $F$ along the Yoneda embedding $Y:\C\to\Pre$. Using the left Kan extension functor in Tannaka duality appears in Brian Day's paper \cite{Day96}. For a  comparison with \cite{Day96}
one should consider the long forgetful functor $\longf:\C\rarr{F}\,_R\M_R\to\Ab$ instead of $F$. 
Then we can say that we consider only enrichment over $\Ab$ but allow for more general functors than strong monoidal ones: these are the essentially strong monoidal functors \cite{Sz: Brussels}.
(Observe that with non-commutative base the target category of the strong monoidal (part) of the fiber
functor is no longer the same as the category over which $\C$ is enriched. )

If $F$ is flat then the EIlenberg-Moore category $\M^Q$ of comodules 
becomes equivalent to a category $\She$ of sheaves over $\C$ w.r.t. some monoidal Grothendieck topology
$\GroT$. This topology is encoded in the structure of a special left exact monoidal idempotent monad $\T$ on $\Pre$ which arises from the fact that the comparison functor $\K:\Pre\to\M^H$ is the reflection of a monoidal
localization $\Ll$. 
If $F$ obeys also the finiteness condition, i.e., $FC$ is right dual in $_R\M_R$ for $C\in\ob\C$, then
the Eilenberg-Moore construction of the comonad $Q$ from the Kan extension $\F$ reduces to the familiar coend construction \cite{Jo-St} of the bialgebroid $H$ from $F$ \cite{Phung}. Therefore our bialgebroids are always such
that $_RH$ is flat and $\M^H$ is equivalent to a monoidal category of sheaves over $\C$.

\subsection{The outline of the paper}

In the explanatory Section 2 the reader can acquaint with the basic notions of tensor pproduct of additive functors, the Day convolution on the presheaf category $\Pre$ and how flatness and essential strong monoidality are inherited to the left Kan extension functor. 

The monoidal adjunction $\F\dashv\G$, which is discussed in Section 3, immediately yields the left exact monoidal comonad $\bimQ=\F\G$. We investigate also in this section the non-monoidal adjunction $\Fsh\dashv\Gsh$ which yields the left exact comonad $\Qsh$ on $\M_R$. We then construct a category equivalence $\psi:\M^Q\simeq \M^\Qsh$ which, in a `very strong sense', respects the forgetful functor $\phish:\,_R\M_R\to\M_R$. 

In Section 4 the comparison functor $\K:\Pre\to\M^Q$ and its right adjoint are studied. The comparison functor
provides the Tannaka factorization of the fiber functor through the comodule category $\M^Q$. We show that
the right adjoint $\Ll$ is a monoidal localization, i.e., fully faithful with a left exact left adjoint. Then we construct
a monad isomorphism between the left exact idempotent monads $\T$ and $\Tsh$ on $\Pre$ corresponding to the
adjunctions $\K\dashv\Ll$ and $\Ks\dashv\Ls$.

In Section 5 first we study left exact monoidal idempotent monads $\T$ in generality and find the condition that makes their categories of modules $\She$ monoidal. (In general only the Kleisli category is known to be monoidal.)
This motivates the definition of monoidal Grothendieck topologies for which the category of sheaves is precisely the Eilenberg-Moore category $\She$. In case of $\T$ is constructed from a flat essentially strong monoidal functor
$\longf$ we discuss the distingushed sheaf $\longg$ which has a monoid structure and
subgenerates all sheaves. It is, as a functor, the pointwise left dual of the fiber functor and also the preimage
of the bialgebroid $H=\F G$ under the (Kan extended) fiber functor.

Section 6 contains our main results on Tannaka duality. We define the notion of fiber functor in Definition 
\ref{def: fiber} and prove the Representation Theorem and the Reconstruction Theorem 
(see theorems  \ref{thm: rep} and \ref{thm: reco}).

In Section 7 imposing further restrictions on the domain category $\C$ we prove in Theorem \ref{thm: duality}
that \textit{coarse fiber functors} exist. These are the fiber functors for which the Grothendieck topology is the
coarsest on $\C$.

In Section 8 we try to relate different fiber functors and their bialgebroids on the same Cauchy complete
autonomous monoidal category $\C$. If $\C$ admits a coarse fiber functor then we can prove a theorem
inspired by Ulbrich's Theorem on the equivalence of fiber functors on $\M^H$ and faithfully flat Galois extensions 
of the base ring and by its generalizations by Schauenburg \cite{Schauenburg: hbg} and by B\"ohm and Brzezi\'nski 
\cite{BB}. 
Finally, for autonomous $\C$ equipped with a monoidal natural isomorphism between left and right dual
objects we construct an invertible antipode on the reconstructed bialgebroid $H$.

\subsection{On notation and terminology}

As a general principle 
we try to balance between the conventions used in category theory and in Hopf algebra theory.
Whenever possible we use general categorical notation \cite{CWM, MM,Borceux}, e.g. identify objects with their unit arrows. But for monads and comonads a deviation from the usual convention seems more appropriate, see 
\ref{sss: mon comon}. 
The monoidal product of $_R\M_R$, as well as of categories
monoidally comonadic over $_R\M_R$, is denoted by $\oR$. But in order to avoid some ambiguities we are forced
to use another symbol in the comonad $\under\obar{R}H$ associated to the bialgebroid 
(actually of the underlying coring). This leads to the unusual expression for the coproduct: $\cop:H\to H\obar{R}H$. 

Tensor product of additive presheaves, as an instance of Day convolution \cite{Day70}, will be denoted by the
symbol $\odot$. 

Boldface letters usually refer to $\Ab$-valued functors. $_R\M_R$ -valued functors are normal Roman while $\M_R$-valued functors are distinguished by a $\sharp$ sign. For example, the fiber functor will appear in three guises: $\longf$, $F$ and $\fsh$.

\subsubsection{Essentially strong monoidal functors}
A monoidal functor $\C\to\M$ is considered as a triple $\bra F,F_2,F_0\ket$ where $F:\C\to\M$ is a functor, $F_2$ is a natural transformation with components $F_{C,D}:FC\ot FD\to F(C\ot D)$ and $F_0$ is an arrow $R\to F I$ where $R$ and $I$ are the unit objects of $\M$ and $\C$, respectively. These data are subject to obey 3 coherence conditions: 1 hexagon for associativity and 2 squares for unitality. We use the terminology \textit{monoidal/strong monoidal/strict monoidal} functor according to whether $F_2$ and $F_0$ are just arrows or isomorphisms or identities, respectively.
A monoidal functor is called \textit{normal} if $F_0$ is invertible.

Every monoidal functor $\longf:\C\to\M$ maps monoids to monoids, in particular $R=\longf I$ has a monoid structure in $\M$. This leads to an essentially unique factorization of $\longf$ as $\C\rarr{F}\,_R\M_R\to\M$ with $F$ normal
monoidal \cite{Sz: Brussels}. The monoidal functor $\longf$ is called essentially strong monoidal if its normal part $F$ is strong monoidal. Therefore the essentially strong monoidal functors can simply be thought of as the strong monoidal functors to a bimodule category composed with the monoidal forgetful functor $_R\M_R\to\M$.

The normal factorization of the essentialy strong monoidal $\longf$ is to be considered as the zeroth step of
non-commutative Tannaka reconstruction: The reconstruction of the base ring $R$ from the a priori data $\bra\C,\longf\ket$.

\subsubsection{Monads and comonads}\label{sss: mon comon}

A monad $\T$ on a monoidal category $\M$ is denoted as a triple $\bra T,\mu,\eta\ket$ where $T$ is the underlying endofunctor, $\mu:T^2\to T$ and $\eta:\M\to T$ are the multiplication and unit of the monad. 
Dually, a comonad is a triple $\bra Q,\cop,\eps\ket$ where $\cop:Q\to Q^2$ is the comultiplication and $\eps:Q\to\M$
is the counit.
In contrast to the categorist notation the Eilenberg-Moore category of the monad $\T$ is denoted by $\M_\T$ and that of a comonad $\Q$ by $\M^\Q$, complying in this way with the notation used in ring, coring and Hopf algebra theory. 
Accordingly, the objects of $\M_\T$ are called $\T$-modules and the objects of $\M^\Q$ are called $\Q$-comodules.
E.g. the latter are pairs $\bra M,\alpha\ket$ where $M\in\ob\M$ and $\alpha:M\to QM$ is the coaction.

\subsubsection{Flatness}

The notion of flatness of a functor is a substitute for left exactness in the situation where finite limits may not exist in
the domain category. The general definition can be found in \cite{Borceux} which, for addive functors on additive categories, can be rephrased as follows.
At first we define the \textit{category of elements} $\Elt F$ of a functor $F:\C\to\Ab$. It has objects $\bra x,C\ket$ where $C\in\ob\C$ and $x\in FC$ and arrows $\bra x,C\ket\rarr{t}\bra y,D\ket$ those $t\in\C(C,D)$ for which
$Ftx=y$ holds. There is an obvious forgetful functor $\Elt F\to \C$.
Now for an additive category $\C$ an additive functor $F:\C\to \Ab$ is called flat if its category $\Elt F$
of elements is cofiltered, i.e., if
\renewcommand{\labelenumi}{(flat-\arabic{enumi})}
\begin{enumerate}
\item Given objects $A,B$ of $\C$ and elements $x\in FA$, $y\in FB$ there exist an object $C$, arrows
$s:C\to A$, $t:C\to B$ and a $z\in FC$ such that $Fsz=x$ and $Ftz=y$. \label{flat-1}
\item Given an arrow $t:B\to C$ in $\C$ and an element $y\in FB$ such that $Fty=0$ there exist
an arrow $s:A\to B$ and an $x\in FA$ such that $Fsx=y$ and $t\ci s=0$. \label{flat-2}
\end{enumerate}
\renewcommand{\labelenumi}{(\arabic{enumi})}

The first axiom could have been ommitted altogether since for
additive categories $\C$ (flat-1) is automatically satisfied by taking $C$ to be the direct sum of $A$ and
$B$. 

We need also flatness of functors $F:\C\to\M$ where $\M$ is an additive category equipped with a canonical forgetful functor to $\Ab$. In all such cases we shall say that $F$ is flat when the composite $\C\to\M\to\Ab$ is flat in the above sense. If the forgetful functor is left adjoint, as happens for categories of modules of a ring $\M_R\to\Ab$ for example or for categories of comodules of $R$-corings $\M^H\to\M_R\to\Ab$, then this functor preserves left Kan extension so it is practically indifferent whether we take the Kan extension of the $\Ab$-valued,  the $\M_R$-valued or the $\M^H$-valued functor.

If $\C$ has kernels, hence all finite limits, then $F$ is flat precisely when it preserves these limits, i.e., it is left exact.

For purposes of the present paper the most important property of flat functors is the following one
\cite[I. Proposition 6.3.8]{Borceux}: $F$ is flat precisely when its left Kan extension along the Yoneda embedding is
left exact (c.f. Lemma \ref{lem: FF left adj}).

\subsubsection{Bialgebroids}

A bialgebra over $k$ is both a monoid and comonoid in the symmetric monoidal category $\M_k$. A bialgebroid $H$ can be thought of as a `bialgebra over a non-commutative ring $R$'. In fact $H$ is a comonoid in $_R\M_R$ (so an $R$-coring)
and a monoid in $_{R^e}\M_{R^e}$ (so an $R^e$-ring) where $R^e=R^\op\ot R$. The compatibility conditions are rather delicate but as it has been shown in \cite{Sz: EM} they are equivalent to the requirement that the monad $\under\am{R^e}H$ on $\M_{R^e}\equiv\,_R\M_R$ associated to the $R^e$-ring has an opmonoidal structure.
Unfortunately this `bimonad' interpretation of bialgebroids will have no use for this paper since we are interested in comodule categories of bialgebroids for which a monoidal comonad description is the more appropriate \cite{Day-Street}.

Throughout this paper bialgebroid will always mean \textit{right bialgebroid}, as they are called in \cite{Ka-Sz}.
The category of comodules of an $R$-bialgebroid is defined as the category $\M^H$ of comodules of its underlying $R$-coring. The monoidal product on $\M^H$ is introduced by first noticing that a right $R$-module $N$ which is equipped with a right $H$-coaction $N\to N\obar{R}H$, $n\mapsto n\nulT\obar{R}n\oneT$, is automatically an $R$-$R$-bimodule \cite[1.4.]{Phung} in such a way that all $H$-comodule morphisms become $R$-$R$-bimodule morphisms. 
Then the monoidal product of $H$-comodules can be introduced by setting $M\oR N$ to be the right $R$-module $M\oR N$ 
equipped with the coaction $m\oR n\mapsto (m\nulT\oR n\nulT)\obar{R}m\oneT n\oneT$. 
This is a very fortunate interplay between the monad $R\ot\under$ and the comonad $\under\obar{R}H$ on $\M^H$ which  will be generalized from bialgebroids to general left exact monoidal comonads in Subsection \ref{ss: Fsh - Gsh}.

For more about bialgebroids and Hopf algebroids we refer to \cite{Bohm} and the references therein.


\section{Extension of functors to presheaves}

\subsection{Tensor product of functors}

For a small $\Ab$-category $\C$ and a pair of additive functors $U:\C^\op\to\Ab$ and $\longf:\C\to\Ab$ one  defines the abelian group $U\amo{\C}\longf$ as the coequalizer
\begin{equation} \label{tensor}
\coprod_{C,D\in\ob\C}UD\ot\C(C,D)\ot \longf C
\parbox{50pt}{
\begin{picture}(50,40)
\put(5,17){\vector(1,0){40}} \put(22,27){$\mathbf{L}$}
\put(5,23){\vector(1,0){40}} \put(22,8){$\mathbf{R}$}
\end{picture}
}
\coprod_{C\in\ob\C}UC\ot \longf C
\parbox{42pt}{
\begin{picture}(42,40)
\put(5,20){\vector(1,0){30}} \put(38,20){\vector(1,0){0}} 
\end{picture}
}
U\amo{\C} \longf
\end{equation}
in $\Ab$ where the maps $\mathbf{L}$, $\mathbf{R}$ are defined by 
\begin{align*}
\mathbf{L}\ci i_{C,D}(u\ot t\ot x)&=i_C(Ut(u)\ot x)\\
\mathbf{R}\ci i_{C,D}(u\ot t\ot x)&=i_D(u\ot \longf t(x))
\end{align*}
for $x\in \longf C$, $u\in UD$ and $t\in\C(C,D)$.
Equivalently, $U\amo{\C}\longf$ is the coend of the functor $U\ot \longf:\C^\op\x\C\to\Ab$, so we write
\[
U\amo{\C}\longf=\int^C UC\ot \longf C\,.
\]

For $u\ot x\in UC\ot \longf C$ we denote its image in the tensor product by $u\am{C}x$. An arbitrary element of
$U\amo{\C}\longf$ is a finite sum of such rank 1 tensors. The rank 1 tensors obey the relations
\begin{equation} \label{tensor rels}
u\cdot t\am{C}x=u\am{D}t\cdot x\,\qquad u\in UD,\ t\in\C(C,D),\ x\in \longf C
\end{equation}
where we introduced the shorthand notation $u\cdot t:=Ut(u)$ and $t\cdot x:=\longf t(x)$.
\begin{exa}
If $\C$ is a 1 object $\Ab$-category, i.e., a ring $S$, then $U$ is a right $S$-module, $\longf$ is a left $S$-module
and $U\amo{\C}\longf$ is the tensor product of $S$-modules $U\am{S}\longf$.
\end{exa}

For natural transformations $\sigma:U\to U'$ and $\tau:\longf\to \longf'$ one can easily see that
\[
U\amo{\C}\longf\to U'\amo{\C}\longf',\quad u\am{A}x\mapsto \sigma_A(u)\am{A}\tau_A(x)
\]
is a well-defined map of abelian groups and this extends the definition of the tensor product over $\C$ to
a bifunctor
\[
\under\amo{\C}\under\,:\Add(\C^\op,\Ab)\x\Add(\C,\Ab)\to\Ab\,.
\]
Fixing $\longf$ and letting $U$ to vary over the presheaf category $\hat\C:=\Add(\C^\op,\Ab)$ we get a functor
\[
\longF:=\,\under\amo{\C}\longf:\Pre\to\Ab\,.
\]
Composing $\longF$ with the Yoneda embedding $Y:\C\to\hat\C$, $A\mapsto\C(\under,A)$ the relations (\ref{tensor rels}) imply that there is a natural isomorphism
\begin{equation} \label{N}
\longN_A:\longF YA\to \longf A,\qquad s\am{C} x\mapsto \longf s(x)\,.
\end{equation}
As it is shown in \cite[X.4.]{CWM} $\longF$ is the left Kan extension of $\longf$ along $Y$.

By \cite[I. Proposition 6.3.8]{Borceux} $\longf$ is flat precisely when $\longF$ is left exact which is part of the next Lemma.

\begin{lem} \label{lem: FF left adj}
For an additive functor $\longF:\Pre\to\Ab$ consider the conditions:
\begin{enumerate}
\item There is an additive (and flat) functor $\longf:\C\to\Ab$
and a natural isomorphism $\under\amo{\C}\longf\cong\longF$.
\item $\longF$ is left adjoint (and left exact).
\end{enumerate}
Then (1)$\Rightarrow$(2) for any small $\Ab$-category $\C$. If $\C$ is additive then also (2)$\Rightarrow$(1).
\end{lem}
\begin{proof}
$(1)\Rightarrow(2)$ The right adjoint of $\under\amo{\C}\longf$ is the functor
\begin{align*}
\G:\Ab&\to\Pre,\\
X&\mapsto\Ab(\longf\under,X)
\end{align*}
Taking the canonical presentation of $\longf$ as the colimit of representable functors,
\[
\longf=\colim\left((\Elt \longf)^\op\rarr{}\C^\op\rarr{}\Add(\C,\Ab)\right)
\]
and using the fact that for each presheaf $U$ the functor $U\amo{\C}\under:\Add(\C,\Ab)\to\Ab$ is left adjoint, too,
we obtain
\[
\longF U=\colim\left((\Elt \longf)^\op\rarr{}\C^\op\rarr{U}\Ab\right)\,.
\]
Therefore if $\longf$ is flat then $\longF$ is (pointwise) a filtered colimit of abelian groups and therefore commutes with
finite limits.

$(2)\Rightarrow(1)$. Let $\longf:=\longF Y$ and $\longF\dashv\longG$. Then
\begin{align*}
\Ab(U\amo{\C}\longf,X)&\cong\hat\C(U,\Ab(\longf\under,X))=\hat\C(U,\Ab(\longF Y\under,X))\cong\\
&\cong\hat\C(U,\hat\C(Y\under,\longG X))\cong\hat\C(U,\longG X)\cong\\
&\cong\Ab(\longF U,X)
\end{align*}
implying that $\longF\cong\,\under\amo{\C}\longf$.
Assume $\longF$ is left exact. Axiom (flat-1) for flatness of $\longf$  holds by additivity of $\C$. In order to verify
(flat -2) let $t:C\to D$ and $x\in\Ker \longf t$. The kernel $V\rarr{\alpha} YC$ of $Yt$ is the presheaf the elements
$\bra v,B\ket$ of which are the arrows $B\rarr{v}C$ such that $t\ci v=0$. 
Composing the kernel $\longF \alpha$ of $\longf t$ with the canonical epimorphism $\coprod_B VB\ot FD\epi V\am{\C}\longf$ 
and using additivity of $\C$ we see that there is an object $B\in\C$ and a $v\am{B}y\in\longF V$ such that 
$\longf vy=x$.
\end{proof}

\subsection{Monoidal structure on the presheaf category}

As usual $\C$ is called a small monoidal $\Ab$-category if it is a small $\Ab$-category equipped
with a monoidal structure in which the monoidal product $\ot: \C\x\C\to\C$ is additive in both arguments.
The unit object is denoted by $I$ and the coherence natural isomorphisms by $\asso_{A,B,C}$, $\luni_C$ and
$\runi_C$. From now on $\C$ always denotes such a category. Presheaves on $\C$ are always $\Ab$-valued
and additive.

Let $U$ and $V$ be presheaves on $\C$ and define the presheaf $U\odot V$ as follows.
For $C\in\ob\C$ let $(U\odot V)C$ be a coequalizer
\begin{gather*}
\coprod_{C',C'',D',D''\in\ob\C}UD'\ot\C(C',D')\ot VD''\ot \C(C'',D'')\ot \C(C,C'\ot C'')\\
\parbox{50pt}{
\begin{picture}(50,40)
\put(5,17){\vector(1,0){40}} \put(22,27){$\mathbf{L}$}
\put(5,23){\vector(1,0){40}} \put(22,8){$\mathbf{R}$}
\end{picture}
}
\coprod_{C',C''\in\ob\C} UC'\ot VC''\ot \C(C,C'\ot C'')
\parbox{50pt}{
\begin{picture}(50,40)
\put(5,20){\vector(1,0){30}} \put(38,20){\vector(1,0){0}} 
\end{picture}
}
(U\odot V)C
\end{gather*}
where the maps $\mathbf{L}$, $\mathbf{R}$ are defined by 
\begin{align*}
\mathbf{L}\ci i_{C',C'',D',D''}(x\ot t'\ot y\ot t''\ot t)&=i_{C',C''}(Ut'x\ot Vt''y\ot t)\\
\mathbf{R}\ci i_{C',C'',D',D''}(x\ot t'\ot y\ot t''\ot t)&=i_{D',D''}(x\ot y\ot ((t'\ot t'')\ci t)
\end{align*}
In other words, the abelian group $(U\odot V)C$ consists of $\ZZ$-linear combinations of words
\begin{equation} \label{Z-generators of conv.prod.}
[x,y,t]^C_{C',C''}\qquad\text{where}\ x\in UC',\ y\in VC'',\ t\in\C(C,C'\ot C'')
\end{equation}
subject to the relations
\[
[Ut'x,Vt''y,t]^C_{C',C''}=[x,y,(t'\ot t'')\ci t]^C_{D',D''}
\]
where $x\in UD'$, $y\in VD''$, $t\in\C(C,C'\ot C'')$, $t'\in\C(C',D')$, $t''\in\C(C'',D'')$ 
and to the obvious $\ZZ$-linearity relations in all the three arguments.

For an arrow $s:C\to D$ in $C$ let
\[
(U\odot V)s:(U\odot V)D\to(U\odot V)C,\quad [x,y,t]^D_{D',D''}\mapsto[x,y,t\ci s]^C_{D',D''}\,.
\]
This defines the object map of $\odot:\hat\C\x\hat\C\to\hat\C$. The arrow map is
\begin{align*}
(\mu\odot\nu)_C&:(U\odot V)C\to(U'\odot V')C\\
&[x,y,t]^C_{C',C''}\mapsto [\mu_{C'}(x),\nu_{C''}(y),t]^C_{C',C''}
\end{align*}
where $\mu:U\to U'$ and $\nu:V\to V'$ are natural transformations.

This definition of the monoidal product of presheaves is nothing but the expansion, in terms of coproducts and coequalizer, of the coend
\[
U\odot V=\int^{C'}\int^{C''} UC'\ot VC''\ot\C(\under,C'\ot C'')\,.
\]

As for the monoidal unit we set
\[
\YI:=YI=\C(\under,I)
\]
where $I$ is the monoidal unit of $\C$.

The natural isomorphisms for associativity, left and right unitalness of $\odot$ can be given in terms of the
corresponding data of $\ot$ in $\C$ as follows:
\begin{align*}
(\asso_{U,V,W})_A&:(U\odot(V\odot W))A\iso((U\odot V)\odot W)A\\
[x,[y,z,s]^E_{C,D},t]^A_{B,E}&\mapsto[[x,y,1]^{B\ot C}_{B,C},z,\asso_{B,C,D}\ci(B\ot s)\ci t]^A_{B\ot C,D}
\end{align*}
\begin{align*}
\luni_U&:\YI\odot U\iso U\\
(\luni_U)_A&:[s,x,t]^A_{B,C}\mapsto U(\luni_C\ci(s\ot C)\ci t)x\\
\runi_U&:U\odot \YI\iso U\\
(\runi_U)_A&:[x,s,t]^A_{B,C}\mapsto U(\runi_B\ci(B\ot s)\ci t)x
\end{align*}
It is left to the reader to verify that the triple $\bra\hat\C,\odot,\hat I\ket$ together with $\asso$, $\luni$ and $\runi$
as above satisfy the axioms for a monoidal category.

\begin{pro} \label{pro: odot}
The category $\hat\C=\Add(\C^\op,\Ab)$ of presheaves over the monoidal $\Ab$-category $\C$ has a monoidal structure $\odot$, unique up to isomorphism,
such that the Yoneda embedding $Y:\C\to\hat\C$ is strong monoidal and such that $\odot$ preserves colimits in both arguments.
\end{pro}
\begin{proof}
Uniqueness is provided by the fact that every presheaf is the colimit of representables.
Let $\bra \hat\C,\odot,\YI\ket$ be the monoidal structure defined above.
Then the strong monoidal structure for $Y$ is
\begin{align}
\label{Y_2}
Y_{C,D}&:YC\odot YD\iso Y(C\ot D)\\
(Y_{C,D})_B&:[s,s',t]^B_{C',D'}\mapsto (s\ot s')\ci t\notag\\
\label{Y_0}
Y_0&:\YI\rarr{=}YI\,.
\end{align}
If $U=\colim U_i$ and $V=\colim V_j$ then these are pointwise colimits and the $\ot$ of $\Ab$
preserves colimits therefore
\[
UC'\ot VC''\ot\C(\under,C'\ot C'')=\colim_{i,j}U_iC'\ot V_jC''\ot\C(\under,C'\ot C'')\,.
\]
Taking the coend of both hand sides and using the fact that colimits can be interchanged
we arrive to
\[
U\odot V=\colim_{i,j}U_i\odot V_j\,.
\]
\end{proof}

How are the monoidal presheaves related to this monoidal structure?
\begin{lem}
For presheaves $U$, $V$ and $W$ there are natural isomorphisms of abelian groups
\begin{align*}
\hat\C(U\odot V,W)&\cong\Nat(\ot(U\x V),W\ot)\\
\hat\C(\hat I, U)&\cong\Ab(\ZZ,UI)
\end{align*}
where, here, $\Nat$ stands for the hom group in the functor category $[\C^\op\x\C^\op,\Ab]$.
\end{lem}
\begin{proof}
To the arrow $\mu:U\odot V\to W$ associate the natural transformation $\nu:\ot(U\x V)\to W\ot$ by 
\[
\nu_{C',C''}(x\ot y):=\mu_{C'\ot C''}([x,y,1]^{C'\ot C''}_{C',C''})
\]
and check that its inverse associates to $\nu$ the arrow
\[
\mu_C([x,y,t]^C_{C',C''})=Wt\ci\nu_{C',C''}(x\ot y)\,.
\]
As for the second isomorphism notice that any $\varphi:\hat I\to U$ has the form $\varphi_C(s)=Usr$ for a unique
$r\in UI$.
\end{proof}
\begin{cor} \label{cor: mon pre= mon in Pre}
The monoids in $\hat C$ are precisely the monoidal presheaves on $\C$.
\end{cor}
\begin{proof}
Let $U$ be a presheaf. Then data $\bra U,\mu,\eta\ket$ where $\mu: U\odot U\to U$ and $\eta:\YI\to U$ are in bijection with data $\bra U,U_2,U_0\ket$ where $U_2:\ot(U\x U)\to U\ot$ and
$U_0:\ZZ\to UI$ by the previous Lemma. Computing the two hand sides of the associativity
condition for $\mu$ on generic (rank 1) elements of $U\odot(U\odot U)$ we obtain
\begin{gather*}
\left(\mu\ci(\mu\odot U)\ci\asso_{U,U,U}\right)_E\left([u_1,[u_2,u_3,s]^B_{C,D},t]^E_{A,B}
\right)=\\
=\mu_E\ci(\mu\odot U)_E\left([[u_1,u_2,1]^{A\ot C}_{A,C},u_3,\asso_{A,C,D}\ci(A\ot s)\ci t
]^E_{A\ot C,D}\right)=\\
=\mu_E\left([\nu_{A,C}(u_1\ot u_2),u_3,\asso_{A,C,D}\ci(A\ot s)\ci t]^E_{A\ot C,D}\right)=\\
=Ut\ci U(A\ot s)\ci U\asso_{A,C,D}\ci \nu_{A\ot C,D}(\nu_{A,C}(u_1\ot u_2)\ot u_3)
\end{gather*}
and
\begin{gather*}
\left(\mu\ci(U\odot\mu)\right)_E\left([u_1,[u_2,u_3,s]^B_{C,D},t]^E_{A,B}\right)=\\
\mu_E\left([u_1,Us\ci\nu_{C.D}(u_2\ot u_3),t]^E_{A,B}\right)=\\
=Ut\ci\nu_{A,B}(u_1\ot(Us\ci\nu_{C,D}(u_2\ot u_3)))=\\
=Ut\ci U(A\ot s)\ci \nu_{A,C\ot D}(u_1\ot\nu_{C,D}(u_2\ot u_3))
\end{gather*}
from which one deduces that $\mu$ is associative iff $\nu$ is associative.
Similarly, one can easily see that $\mu$ is unital iff $\nu$ is unital.
\end{proof}


Recall that a monoidal category $\C$ is called \textit{left closed} if for all object $A$ there is a right adjoint $[A,\under]$ of the endofunctor $\under\ot A:\C\to\C$ and it is called \textit{right closed} if if for all object $A$
there is a right adjoint $\{A,\under\}$ of $A\ot\under$.
Applying the general results of \cite{Day70} to our $\Ab$-enriched situation we obtain 
\begin{lem}
If $\C$ is left (right) closed then so is $\Pre$ with left and right internal homs given by
\begin{align*}
[U,V]A&=\int_B \Ab(UB,V(A\ot B))=\Pre(U,V(A\ot\under))\\
\{U,V\}B&=\int_A\Ab(UA,V(A\ot B)) =\Pre(U,V(\under\ot B))
\end{align*}
respectively.
\end{lem}


\subsection{The monoidal extension $\longF$}

Let $\longf:\C\to\Ab$ be an additive functor. For any monoidal structure $\longf_2:\ot(\longf\x \longf)\to \longf\ot$, $\longf_0:\ZZ\to \longf I$ on $\longf$
there is a monoidal structure on $\longF=\,\under\amo{\C}\longf$ defined as follows.
\begin{align}
\label{FF_2}
\longF_{U,V}:\longF U\ot\longF V &\to \longF(U\odot V)\\
(u\am{C}x)\ot(v\am{D}y)&\mapsto [u,v,1]^{C\ot D}_{C,D}\am{C\ot D} \longf_{C,D}(x\ot y)\notag\\
\label{FF_0}
\longF_0:\ZZ&\to \longF \YI\\
1&\mapsto I\am{I}\longf_0(1)\notag
\end{align}
%
\begin{lem}
The monoidal functor $\longF$ is an extension of $\longf$ in the sense of the natural isomorphism (\ref{N}) being a monoidal natural isomorphism $\longN_A:\longF YA\iso \longf A$, i.e., the equations
\begin{align}
\label{N is multiplicative}
\longN_{A\ot B}\ci\longF Y_{A,B}\ci\longF_{YA,YB}&=\longf_{A,B}\ci(\longN_A\ot \longN_B)\\
\label{N is unital}
\longN_0\ci \longF Y_0\ci \longF_0&= \longf_0
\end{align}
hold for all $A,B\in\ob\C$.
\end{lem}
\begin{proof}
Evaluated on rank 1 tensors $\longF_2$ on representable presheaves can be written as
\begin{gather*}
\longF_{YA,YB}((u\am{C}x)\ot(v\am{D}y))=[u,v,1]^{C\ot D}_{C,D}\am{C\ot D} \longf_{C,D}(x\ot y)=\\
=[1_A,1_B,u\ot v]^{C\ot D}_{A,B}\am{C\ot D} \longf_{C,D}(x\ot y)=\\
=(YA\odot YB)(u\ot v)([1_A,1_B,1]^{A\ot B}_{A,B})\am{C\ot D} \longf_{C,D}(x\ot y)=\\
=[1_A,1_B,1]^{A\ot B}_{A,B})\am{A\ot B} \longf_{A,B}(\longf(u)x\ot \longf(v)y)=\\
=(Y^{-1}_{A,B})_{A\ot B}(1_{A\ot B})\am{A\ot B} \longf_{A,B}(\longf(u)x\ot \longf(v)y)
\end{gather*}
therefore
\[
\longF Y_{A,B}\ci\longF_{YA,YB}((u\am{C}x)\ot(v\am{D}y))=1_{A\ot B}\am{A\ot B}\longf_{A,B}(\longf(u)x\ot \longf(v)y)
\]
from which (\ref{N is multiplicative}) follows. Equation (\ref{N is unital}) is obvious from the definitions (\ref{FF_0}),
(\ref{Y_0}) and (\ref{N}).
\end{proof}

\begin{pro} \label{pro: esss}
Let $\longF:\hat \C\to\Ab$ be the monoidal functor extending the monoidal functor $\longf:\C\to\Ab$ as defined above. Then
\begin{enumerate}
\item $\longN_I:\longF YI\iso \longf I$ is the underlying map of a ring isomorphism from $\R:=$ \newline
$\bra \longF \YI,\longF_{\YI,\YI},\longF_0\ket$ to
$R=\bra \longf I,\longf_{I,I},\longf_0\ket$ and
\item $\longF$ is essentially strong monoidal  iff $\longf$ is essentially strong monoidal.
\end{enumerate}
\end{pro}
\begin{proof}
$(1)$ This is clear from (\ref{N is multiplicative}), (\ref{N is unital}).

$(2)$ Equation (\ref{N is multiplicative}) extends to the serially commuting diagram
\[
\parbox{300pt}{
\begin{picture}(300,110)

\put(10,10){$\longf A\ot \longf I\ot \longf B$}
\put(0,90){$\longF YA\ot \longF \YI\ot \longF YB$}
\put (45,80){\vector(0,-1){55}} \put(50,50){$\sst \longN_A\ot \longN_I\ot \longN_B$}

\put(100,11){\vector(1,0){30}}
\put(100,14){\vector(1,0){30}}
\put(100,91){\vector(1,0){30}}
\put(100,94){\vector(1,0){30}}

\put(150,10){$\longf A\ot \longf B$}
\put(140,90){$\longF YA\ot \longF YB$}
\put(170,80){\vector(0,-1){55}} \put(175,50){$\sst \longN_A\ot \longN_B$}

\put(215,12){\vector(1,0){40}} \put(225,20){$\sst \longf_{A,B}$}
\put(215,92){\vector(1,0){40}} \put(220,100){$\sst \longF_{YA,YB}$}

\put(270,10){$\longf(A\ot B)$}
\put(290,43){\vector(0,-1){20}} \put(300,32){$\sst \longN_{A\ot B}$}
\put(270,50){$\longF Y(A\ot B)$}
\put(290,83){\vector(0,-1){20}} \put(300,75){$\sst \longF Y_{A,B}$}
\put(265,90){$\longF(YA\odot YB)$}

\end{picture}
}
\]
with all vertical arrows being isomorphisms. From this we see that $\longf$ is essentially strong iff
the first row is a coequalizer for all $A$ and $B$. Since every presheaf is a colimit of representables and $\odot$
preserves these colimits by Proposition \ref{pro: odot} the next diagram with vertical arrows being the colimiting
cones
\[
\parbox{311pt}{
\begin{picture}(311,95)

\put(10,10){$\longF U\ot \longF \YI\ot \longF V$}
\put(0,75){$\longF YA_i\ot \longF \YI\ot \longF YB_j$}
\put (45,65){\vector(0,-1){40}} 

\put(100,11){\vector(1,0){30}}
\put(100,14){\vector(1,0){30}}
\put(100,76){\vector(1,0){30}}
\put(100,79){\vector(1,0){30}}

\put(150,10){$\longF U\ot \longF V$}
\put(140,75){$\longF YA_i\ot \longF YB_j$}
\put(175,65){\vector(0,-1){40}} 

\put(215,12){\vector(1,0){40}} \put(225,20){$\sst \longF_{U,V}$}
\put(215,77){\vector(1,0){40}} \put(220,85){$\sst \longF_{YA_i,YB_j}$}

\put(270,10){$\longF(U\odot V)$}
\put(293,65){\vector(0,-1){40}} 
\put(263,75){$\longF(YA_i\odot YB_j)$}

\end{picture}
}
\]
implies that $\longF$ is essentially strong whenever the $\longF_{YA,YB}$ are all coequalizers.
Vice versa, if $\longF$ is essentially strong then so is its composition $\longF Y$ with the strong monoidal
$Y$ and this composite is isomorphic to $\longf$ by $\longN$.
\end{proof}

\section{The adjunction associated to a fiber functor} \label{s: FG}

\subsection{The strong part of $\longF$}

Assuming $\longf$ is an essentially strong monoidal functor we have the essentially strong monoidal extension $\longF$
to the presheaf category. The canonical decompositions of $\longf$ and $\longF$ in the sense of \cite{Sz: Brussels}
yield strong monoidal functors $\bimf:\C\to\,_R\M_R$ and $\bimF:\Pre\to\,_\R\M_\R$.

Not willing to use both $\R$ and $R$, however, we shall redefine the strong part $\bimF$ of $\longF$ by composing the canonical strong part $\Pre\to\,_\R\M_\R$ with the isomorphism of categories $_\R\M_\R\to\,_R\M_R$ induced by the isomorphism $\R\iso R$ of Proposition \ref{pro: esss} (1) . Then we can write the 2-cell $\longN$ as
\begin{equation}\label{diag: N}
\longN\quad =\qquad
\parbox{200pt}{
\begin{picture}(200,90)
\put(0,10){$\C$}
\put(0,77){$\hat\C$}
\put(2,25){\vector(0,1){45}} \put(5,42){$Y$}
\put(12,16){\vector(3,2){35}} \put(35,22){$\bimf$}
\put(12,77){\vector(3,-2){35}} \put(35,64){$\bimF$}
\put(40,44){$_R\M_R$}
\put(23,42){$\bimN$}
\put(70,45){\vector(1,0){65}} \put(93,48){$\philong$}
\put(150,43){$\Ab$}
\put(20,83){\vector(4,-1){120}} \put(70,75){$\longF$}
\put(20,11){\vector(4,1){120}} \put(70,15){$\longf$}
\end{picture}
}
\end{equation}
the composite of an invertible $\bimN:\bimF Y\iso\bimf$ with two identity 2-cells.
In this way both $\bimf$ and $\bimF$ are strong monoidal functors to the same bimodule category.
Inserting $\longN=\philong \bimN$ into (\ref{N is multiplicative}), (\ref{N is unital}) we see that $\bimN$ is a monoidal natural isomorphism.


\subsection{The monoidal adjunction $\bimF\dashv\bimG$} \label{ss: bimF - bimG}

\begin{lem} \label{lem: left adjoint}
Let $\longF:\hat\C\to\Ab$ be an essentially strong monoidal functor and $\Pre\rarr{\bimF}\,_R\M_R\to\Ab$ be a monoidal
factorization of $\longF$ with $\bimF$ strong monoidal. Then the following are equivalent:
\begin{enumerate}
\item There is an essentially strong monoidal $\longf:\C\to\Ab$ and a monoidal natural isomorphism
$\longF\cong\,\under\amo{\C}\longf$.
\item The underlying functor of $\longF$ is left adjoint.
\item The underlying functor of $\bimF$ is left adjoint.
\item $\bimF$ is left adjoint in the 2-category $\MonCat$.
\end{enumerate}
\end{lem}
\begin{proof}
$(4)\Rightarrow(3)$ Forgetting the monoidal structure this is obvious.
$(3)\Rightarrow(2)$ Since the forgetful functor $_R\M_R\to\Ab$ has a right adjoint,
the coinduction functor $\Ab(R^\op\ot R,\under)$, $\longF$ is the composite of left adjoint functors.

$(2)\Rightarrow(1)$ Choose an adjunction $\longF\adj\longG$
and define $\longf:=\longF Y$ where $Y:\C\to\hat\C$ is the Yoneda embedding. Then
\begin{align*}
\Ab(U\amo{\C}\longf,X)&\cong\hat\C(U,\Ab(\longf\under,X))=\hat\C(U,\Ab(\longF Y\under,X))\cong\\
&\cong\hat\C(U,\hat\C(Y,\longG X))\cong\hat\C(U,\longG X)\cong\\
&\cong\Ab(\longF U,X)
\end{align*}
implying that $\longF\cong\,\under\amo{\C}\longf$, as additive functors. Giving monoidal structure on $\longf$ by requiring $\longf=\longF Y$ to be a composite of monoidal functors we are done.

$(1)\Rightarrow(4)$ Let $\bimf$ be the strong part of $\longf$ and define 
\[
\bimG:\,_R\M_R\to\hat\C,\quad \bimG M:=\,_R\M_R(\bimf\under,M)\,.
\]
Then using the hom-tensor relation $_R\M_R(X\ot N,M)\cong\Ab(X,\,_R\M_R(N,M))$
elements $\mu=\{\mu_C:UC\ot \bimf C\to M\}_C$ of the hom-group
$_R\M_R(U\amo{\C}\bimF,M)$ are in bijection with families $\nu$ of group homomorphisms
$\nu_C:UC\to\,_R\M_R(\bimf C,M)$ satisfying $\nu_D\ci Ut=\,_R\M_R(\bimf t,M)\ci\nu_C$ for
$t\in\C(D,C)$, i.e., with elements $\nu$ of the hom-group $\hat\C(U,\,_R\M_R(\bimf\under,M))$.
This proves the adjunction
\[
\under\amo{\C}\bimf\quad\dashv\quad_R\M_R(\bimf,\under)
\]
i.e., $\bimF\dashv\bimG$ as ordinary functors. Since $\bimF$ is strong, we may consider it
as a (strong) opmonoidal functor. Then its right adjoint $\bimG$ has a canonical monoidal structure such that the unit $\bimeta:\Pre\to\bimG\bimF$ and counit $\bimeps:\bimF\bimG\to\,_R\M_R$ of the adjunction are monoidal natural transformations \cite{Kelly}. According to this, the monoidal structure of $\bimG$ is
\begin{align*}
\bimG_{M,N}&=\bimG(\bimeps_M\oR\bimeps_N)\ci\bimG\bimF^{-1}_{\bimG M,\bimG N}\ci 
\bimeta_{\bimG M\odot\bimG N}\\
\bimG_0&=\bimG\bimF^{-1}_0\ci\bimeta_\YI\,.
\end{align*}
Computing them explicitly one obtains
\begin{align*}
\bimG M\odot\bimG N&={\sst\int^{C'}\int^{C''}\,_R\M_R(\bimf C',M)\ot\,_R\M_R(\bimf C'',N)\ot
\C(\under,C'\ot C'')}\rarr{}\\
&\rarr{\bimG_{M,N}}\bimG(M\oR N)\\
[g',g'',t]^A_{C',C''}&\mapsto (g'\oR g'')\ci \bimf^{-1}_{C',C''}\ci\bimf t\\
\bimG_0:\YI&\to\bimG R\\
(C\rarr{s}I)&\mapsto (\bimf C\rarr{\bimf(s)}R)
\end{align*}
where $R$ denotes also the bimodule $_RR_R$, the monoidal unit of $_R\M_R$.
\end{proof}

Explicit formulas for the unit and counit of the adjunction $\bimF\dashv\bimG$ are
\begin{alignat}{2}
\bimeta&:\Pre\to\bimG\bimF,&\qquad \bimeta_UC&:UC\to\,_R\M_R(\bimf C,\bimF U)\label{bimeta}\\
&&&\qquad u\mapsto\{x\mapsto u\am{C}x\},\notag\\
\bimeps&:\bimF\bimG\to\,_R\M_R,&\qquad \bimeps_M&:\,_R\M_R(\bimf\under,M)\amo{\C}\bimf\to M\label{bimeps}\\
&&&\qquad h\am{C}x\mapsto h(x).\notag
\end{alignat}

\begin{cor} \label{cor: bimQ}
For a small monoidal $\Ab$-category $\C$ and for an essentially strong monoidal functor $\longf:\C\to\Ab$ the
functor $\bimF\bimG$ is underlying a monoidal comonad $\bimQ$ on the category $_R\M_R$ of bimodules over $R=\longf I$.
If furthermore $\longf$ is flat then the comonad is left exact.
\end{cor}

Explicitly, $\bimQ=\bra\bimF\bimG,\bimcop,\bimeps\ket$, where $\bimcop=\bimF\bimeta\bimG$. The monoidal structure of $\bimQ$ is given by the composition of the monoidal functors $\bimF$ and $\bimG$, i.e.,
\begin{align*}
\bimQ_{M,M'}&=\bimF\bimG_{M,M'}\ci\bimF_{\bimG M,\bimG M'}\ :\ \bimQ M\oR\bimQ M'\to\bimQ(M\oR M'),\\
\bimQ_0&=\bimF\bimG_0\ci\bimF_0\ :\ R\to\bimQ R\,.
\end{align*}

Let $\M^\bimQ$ denote the Eilenberg-Moore category of $\bimQ$-comodules. Since $\bimQ$ is left exact comonad on an abelian category, its Eilenberg-Moore category $\M^\bimQ$ is cocomplete and abelian.
It inherits a monoidal structure $\bra\M^\bimQ,\oR,\bra R,\bimQ_0\ket\ket$ where
the monoidal product of $\bimQ$-comodules is
\begin{equation*}
\bra M,\alpha\ket\oR\bra M',\alpha'\ket:=\bra M\oR M',\bimQ_{M,M'}\ci(\alpha\oR\alpha')\ket
\end{equation*}
and the monoidal unit is $R$ as an $R$-$R$-bimodule equipped with coaction $\bimQ_0:R\to\bimQ R$.
The forgetful functor $\M^\bimQ\to\,_R\M_R$ is then automatically strong monoidal which is a rationale for denoting the monoidal product in both categories by the same symbol.






\subsection{The adjunction $\Fsh\dashv\Gsh$} \label{ss: Fsh - Gsh}

In this subsection we 
study another adjunction associated to our fiber functor which
yields a comonad on $\M_R$ and therefore cannot be monoidal. Still it is, in a sense, equivalent to the monoidal
comonad $\bimQ$ on $_R\M_R$. The situation is similar to corings and bialgebroids. Comodules of corings
are defined on one sided modules and so are the comodules of bialgebroids although the latter ones
have a monoidal product. Thus the present Subsection can be considered as a comonadic version of
\cite[1.4.]{Phung}.

In the process of forgetting $_R\M_R\to\Ab$ there is an intermediate step when we forget only the left $R$-module structures: $\phish:\,_R\M_R\to\M_R$.  This defines half lifts of the long forgetful functors denoted by $\fsh$ and $\Fsh$, respectively.
\begin{align*}
F&=\underset{\fsh}{\underbrace{\C\ \rarr{\bimf}\ _R\M_R\ \rarr{\phish}\ \M_R}}\ \rarr{}\ \Ab\\
\F&=\underset{\Fsh}{\underbrace{\Pre\ \rarr{\bimF}\ _R\M_R\ \rarr{\phish}\ \M_R}}\ \rarr{}\ \Ab
\end{align*}
Combining faithfulness of the forgetful functor $\phish$ with the adjunction $\bimF\dashv\bimG$,
\begin{equation}
\begin{CD}
_R\M_R(\bimF U,M)@>\phish_{\bimF U,M}>>\M_R(\Fsh U,\phish M)\\
@V\cong VV @VV\cong V\\
\Pre(U,\bimG M)@>\Pre(U,\iota_M)>>\Pre(U,\Gsh\phish M)
\end{CD}
\end{equation}
we obtain, on the one hand, a right adjoint $\Gsh=\{N\mapsto\M_R(\fsh\under,N)\}$ of $\Fsh$ and on the other hand,
by the Yoneda Lemma, a monic arrow $\iota:\bimG\into\Gsh\phish$. Explicitly, $(\iota_M)_C$ maps $f\in\,_R\M_R(\bimf C,M)$ to $f$ considered merely as an element of $\M_R(\fsh C,\phish M)$. The unit and counit of
$\Fsh\dashv\Gsh$ are given by
\begin{alignat}{2}
\etash&:\Pre\to\Gsh\Fsh,&\qquad \etash_UC&:UC\to\M_R(\fsh C,\Fsh U)\label{etash}\\
&&&\qquad u\mapsto\{x\mapsto u\am{C}x\},\notag\\
\epsh&:\Fsh\Gsh\to\M_R,&\qquad \epsh_N&:\M_R(\fsh\under,N)\amo{\C}\fsh\to N\label{epsh}\\
&&&\qquad f\am{C}x\mapsto f(x).\notag
\end{alignat}
The comonad $\bra \Fsh\Gsh,\copsh,\epsh\ket$ where $\copsh:=\Fsh\etash\Gsh$ will be denoted by $\Qsh$ and its Eilenberg-Moore category $(\M_R)^\Qsh$ of comodules simply by $\M^\Qsh$. This category has no apparent monoidal structure.

Later it will be useful to describe $\iota_M$ as an equalizer, hence a kernel, in $\Pre$. For this purpose we introduce some notations.
Let us consider $_R\M_R$ as the Eilenberg-Moore category for the monad $R\ot\under$ on $\M_R$. Then a bimodule $M$ can be identified with the pair $\bra \phish M,\lambda_M\ket$ where $\lambda_M$ denotes  the ring
homomorphism $R\to\End\phish M$ (which is more familiar than, but equivalent to, an action map
$R\ot\phish M\to\phish M$).
Since $\bimf C$ are bimodules for all $C\in\ob\C$, we can define $\ell(r)_C:=\lambda_{\bimf C}(r)\in\End \fsh C$
for any $r\in R$ and obtain the self natural transformation $\ell(r)\in\End\fsh$ of left action by $r\in R$ on the functor $\fsh$. This induces two more natural transformations, $\overrightarrow{\ell}(r):=\under\amo{\C}\ell(r)\in\End\Fsh$ and $\overleftarrow{\ell}(r):=\M_R(\ell(r),\under)\in\End\Gsh$.
Since an $R$-module map $f\in\M_R(\fsh C,\phish M)$ is an $R$-$R$-bimodule map precisely when
$\lambda_M(r)\ci f=f\ci\ell(r)_C$, $\forall r\in R$, and since limits in $\Pre$ are taken pointwise,
\begin{equation} \label{eq: iota}
\bimG M\overset{\iota_M}{\longrightarrowtail} \Gsh\phish M\pair{\Gsh\Lambda_M}{\eLl_{\phish M}}
\Gsh\prod_{r\in R}\phish M
\end{equation}
is an equalizer in $\Pre$ where $\Lambda_M$, natural in $M\in\,_R\M_R$, and $\eLl_N$, natural in $N\in\M_R$, are uniquely defined by
\begin{align}
\label{def Lambda}
p_r\ci\Lambda_M&=\lambda_M(r),\qquad\forall r\in R\\
\label{def eLl}
\Gsh p_r\ci\eLl_N&=\elll(r)_N,\qquad\forall r\in R
\end{align}
where $p_r$ denote the projections of the product $\prod_r N$ in which case the $\Gsh p_r$ are also projections of a product since $\Gsh$ is right adjoint.

Since $\Fsh=\phish\bimF$ where $\phish$ is a right adjoint, left exactness of $\bimF$ implies left exactness of $\Fsh$. Therefore
\begin{equation} \label{eq: j}
\phish\bimF\bimG M\overset{\Fsh\iota_M}{\longrightarrowtail} \Fsh\Gsh\phish M\pair{\Fsh\Gsh\Lambda_M}{\Fsh\,\eLl_{\phish M}}
\Fsh\Gsh\prod_{r\in R}\phish M
\end{equation}
is an equalizer in $\M_R$.
\begin{lem}\label{lem: j}
The pair $\bra\phish,j\ket$ consisting of  the forgetful functor $\phish:\,_R\M_R\to\M_R$ and the
natural transformation $j:=\Fsh\iota:\phish\bimQ\to\Qsh\phish$ of (\ref{eq: j}) is a morphism of comonads $\bimQ\to\Qsh$, i.e.,
\begin{align}
\label{j1}
\copsh\phish\ci j&=\Qsh j\ci j\bimQ\ci\phish\bimcop\,,\\
\label{j2}
\epsh\phish\ci j&=\phish\bimeps\,.
\end{align}
\end{lem}
\begin{proof}
Comparing (\ref{bimeps}) with (\ref{epsh}) equation (\ref{j2}) immediately follows taking into account that $j$ maps
the generic element $h\am{C}x$ to $f\am{C}x$ where $f$ is $h$ considered as a map in $\M_R$ thus $f(x)=h(x)$.
In order to prove (\ref{j1}) we need an analogue of (\ref{j2}) which compares the two units. We claim that
\begin{equation}\label{j3}
\iota\bimF\ci\bimeta=\etash\,.
\end{equation}
Indeed, equations (\ref{bimeta}), (\ref{etash}) are adjusted together by $\iota_{\bimF U}C:\bimG\bimF UC\to\Gsh\Fsh UC$ which is the map $_R\M_R(\bimf C,\bimF U)\to\M_R(\fsh C,\Fsh U)$ sending $h$ to its underlying right $R$-module map. Now the proof of (\ref{j1}) is given by the calculation
\begin{align*}
\copsh\phish\ci j&=\Fsh(\etash\Gsh\phish\ci\iota)=\\
&=\Fsh(\Gsh\Fsh\iota\ci\etash\bimG)=\\
&=\Fsh\Gsh\Fsh\iota\ci\Fsh(\iota\bimF\ci\bimeta)\bimG=\\
&=\Qsh j\ci j\bimQ\ci \phish\bimcop\,.
\end{align*}
\end{proof}
\begin{cor} \label{cor: psh}
The comonad morphism $\bra\phish,j\ket$ induces a functor $\psh:\M^\bimQ\to\M^\Qsh$ which sends the $\bimQ$-comodule $\bra M, \alpha\ket$ to the $\Qsh$-comodule $\bra\phish M, j_M\ci\phish\alpha\ket$ and the arrow $\bra M,\alpha\ket\rarr{t}\bra M',\alpha'\ket$ to the arrow $\phish t$. Therefore the diagram
\begin{equation*}
\begin{CD}
\M^\bimQ@>\psh>>\M^\Qsh\\
@V{\F^\bimQ}VV @VV{\F^\Qsh}V\\
_R\M_R@>\phish>>\M_R
\end{CD}
\end{equation*}
with the vertical arrows denoting the obvious forgetful functors is an identity 2-cell in $\CAT$.
\end{cor}
\begin{proof}
Although this is well-known, see \cite{Street: Monads I.}, we give the explicit calculations:
\begin{align*}
\copsh\phish M\ci jM\ci\phish\alpha&=\Fsh\etash\Gsh\phish M\ci\Fsh\iota M\ci\phish\alpha=\\
&=\Fsh\Gsh\Fsh\iota M\ci\Fsh\etash\bimG M\ci\phish\alpha=\\
&=\Fsh\Gsh\Fsh\iota M\ci\Fsh\iota\bimF\bimG M\ci\phish\bimcop M\ci\phish\alpha=\\
&=\Fsh\Gsh\Fsh\iota M\ci\Fsh\iota\bimF\bimG M\ci\phish\bimF\bimG\alpha\ci\phish\alpha=\\
&=\Fsh\Gsh\Fsh\iota M\ci\Fsh\Gsh\phish\alpha\ci\Fsh\iota M\ci\phish\alpha=\\
&=\Qsh(jM\ci\phish\alpha)\ci(j M\ci\phish\alpha)
\end{align*}
and
\[
\epsh\phish M\ci \Fsh\iota M\ci\phish\alpha=\phish\bimeps M\ci\phish\alpha=\phish M\,.
\]
For a morphism $\bra M,\alpha\ket\rarr{h}\bra N,\beta\ket$ of $\bimQ$-comodules
$\psi^\pisharp\bra M,\alpha\ket\rarr{\phish h}\psi^\pisharp\bra N,\beta\ket$ is a morphism of $\Qsh$-comodules. Indeed,
\begin{align*}
\Fsh\Gsh\phish h\ci\Fsh\iota M\ci\phish\alpha&=\Fsh\iota N\ci\phish\bimF\bimG h\ci\phish\alpha=\\
&=\Fsh\iota N\ci\phish\beta\ci\phish h\,.
\end{align*}
\end{proof}
The surprising fact is that the functor $\psh$ is an isomorphism of categories as we shall see soon.
At first we construct an $R$-bimodule structure on $\Qsh$-comodules due to the fact that $\Qsh$
carries a left action of $R$ since the $\Fsh$ does. For a $\Qsh$-comodule $\bra N,\beta\ket$ let
$\hat N=\bra N,\lambda_{\hat N}\ket$ be the $R\ot\under$-module defined by
\begin{equation} \label{lambda Nhat}
r\in R\mapsto \lambda_{\hat N}(r):=\ N\rarr{\beta}\Qsh N\longrarr{\lambda_{\bimF\Gsh N}(r)}\Qsh N\rarr{\epsh N}N
\end{equation}
where $\lambda_{\bimF\Gsh N}(r)$ is, of course, the same thing as $\overrightarrow{\ell(r)}_{\Gsh N}$.
\begin{lem}
For all $r\in R$ we have the identities
\begin{align}
\label{Take for epsh}
\epsh\ci\overrightarrow{\ell}(r)\Gsh&=\epsh\ci\Fsh\overleftarrow{\ell}(r)\\
\label{Take for etash}
\Gsh\overrightarrow{\ell}(r)\ci\etash&=\overleftarrow{\ell}(r)\Fsh\ci\etash\\
\label{Take for copsh}
\Fsh\Gsh\overrightarrow{\ell}(r)\Gsh\ci\copsh&=\Fsh\overleftarrow{\ell}(r)\Fsh\Gsh\ci\copsh
\end{align}
\end{lem}
\begin{proof}
To prove (\ref{Take for epsh}) evaluate its $N$-component on $h\am{C}x\in\Gsh N\amo{\C}\fsh$:
\begin{align*}
\epsh_N(\ellr(r)\Gsh N(h\am{C}x))&=\epsh_N(h\am{C}\ell(r)_C(x))=h(\ell(r)_C(x))=\\
&=\epsh_N(h\ci\ell(r)_C\am{C}x)=\epsh_N(\Fsh\elll(r)_N(h\am{C}x))\,.
\end{align*}
To prove (\ref{Take for etash}) evaluate the $C$-component of its $U$-component on $u\in UC$:
\[
(\Gsh\ellr(r)_U)_C\ci(\etash_U)_C(u)=\{x\mapsto u\am{C}\ell(r)_C(x)\}=
(\elll(r)\Fsh U)_C\ci(\etash_U)_C(u)\,.
\]
Applying $\Fsh$ from the left and $\Gsh$ from the right (\ref{Take for etash}) implies (\ref{Take for copsh}).
\end{proof}
\begin{lem} \label{lem: bimodule lift of beta}
For each $\Qsh$-comodule $\bra N,\beta\ket$ the $\lambda_{\hat N}:R\to \End N$ defined in (\ref{lambda Nhat})
is a ring homomorphism such that $\beta=\phish\hat\beta$ for a unique $R$-bimodule map $\hat\beta:\hat N\to\bimF\Gsh N$. Moreover \, the identities
\begin{align}
\label{Take for beta}
\Fsh\Gsh\lambda_{\hat N}(r)\ci\beta&=\Fsh\elll(r)N\ci\beta\,\quad r\in R\\
\label{bimodule lifted coaction}
\bimF\Gsh\beta\ci\hat\beta&=\bimF\etash\Gsh N\ci\hat\beta
\end{align}
hold true.
\end{lem}
\begin{proof}
Let us see at first if $\beta$ can be lifted to a bimodule map:
\begin{align*}
\beta\ci\lambda_{\hat N}(r)&=\beta\ci \epsh N\ci \ellr(r)\Gsh N\ci\beta\eqby{Take for epsh}
=\beta\ci \epsh N\ci \Fsh\elll(r) N\ci\beta=\\
&=\epsh\Fsh\Gsh N\ci\Fsh\elll(r)\Fsh\Gsh N\ci\Fsh\Gsh\beta\ci\beta=\\
&=\epsh\Fsh\Gsh N\ci\Fsh\elll(r)\Fsh\Gsh N\ci\copsh N\ci\beta\eqby{Take for copsh}\\
&=\epsh\Fsh\Gsh N\ci\Fsh\Gsh\ellr(r)\Gsh N\ci\copsh N\ci\beta=\\
&=\ellr(r)\Gsh N\ci\epsh\Fsh\Gsh N\ci\copsh N\ci\beta=\ellr(r)\Gsh N\ci\beta=\\
&=\lambda_{\bimF\Gsh N}(r)\ci\beta\,.
\end{align*}
This implies that for all $r,r'\in R$
\begin{align*}
\lambda_{\hat N}(r)\lambda_{\hat N}(r')&=\epsh N\ci\ellr(r)\Gsh N\ci\beta \ci \lambda_{\hat N}(r')=\\
&=\epsh N\ci\ellr(r)\Gsh N\ci \ellr(r')\Gsh N\ci\beta=\epsh N\ci\ellr(rr')\Gsh N\ci\beta=\\
&=\lambda_{\hat N}(rr')
\end{align*}
and $\lambda_{\hat N}(1_R)=\epsh N\ci\beta=N$, obviously.
Next we show (\ref{Take for beta}).
\begin{align*}
\Fsh\Gsh\lambda_{\hat N}(r)\ci \beta&=\Fsh\Gsh\epsh N\ci\Fsh\Gsh\ellr(r)\Gsh N\ci\copsh N\ci\beta
\eqby{Take for copsh}\\
&=\Fsh\Gsh\epsh N\ci\Fsh\elll(r)\Fsh\Gsh N\ci\copsh N\ci\beta=\\
&=\Fsh\elll(r)N\ci\Fsh\Gsh\epsh N\ci\copsh\ci\beta=\Fsh\elll(r)N\ci\beta.
\end{align*}
Finally, applying the faithful $\phish$ to (\ref{bimodule lifted coaction}) we get the associativity property of the
$\Qsh$-coaction $\beta$. Hence (\ref{bimodule lifted coaction}) is an identity, too.
\end{proof}
\begin{lem} \label{lem: Xi}
The correspondence $\bra N,\beta\ket\mapsto \hat N$ defined by (\ref{lambda Nhat}) is the object map of a unique functor $\Xi:\M^\Qsh\to\,_R\M_R$ such that $\phish\Xi=\F^\Qsh$.
\end{lem}
\begin{proof}
Since $\F^\Qsh$ maps an arrow $\bra N,\beta\ket\rarr{f}\bra N',\beta'\ket$ to its underlying $R$-module map
$f:N\to N'$, it suffices to show that all such $f$-s are actually left $R$-module maps, too. Indeed,
\begin{align*}
\lambda_{\hat N'}(r)\ci f&=\epsh_{N'}\ci \ellr(r)\Gsh N'\ci\beta'\ci f
=\epsh_{N'}\ci \ellr(r)\Gsh N'\ci\Fsh\Gsh f\ci\beta=\\
&=\epsh_{N'}\ci\Fsh\Gsh f\ci \ellr(r)\Gsh N\ci\beta=f\ci \epsh_{N}\ci \ellr(r)\Gsh N\ci\beta=\\
&=f\ci\lambda_{\hat N}(r)
\end{align*}
for all $r\in R$.
\end{proof}
Next we want to show that the bimodule $\hat N$ constructed by $\Xi\bra N,\beta\ket$ is underlying a
$\bimQ$-comodule $\bra \hat N,\alpha\ket$ in such a way that $\bra N,\beta\ket\mapsto\bra\hat N,\alpha\ket$ provides an inverse functor of $\psh$.
For this purpose we define $\alpha$ by the universal property of the equalizer (\ref{eq: j}), i.e.,
by the diagram
\begin{equation}\label{diag: alpha beta}
\parbox{100pt}{
\begin{picture}(100,80)
\put(0,10){$\phish\bimF\bimG\hat N$}
\put(38,10){$\longerrightarrowtail$}
\put(63,18){$j\hat N$}
\put(100,10){$\Qsh N$}
\put(7,70){$N$} \put(-12,40){$\phish\alpha$}
\put(20,70){\vector(2,-1){85}}
\put(57,55){$\beta$}
\dashline{3}(12,65)(12,25) \put(12,25){\vector(0,-1){0}}
\end{picture}
}
\end{equation}
To this end we need to show that $\beta$ equalizes the pair
$\pair{\Fsh\Gsh\Lambda_{\hat N}}{\Fsh\,\eLl_N}$ which seems to follow easily from (\ref{Take for beta})
since the latter can be obtained from the pair by applying the `projections' $\Fsh\Gsh p_r$. However, this is a wrong
argument because in passing from (\ref{eq: iota}) to (\ref{eq: j}) the $\Fsh$, although preserved the equalizer,  destroyed the product structure of $\Gsh\prod_r N$. Still the expectation holds true by the next
\begin{lem} \label{lem: psiinverse object map}
For any $\Qsh$-comodule $\bra N,\beta\ket$ let $\hat N$ be the $R$-bimodule constructed in (\ref{lambda Nhat}).
Then $\beta$ equalizes the pair $\pair{\Fsh\Gsh\Lambda_{\hat N}}{\Fsh\,\eLl_N}$ and the
unique arrow denoted by $\phish\alpha$ in (\ref{diag: alpha beta}) can be lifted to $_R\M_R$
as an arrow $\alpha$ that makes the pair $\bra\hat N,\alpha\ket$ a $\bimQ$-comodule.
\end{lem}
\begin{proof}
Let $p_r$ and $p'_r$ for $r\in R$ be the projections of the product $\prod_r N$ and $\prod_r \Qsh N$, respectively.
Define 
\begin{alignat}{2}
\eLr_N&:\Qsh N\to\prod_{r\in R}\Qsh N\quad&\text{by}\quad\ p'_r\ci\eLr_N&=\ellr_{\Gsh N}(r),\ \forall r\\
E_N&:\prod_{r\in R}\Qsh N\to\prod_{r\in R} N\quad&\text{by}\quad\ p_r\ci E_N&=\epsh_N\ci p'_r\,,\ \forall r\,.
\end{alignat}
Then
\begin{align*}
p_r\ci\Lambda_{\hat N}&\eqby{def Lambda}\epsh_N\ci\ellr_{\Gsh N}(r)\ci\beta=\\
&=\epsh_N\ci p'_r\ci\eLr_N\ci\beta=p_r\ci E_N\ci\eLr_N\ci\beta
\end{align*}
implies
\begin{equation} \label{expr for Lambda}
\Lambda_{\hat N}=E_N\ci\eLr_N\ci\beta\,.
\end{equation}
On the other hand,
\begin{align*}
\Gsh p_r\ci\eLl_N&\eqby{def eLl}\elll(r)_N=\elll(r)_N\ci\Gsh\epsh_N\ci\Gsh\beta=\Gsh\epsh_N\ci\elll(r)_{\Qsh N}\ci
\Gsh\beta=\\
&=\Gsh\epsh_N\ci\Gsh p'_r\ci\eLl_{\Qsh N}\ci\Gsh\beta=\\
&=\Gsh p_r\ci\Gsh E_N\ci\eLl_{\Qsh N}\ci\Gsh \beta
\end{align*}
implies
\begin{equation} \label{expr for eLl}
\eLl_N=\Gsh E_N\ci\eLl_{\Qsh N}\ci\Gsh\beta\,.
\end{equation}
Taking the $\Qsh$ of (\ref{expr for Lambda}) and the $\Fsh$ of (\ref{expr for eLl}) we see that $\beta$ should equalize
the two composites in the diagram (not a commutative one)
\begin{equation}
\parbox{220pt}{
\begin{picture}(220,60)
\put(-10,25){$\Qsh N$}
\put(10,22){\vector(2,-1){30}} \put(13,5){$\Qsh\beta$}
\put(10,34){\vector(2,1){30}} \put(15,45){$\Qsh\beta$}
\put(40,0){$\Qsh^2 N$}
\put(40,50){$\Qsh^2 N$}
\put(65,2){\vector(1,0){50}} \put(70,8){$\Fsh\eLl_{\Qsh N}$}
\put(65,52){\vector(1,0){50}} \put(75,55){$\Qsh\eLr_N$}
\put(120,0){$\Qsh\prod_r\Qsh N$}
\put(120,50){$\Qsh\prod_r\Qsh N$}
\put(165,6){\vector(2,1){30}} \put(180,5){$\Qsh E_N$}
\put(165,48){\vector(2,-1){30}} \put(177,45){$\Qsh E_N$}
\put(200,25){$\Qsh\prod_r N$}
\end{picture}
}
\end{equation}
When composed with $\beta$ the two $\Qsh\beta$ arrows can be replaced with $\copsh_N$. The resulting
diagram is now the $\Fsh$ of a commutative diagram. As a matter of fact, 
\begin{align*}
\Gsh p_r\ci\Gsh E_N\ci\Gsh\eLr_N\ci\etash_{\Gsh N}&=\Gsh\epsh_N\ci\Gsh p'_r\ci\Gsh\eLr_N\ci\etash_{\Gsh N}=\\
&=\Gsh\epsh_N\ci\Gsh\ellr(r)_{\Gsh N}\ci\etash_{\Gsh N}=\\
&\eqby{Take for etash}\Gsh\epsh_N\ci\elll(r)_{\Qsh N}\ci\etash_{\Gsh N}=\\
&=\Gsh\epsh_N\ci\Gsh p'_r\ci\eLl_{\Qsh N}\ci\etash_{\Gsh N}=\\
&=\Gsh p_r\ci\Gsh E_N\ci\eLl_{\Qsh N}\ci\etash_{\Gsh N}
\end{align*}
holds for all $r\in R$ and the $\Gsh p_r$ are projections of a product. This finishes the proof of
$\Qsh\Lambda_{\hat N}\ci\beta=\Fsh\eLl_N\ci\beta$ and therefore the existence of a unique arrow in $\M_R$ making (\ref{diag: alpha beta}) commutative. That this arrow can be lifted to $_R\M_R$ follows from that both $\beta$ and $\Fsh\iota_{\hat N}$ are in the image of $\phish$, see Lemma \ref{lem: bimodule lift of beta}, and from the fact that $\Fsh\iota{\hat N}$ is monic. This $\alpha:\hat N\to \bimQ \hat N$ therefore satisfies
\begin{equation}\label{j alpha hatbeta}
\bimF\iota_{\hat N}\ci\alpha=\hat\beta
\end{equation}
where $\hat\beta$ was defined in Lemma \ref{lem: bimodule lift of beta}.

Finally, 
\begin{gather*}
\bimF\iota\bimF\Gsh N\ci\bimF\bimG\bimF\iota\hat N\ci(\bimQ\alpha\ci\alpha)\eqby{j alpha hatbeta}\\ =\bimF\iota\bimF\Gsh N\ci\bimF\bimG\hat\beta\ci\alpha=[\text{interchange}]\\
=\bimF\Gsh\beta\ci\bimF\iota\hat N\ci\alpha\eqby{j alpha hatbeta}\\
=\bimF\Gsh\beta\ci\hat\beta\eqby{bimodule lifted coaction}\\
=\bimF\etash\Gsh N\ci\hat\beta\eqby{j alpha hatbeta}\\
=\bimF\etash\Gsh N\ci\bimF\iota\hat N\ci\alpha\eqby{j3}\\
=\bimF\iota\bimF\Gsh N\ci\bimF\bimeta\Gsh N\ci\bimF\iota\hat N\ci\alpha=[\text{interchange}]\\
=\bimF\iota\bimF\Gsh N\ci\bimF\bimG\bimF\iota\hat N\ci(\bimcop\hat N\ci\alpha)
\end{gather*}
proves associativity of the coaction $\alpha$ and
\begin{align*}
\phish(\bimeps\hat N\ci\alpha)&\eqby{j2}\epsh N\ci\Fsh\iota\hat N\ci\phish\alpha=\\
&\eqby{diag: alpha beta} \epsh N\ci\beta=N
\end{align*}
proves its counitality. 
\end{proof}

\begin{pro}\label{pro: M^Qhat=M^Q}
The functor $\psh:\M^\bimQ\to\M^\Qsh$ defined in Corollary \ref{cor: psh} is an isomorphism of categories
and the functor $\Xi:\M^\Qsh\to\,_R\M_R$ defined in Lemma \ref{lem: Xi} is comonadic.
\end{pro}
\begin{proof}
The object map of a functor ${\psh}^{-1}:\M^\Qsh\to\M^\bimQ$ has been given by $\bra N,\beta\ket\mapsto\bra\hat N,\alpha\ket$ in Lemma \ref{lem: psiinverse object map}. For arrows $\bra N,\beta\ket\rarr{f}\bra N',\beta'\ket$
we can define ${\psh}^{-1}f$ as $\hat f$ which is the bimodule map $\hat f:\hat N\to\hat N'$ of Lemma
\ref{lem: Xi}. Indeed, $\hat f$ satisfies $\bimQ \hat f\ci\alpha=\alpha'\ci \hat f$ because
\begin{align*}
j_{\hat N'}\ci\phish\bimQ \hat f\ci\phish\alpha&=\Qsh f\ci j_{\hat N}\ci\phish\alpha=\Qsh f\ci\beta=\\
&=\beta'\ci f=j_{\hat N'}\ci\phish\alpha'\ci\phish\hat f
\end{align*}
and $j$ is monic, $\phish$ is faithful.

The composite functor $\psh{\psh}^{-1}$ maps the object $\bra N,\beta\ket$ at first to $\bra \hat N,\alpha\ket$
and then to $\bra N, j_{\hat N}\ci\phish\alpha\ket$, see Corollary \ref{cor: psh}. But this object is just $\bra N,\beta\ket$ by (\ref{diag: alpha beta}). Therefore $\psh{\psh}^{-1}$ is the identity functor.
The composite functor ${\psh}^{-1}\psh$ maps the object $\bra M,\alpha\ket$ at first to $\bra \phish M,j_M\ci\phish\alpha\ket$ and then back to $\bra M,\alpha\ket$ since $\alpha$ is the unique solution in diagram
(\ref{diag: alpha beta}). Therefore ${\psh}^{-1}\psh$ is the identity functor, too.

As for the comonadicity of $\Xi$ notice that $\Xi=\F^\bimQ{\psh}^{-1}$ and the canonical forgetful functor $\F^\bimQ$ is comonadic.
\end{proof}

\subsection{Digression: The actegory picture}

The relation of the comonad $\Qsh$ on $\M_R$ to the monoidal comonad $\bimQ$ on $_R\M_R$ suggests
an actegory interpretation. 

Consider $\M_R$ as a right $_R\M_R$-actegory, i.e., a category on which the monoidal category $_R\M_R$ acts
on the right. The forgetful functor $\phish:\,_R\M_R\to\M_R$ then becomes a morphism from the right regular 
$_R\M_R$-actegory. This is manifested in the natural isomorphism 
\[
\phish_{L,M}:\phish L\oR M\iso\phish(L\oR M),\qquad L,M\in\,_R\M_R\,.
\]
Then we can also define
\[
\Fsh_{U,V}:=\phish\bimF_{U,V}\ci\phish_{\bimF U,\bimF V}:\ \Fsh U\oR\bimF V\iso\Fsh(U\odot V)
\]
for $U,V\in\Pre$,
\[
\Gsh_{N,M}:\Gsh N\odot\bimG M\to\Gsh(N\oR M),\qquad N\in\M_R,\ M\in\,_R\M_R
\]
by
\[
\Gsh_{N,M}:=\Gsh(\epsh_N\oR\bimeps_M)\ci\Gsh{\Fsh}^{-1}_{\Gsh N,\bimG M}\ci\etash_{\Gsh N\odot\bimG M}
\]
and finally
\[
\Qsh_{N,M}:=\Fsh\Gsh_{N,M}\ci\Fsh_{\Gsh N,\bimG M}:\ \Qsh N\oR\bimQ M\to\Qsh(N\oR M),\ \ N\in\M_R,\ M\in\,_R\M_R\,.
\]
All these natural transformations obey coherence conditions which look like the relations of monoidal functors 
except that the left most object is `smashed' to the actegory. Also the monoidality relations of $\bimeta$, $\bimeps$ have analogous actegorical counterparts involving one $\bimeta$ (resp. $\bimeps$) and two
$\etash$ (resp. $\epsh$). 

\begin{lem}
The monic natural transformation $\iota:\bimG\to\Gsh\phish$ defined in (\ref{eq: iota}) is such that
for all bimodules $L,M\in\,_R\M_R$ the diagram
\begin{equation*}
\parbox{220pt}{
\begin{picture}(220,100)
\put(-5,0){$\Gsh(\phish L\odot\bimG M)$}
\put(10,80){$\bimG L\odot\bimG M$}

\put(30,70){\vector(0,-1){50}} \put(-15,40){$\iota_L\odot\bimG M$}

\put(70,2){\vector(1,0){50}} \put(77,12){$\Gsh_{\phish L,M}$}

\put(70,82){\vector(1,0){50}} \put(82,90){$\bimG_{L,M}$}

\put(125,0){$\Gsh(\phish L\oR M)$}
\put(135,80){$\bimG(L\oR M)$}

\put(170,13){\vector(2,1){35}} \put(200,16){$\Gsh\phish_{L,M}$}
\put(170,72){\vector(2,-1){35}} \put(190,68){$\iota_{L\oR M}$}

\put(180,40){$\Gsh\phish(L\oR M)$}
\end{picture}
}
\end{equation*}
 is commutative.
\end{lem}
\begin{proof}
The $A$ component $(\iota_L)_A$ of $\iota_L$ maps $f:\bimf A\to L$ to $\phish f:\fsh A\to\phish L$.
Therefore on a generic element $[f,g,t]_{A,B}^C$ the lower threefold composite performs
\begin{align*}
[f,g,t]_{A,B}^C&\mapsto[\phish f,g,t]_{A,B}^C\ \mapsto\ (\phish f\oR g)\ci(\fsh)^{-1}_{A,B}\ci\fsh t\\
&\mapsto\phish(f\oR g)\ci\phish_{\bimf A,\bimf B}\ci(\fsh)^{-1}_{A,B}\ci\fsh t=\\
&=\phish((f\oR g)\ci\bimf^{-1}_{A,B}\ci \bimf t)
=\phish\bimG_{L,M}([f,g,t]_{A,B}^C)=\\
&=\iota_{L\oR M}\ci\bimG_{L,M}([f,g,t]_{A,B}^C)
\end{align*}
which is the same as what the upper composite does.
\end{proof}

\begin{cor}
$\M^\Qsh$ is a right $\M^\bimQ$-actegory by defining 
\[
\bra N,\beta\ket\oR\bra M,\alpha\ket:=\bra N\oR M,\Qsh_{N,M}\ci(\beta\oR\alpha)\ket
\]
and the functor $\psh:\M^\bimQ\to\M^\Qsh$ of Corollary \ref{cor: psh} can be endowed with a natural isomorphism
\[
\psh_{\bra M,\alpha\ket,\bra M',\alpha'\ket}:\psh\bra M,\alpha\ket\oR\bra M',\alpha'\ket
\to\psh(\bra M,\alpha\ket\oR\bra M',\alpha'\ket)
\]
lifting $\phish_2$
so that $\psh$ becomes an isomorphism of $\M^\bimQ$-actegories from the regular actegory.
\end{cor}
\begin{proof}
Showing that $\phish_2$ lifts to $\psh_2$ means that the components $\phish_{M,M'}$ are morphisms of 
comodules. Apart from a naturality square for $\phish_2$ this  is commutativity of
\begin{equation*}
\begin{CD}
\phish\bimQ M\oR\bimQ M'@>j_M\oR\bimQ M'>>\Qsh\phish M\oR\bimQ M'@>\Qsh_{\phish M,M'}>>\Qsh(\phish M\oR M')\\
@V{\phish_{\bimQ M,\bimQ M'}}VV @. @VV{\Qsh\phish_{M,M'}}V\\
\phish(\bimQ M\oR\bimQ M')@>\phish\bimQ_{M,M'}>>\phish\bimQ(M\oR M')@>j_{M\oR M'}>>\Qsh\phish(M\oR M')
\end{CD}
\end{equation*}
Since $j=\Fsh\iota$, commutativity of this hexagon can be shown using the pentagon diagram for $\iota$ in the above Lemma:
\begin{gather*}
\Qsh\phish_{M,M'}\ci\Qsh_{\phish M,M'}\ci(j_M\oR \bimQ M')=\\
\Qsh\phish_{M,M'}\ci\Fsh\Gsh_{\phish M,M'}\ci\Fsh_{\Gsh\phish M,\bimG M'}\ci(\Fsh\iota_M\oR\bimF\bimG M')=\\
\Fsh(\Gsh\phish_{M.,M'}\ci\Gsh_{\phish M,M'}\ci(\iota_M\odot\bimG M'))\ci\Fsh_{\bimG M,\bimG M'}=\\
\Fsh(\iota_{M\oR M'}\ci\bimG_{M,M'})\ci\Fsh_{\bimG M,\bimG M'}=\\
j_{M\oR M'}\ci\phish\bimF\bimG_{M,M'}\ci\phish\bimF_{\bimG M,\bimG M'}\ci\phish_{\bimF\bimG M,\bimF\bimG M'}=\\
j_{M\oR M'}\ci\phish\bimQ_{M,M'}\ci\phish_{\bimQ M,\bimQ M'}\,.
\end{gather*}
\end{proof}

\section{The comparison functor} \label{s: comp}

\subsection{The adjunction $\bimK\dashv\bimL$}

The comparison functor $\bimK$ associated to the left adjoint functor $\bimF$ is the functor
\begin{equation}
\bimK:\Pre\to\M^\bimQ,\quad \bimK U=\bra\bimF U,\bimF\bimeta_U\ket.
\end{equation}
The monoidal structure of $\bimF$ induces the following monoidal structure on $\bimK$:
\begin{align*}
 \bimK_{U,V}&=\bra\bimF U\oR\bimF V,\bimQ_{\bimF U,\bimF V}\ci(\bimF\bimeta_U\oR\bimF\bimeta_V)\ket
\rarr{\bimF_{U,V}}\bra\bimF(U\odot V),\bimF\bimeta_{U\odot V}\ket\\
\bimK_0&=\bra R,\bimQ_0\ket\rarr{\bimF_0}\bra\bimF \YI,\bimF\bimeta_\YI\ket.
\end{align*}
\begin{lem} \label{lem: bimK}
The comparison functor $\bimK:\Pre\to\M^\bimQ$ is a left exact strong monoidal functor which is uniquely determined
by the factorizations
\begin{equation*}
\bimF=\F^\bimQ\bimK,\qquad \G^\bimQ=\bimK\bimG
\end{equation*}
of monoidal functors. In particular, we obtain $\bimf$ as the composite
\begin{equation*}
\begin{CD}
\Pre@>\bimK>>\M^\bimQ\\
@AYAA @VV{\F^\bimQ}V\\
\C@>\bimf>>_R\M_R
\end{CD}
\end{equation*}
of three strong monoidal functors.
\end{lem}
\begin{proof}
The characterization of the comparison functor by these factorizations is a well known fact of (co)monad theory.
Here we only need to consider comonads in the 2-category $\MonCat$. The diagram is then obvious.
Strong monoidality of $\bimK$ follows either from the above explicit formulas for $\bimK_2$ and $\bimK_0$ or,
formally, from the factorization $\bimF=\F^\bimQ\bimK$ using that both$\bimF$ and $\F^\bimQ$ are strong
monoidal and $\F^\bimQ$ reflects isomorphisms. Left exactness of $\bimK$ also follows from $\bimF=\F^\bimQ\bimK$. As a matter of fact, $\bimF$ is left exact and $\F^\bimQ$, being faithful, reflects monics
therefore $\bimK$ preserves monics. But both $\Pre$ and $\M^\bimQ$ are abelian categories, hence by standard
arguments one can see that $\bimK$ preserves kernels. Since $\bimK$ is additive, it preserves finite limits, too.
\end{proof}

The functor $\bimK$ can be equivalently formulated as the pair $\bra\bimF,\alpha\ket$ where $\alpha=\bimF\bimeta:\bimF\to\bimQ\bimF$
is a monoidal natural transformation satisfying the coaction conditions: $\bimQ\alpha\ci\alpha=\bimcop\bimF\ci\alpha$
and $\bimeps\bimF\ci\alpha=\F$. 
Similarly, if $O$ is a monoidal comonad on $_R\M_R$ and $\E:\Pre\to\M^O$ is a strong monoidal functor such that
$\F^O\E=\bimF$ then $\E U=\bra\bimF U,\omega_U\ket$ with $\omega:\bimF\to O\bimF$ a monoidal natural transformation satisfying the coaction conditions for $O$. However, the pair $\bra\bimF,\alpha\ket$ is universal among them.
\begin{lem}
For a strong monoidal left adjoint $\F:\Pre\to\,_R\M_R$ consider the category of factorizations $\Pre\rarr{\E}\M^O\to\,_R\M_R$ of $\F$
through the forgetful functor $\F^O:\M^O\to\,_R\M_R$ of a monoidal comonad $O$ on $_R\M_R$. 
Then the comparison functor $\bimK$ is an initial object in this category.
\end{lem}
\begin{proof}
Hint: By adjunction from behind $\Nat(\F,O\F)\cong\Nat(\F\G,O)$, $\omega\mapsto\varphi$, there is a bijection between the set of monoidal natural transformations $\omega:\F\to O\F$ satisfying the coaction conditions and the set of monoidal comonad morphisms $\bra \,_R\M_R,\varphi\ket:\bimQ\to O$. Therefore all factorizations of $\F$ have the form
\[
\E U=\bra\F U,\varphi_{\F U}\ci\F\eta_U\ket,\quad U\in\Pre.
\]
This $\varphi$ is unique since $\omega\mapsto O\eps\ci\omega\G$ inverts $\varphi\mapsto \varphi\F\ci\F\eta$.
\end{proof}

For later convenience we introduce a functor $\bimk:\C\to\M^\bimQ$ isomorphic to the composite functor $\bimK Y$
as follows:
\begin{equation} \label{bimk}
\bimk C:=\bra \bimf C,\delta_C\ket,\quad\text{where}\quad \delta_C :=\bimF\bimG \hat N_C\ci\bimF\bimeta_{YC}
\ci{\hat N}^{-1}_C\,.
\end{equation}
Then a natural isomorphism $\bimK Y\iso\bimk$ is given by $\bra \bimF YC,\bimF\bimeta_{YC}\ket
\rarr{\hat N_C}\bra \bimf C,\delta_C\ket$ with $\hat N$ defined in (\ref{diag: N}).







It is well known in comonad theory that the comparison functor has a right adjoint precisely if its domain category
has certain equalizers. The category of presheaves $\Pre$ has all equalizers, in fact it is complete, cocomplete and even a Grothendieck category, therefore a right adjoint $\bimL$ of $\bimK$ exists. It can be defined by choosing
equalizers
\begin{equation} \label{eq: L}
\bimL\bra M,\alpha\ket\overset{i_{\bra M,\alpha\ket}}{\longrightarrowtail}\bimG M
\pair{\bimeta\bimG M}{\bimG\alpha}\bimG\bimF\bimG M
\end{equation}
for all objects $\bra M,\alpha\ket\in\M^\bimQ$. Then the action of $\bimL$ on arrows is uniquely determined
and makes $i$ a natural transformation $\bimL\to\bimG\F^\bimQ$.

Computing this presheaf on an object $C$ one obtains, up to isomorphism, the abelian group
\begin{equation}\label{L}
\bimL\bra M,\alpha\ket C=\M^\bimQ(\bimk C,\bra M,\alpha\ket)
\end{equation}
which consists of bimodule maps $f:\bimf C\to M$ satisfying $\alpha\ci f=\bimQ f\ci\delta_C$.
More precisely, setting (\ref{L}) to be the equality for all $C$ defines a choice of the equalizer in (\ref{eq: L}).
With this definition the natural transformation $i$ just embeds this abelian group into $_R\M_R(\bimf C,M)=\bimG MC$.

The counit $\theta$ and the unit $\nu$ of the adjunction $\bimK\dashv\bimL$ can be obtained as the unique arrows making the diagrams
\begin{equation}\label{diag: theta}
\parbox{250pt}{
\begin{picture}(200,80)(-30,0)
\put(7,10){$M$}
\put(22,10){$\longerrightarrowtail$} \put(48,18){$\alpha$}
\put(90,10){$\bimF\bimG M$}
\put(-10,70){$\bimF\bimL\bra M,\alpha\ket$} \put(-30,40){$\F^\bimQ\theta_{\bra M,\alpha\ket}$}
\put(25,63){\vector(2,-1){70}}
\put(57,50){$\bimF i_{\bra M,\alpha\ket}$}
\dashline{3}(12,63)(12,23) \put(12,23){\vector(0,-1){0}}
\put(123,11){\vector(1,0){50}} \put(136,-3){$\bimF\bimG\alpha$}
\put(123,14){\vector(1,0){50}} \put(140,20){$\bimcop_M$}
\put(180,10){$\bimF\bimG\bimF\bimG M$}
\end{picture}
}
\end{equation}
for all $\bimQ$ comodule $\bra M,\alpha\ket$ and
\begin{equation}\label{diag: nu}
\parbox{250pt}{
\begin{picture}(200,80)(-30,0)
\put(0,10){$\bimL\bimK U$}
\put(24,10){$\longerrightarrowtail$} \put(48,18){$i_{\bimK U}$}
\put(90,10){$\bimG\bimF U$}
\put(7,70){$U$} \put(-3,43){$\nu_U$}
\put(25,63){\vector(2,-1){70}}
\put(57,50){$\bimeta_U$}
\dashline{3}(12,63)(12,25) \put(12,25){\vector(0,-1){0}}
\put(120,11){\vector(1,0){50}} \put(130,-3){$\bimG\bimF\bimeta U$}
\put(120,14){\vector(1,0){50}} \put(130,20){$\bimeta\bimG\bimF U$}
\put(180,10){$\bimG\bimF\bimG\bimF U $}
\end{picture}
}
\end{equation}
for all presheaves $U$, respectively, commutative.
\begin{pro} \label{pro: bimL}
The adjunction $\nu,\theta:\bimK\dashv\bimL:\M^\bimQ\to\Pre$ is a monoidal adjunction in which
$\theta$ is invertible and $\K$ is left exact strong opmonoidal. Therefore $\bimL$ is a fully faithful monoidal functor with a left exact left adjoint, i.e., a monoidal localization.
\end{pro}
\begin{proof}
Consider the strong monoidal functor $\bimK$ as an opmonoidal functor. Then its right adjoint carries a unique
monoidal structure such that the given adjunction data $\nu$ and $\theta$ become monoidal natural transformations.
That makes the adjunction $\nu,\theta:\bimK\dashv\bimL:\M^\bimQ\to\Pre$ automatically an adjunction in $\MonCat$.

Since $\bimF$ is left exact, it preserves the equalizer (\ref{eq: L}). Therefore the dashed arrow in (\ref{diag: theta})
is an isomorphism. But $\F^\bimQ$ is comonadic hence reflects isomorphisms. Thus $\theta$ is an isomorphism, i.e.,
$\bimL$ is full and faithful.
\end{proof}

\subsection{The adjunction $\Ks\dashv\Ls$}

The construction of $\bimK\dashv\bimL$ in the previous subsection can be repeated without much change
by starting with the functor $\Fsh$ instead of $\bimF$. The only difference is that we will be lacking of any monoidal
structure of these functors.

The comparison functor $\Ks:\Pre\to\M^\Qsh$, defined by the factorizations $\Fsh=\F^\Qsh\Ks$, $\G^\Qsh=\Gsh\Ks$,
is given by
\[
\Ks (U\rarr{\chi}V)\ =\ \left(\bra\Fsh U,\Fsh\etash_U\ket\ \rarr{\Fsh\chi}\ \bra\Fsh V,\Fsh\etash_V\ket\right)\,.
\]
A right adjoint for $\Ks$ can be defined by the equalizers
\begin{equation}
\Ls\bra N,\beta\ket\overset{\ish_{\bra N,\beta\ket}}{\longrightarrowtail}\Gsh N
\pair{\etash\Gsh N}{\Gsh\beta}\Gsh\Fsh\Gsh N
\end{equation}
for $\bra N,\beta\ket\in\M^\Qsh$.
The analogue of Proposition \ref{pro: bimL} can now be stated without proof.
\begin{pro} \label{pro: Ls}
The comparison functor $\Ks$ is the reflection of a localization$ \Ls:\M^\Qsh\to\Pre$ which can be given on objects
$\bra N,\beta\ket\in\M^\Qsh$ as the subfunctor
\[
\Ls\bra N,\beta\ket C=\M^\Qsh(\ksh C,\bra N,\beta\ket)\quad\subset\quad \M_R(\fsh C,N)=\Gsh N C
\]
where $\ksh:=\psh\bimk:\C\to\M^\Qsh$ with $\psh$ denoting the category equivalence defined in
Corollary \ref{cor: psh}.
\end{pro}
Explicit formula for the functor $\ksh$ is
$\ksh C:=\bra \fsh C,\deltash_C\ket$ where $\deltash:=j\bimf\ci\phish\delta$ with $\delta_C$ defined in (\ref{bimk}).

The unit and counit of $\Ks\dashv\Ls$ are denoted by $\nush$ and $\thetash$, respectively. They are uniquely
determined by the equations
\begin{align}
\label{nush}
\ish_{\Ks U}\ci\nush_U&=\etash_U\\
\label{thetash}
\beta\ci\F^\Qsh\thetash_{\bra N,\beta\ket}&=\Fsh\ish_{\bra N,\beta\ket}
\end{align}
for all objects $U\in\Pre$, $\bra N,\beta\ket\in\M^\Qsh$, in complete analogy with (\ref{diag: nu}) and 
(\ref{diag: theta}).

In the rest of this subsection we wish to compare the two adjunctions $\bimK\dashv\bimL$ and $\Ks\dashv\Ls$ in order to see the relation of $\bimL\bimK$ to $\Ls\Ks$. At first we compare the comparison functors.
Computing
\begin{align*}
\psh\bimK U&=\psh\bra\bimF U,\bimF\bimeta U\ket=\bra\Fsh U,\Fsh\iota\bimF U\ci\Fsh\bimeta U\ket\\
&\eqby{Take for etash}\bra\Fsh U,\Fsh\etash U\ket=\Ks U
\end{align*}
we get an equality of functors:
\begin{equation}
\psh\bimK\ =\ \Ks\,.
\end{equation}
This complies with the definition of $\ksh$ as $\psh\bimk$ in the above Proposition and leads to another variant of the natural isomorphism $N:\F Y\iso F$ of (\ref{diag: N}). As a matter of fact, since
\begin{align*}
\deltash\ci \phish\bimN&=\Fsh\iota\bimf\ci\phish\delta\ci\phish\bimN\eqby{bimk}
\Fsh\iota\bimf\ci\phish\bimQ\bimN\ci \Fsh\bimeta Y=\\
&=\Qsh\phish\bimN\ci\Fsh\iota\bimF Y\ci\Fsh\bimeta Y\eqby{Take for etash}\Qsh\phish\bimN\ci\Fsh\etash Y\,,
\end{align*}
the arrow
\[
\bra \Fsh YC, \Fsh\etash YC\ket\rarr{N_C}\bra \fsh C,\deltash C\ket
\]
is a morphism of $\Qsh$-comodules and defines the $C$-component of a natural isomorphism $\NQ:\Ks Y\iso \ksh$
such that $\F^\Qsh\NQ=\Nsh:=\phish\bimN$.

Next we compare the equalizers defining $\bimL$ and $\Ls$.
\begin{lem} \label{lem: m}
There is a unique natural transformation $m:\bimL\to\Ls\psh$ such that $\ish\psh\ci m=\iota\F^\bimQ\ci\bimi$.
\end{lem}
\begin{proof}
The equalizers defining $\bimL$ and $\Ls$ are connected by the following serially commuting diagram:
\begin{equation}
\parbox{230pt}{
\begin{picture}(230,100)
\put(-10,0){$\Ls\psh\bra M,\alpha\ket$}
\dashline{3}(22,76)(22,12)  \put(22,12){\vector(0,-1){0}}    \put(-6,40){$\sst m_{\bra M,\alpha\ket}$}
\put(4,80){$\bimL\bra M,\alpha\ket$}
\put(48,0){$\longerrightarrowtail$}  \put(58,8){$\sst \ish\psh\bra M,\alpha\ket$}
\put(48,80){$\longerrightarrowtail$}  \put(62,87){$\sst \bimi\bra M,\alpha\ket$}
\put(110,0){$\Gsh\phish M$}
\put(125,76){\vector(0,-1){64}} \put(110,40){$\sst\iota_M$}
\put(118,80){$\bimG M$}
\put(150,1){\vector(1,0){60}} \put(155,-8){$\sst \Gsh(i_M\ci\phish\alpha)$}
\put(150,4){\vector(1,0){60}} \put(157,8){$\sst \etash\Gsh\phish M$}
\put(150,81){\vector(1,0){60}} \put(175,71){$\sst\bimG\alpha$}
\put(150,84){\vector(1,0){60}} \put(170,88){$\sst\bimeta\bimG M$}
\put(215,0){$\Gsh\Fsh\Gsh\phish M$}
\put(240,76){\vector(0,-1){64}} \put(215,40){$\sst\iota\bimF\iota_M$}
\put(225,80){$\bimG\bimF\bimG M$}
\end{picture}
}
\end{equation}
Indeed,
\begin{align*}
\iota\bimF\iota_M\ci\bimeta\bimG M&=(\iota\bimF\ci\bimeta)\Gsh\phish M\ci\iota_M=\etash\Gsh\phish M\ci\iota_M\\
\text{and}\qquad \iota\bimF\iota_M\ci\bimG\alpha&=\Gsh\Fsh\iota_M\ci\iota\bimF\bimG M\ci\bimG\alpha
=\Gsh\Fsh\iota_M\ci\Gsh\phish\alpha\ci\iota_M\,.
\end{align*}
Therefore $\iota_M\ci\bimi_{\bra M,\alpha\ket}$ factors uniquely through $\ish_{\psh\bra M,\alpha\ket}$.
\end{proof}
\begin{lem} \label{lem: m iso}
$m:\bimL\to\Ls\psh$ is a natural isomorphism satisfying $m\bimK\ci\bimnu=\nush$.
\end{lem}
\begin{proof}
Since $\ish$ is monic, the calculation
\begin{align*}
\ish\Ks\ci m\bimK\ci\bimnu&=(\ish\psh\ci m)\bimK\ci\bimnu=\iota\bimK\ci\bimi\bimK\ci\bimnu=\\
&\eqby{diag: nu}\iota\bimK\ci\bimeta\eqby{Take for etash}\etash=\\
&\eqby{nush}\ish\Ks\ci\nush
\end{align*}
proves the relation and therefore $m$ can be expressed as
\begin{align} \label{m by adj}
\Ls\psh\bimtheta\ci\nush\bimL&=\Ls\psh\bimtheta\ci m\bimK\bimL\ci\bimnu\bimL=
m\ci\bimL\bimtheta\ci\bimnu\bimL=\\
&=m\,. \notag
\end{align}
It follows that $m$ is the standard isomorphism connecting two right adjoints of a functor. Namely,
$\bimnu,\bimtheta:\bimK\dashv \bimL$ and $\nush,{\psh}^{-1}\thetash\psh:\bimK\dashv \Ls\psh$ are two
adjunctions, therefore $m^{-1}=\bimL{\psh}^{-1}\thetash\psh\ci\bimnu\Ls\psh$ is the inverse of $m$.
\end{proof}

The next Proposition summarizes the content of this section and serves also as input for the next section.
\begin{pro} \label{pro: T}
Let $\C$ be a small monoidal $\Ab$-category and $F:\C\to\Ab$ a flat, essentially strong monoidal functor.
Let $\bimF:\Pre\to\,_R\M_R$ be the strong part of the left Kan extension of $F$ and let $\Fsh:=\phish\bimF:\Pre\to\M_R$.
Then there is a construction of
\begin{enumerate}
\item a monoidal localization $\bimL:\M^\bimQ\to\Pre$ with reflection being the comparison functor $\bimK$
associated to the left adjoint $\bimF$
and with adjunction data $\bimnu,\bimtheta:\bimK\dashv\bimL$,
\item a localization $\Ls:\M^\Qsh\to\Pre$ with reflection being the comparison functor $\Ks$ associated to the
left adjoint $\Fsh$ and with adjunction data $\nush,\thetash:\Ks\dashv\Ls$,
\item a left exact, monoidal idempotent monad $\bimT=\bra \bimt,\bimmu,\bimnu\ket$ on $\Pre$, where $\bimt=\bimL\bimK$ and $\bimmu:=\bimL\bimtheta\bimK$, satisfying the property
\begin{equation} \label{eq: special}
\bimt(U\odot\bimnu_V),\quad \bimt(\bimnu_U\odot V)\quad\text{are invertible}\quad\forall U,V\in\ob\Pre,
\end{equation}
\item a left exact idempotent monad $\Tsh=\bra\tsh,\mush,\nush\ket$ on $\Pre$, where $\tsh=\Ls\Ks$ and
$\mush:=\Ls\thetash\Ks$
\item a monad isomorphism $\bimT\iso\Tsh$.
\end{enumerate}
\end{pro}
\begin{proof}
$(1)$ and $(2)$ have been shown in Propositions \ref{pro: bimL} and \ref{pro: Ls}, respectively.

In $(3)$ the only nontrivial fact is property (\ref{eq: special}). It suffices to prove that $\bimK(U\odot\bimnu_V)$
is invertible for all presheaves $U$ and $V$. This follows from that $\bimK$ is strong monoidal and $\bimL$ is fully faithful. Indeed, $\bimtheta\bimK\ci\bimK\bimnu=\bimK$ by adjunction, hence
$\bimK(U\odot\bimnu_V)=\bimK_{U,TV}\ci(\bimK U\oR\bimK\bimnu_V)\ci\bimK^{-1}_{U,V}$ is invertible.

$(4)$ is obvious and the monad isomorphism of $(5)$ is $m\bimK$ provided by Lemma \ref{lem: m iso} since
we not only have $m\bimK\ci\bimnu=\nush$ but
\begin{align*}
{\psh}^{-1}\thetash\psh\ci\bimK m&\eqby{m by adj}{\psh}^{-1}\thetash\psh\ci{\psh}^{-1}\Ks\Ls\psh\bimtheta
\ci\bimK\nush\bimL=\\
&=\bimtheta\ci{\psh}^{-1}\thetash\psh\bimK\bimL\ci\bimK\nush\bimL
=\bimtheta\ci{\psh}^{-1}\left(\thetash\Ks\ci\Ks\nush\right)\bimL=\\
&=\bimtheta
\end{align*}
as well, therefore 
\begin{align*}
m\bimK\ci\bimmu&=m\bimK\ci\bimL\bimtheta\bimK=\Ls\psh\bimtheta\bimK\ci m\bimK\bimL\bimK=\\
&=\Ls\thetash\psh\bimK\ci\Ls\Ks m\bimK\ci m\bimK\bimL\bimK=\\
&=\mush\ci \tsh m\bimK\ci m\bimK \bimt\,.
\end{align*}
\end{proof}

\section{The monoidal idempotent monad $\T$ and its sheaves} \label{s: T}

In the previous Section we have constructed a left exact monoidal idempotent monad $\bimT$ on the presheaf category $\Pre$ satisfying a special property (\ref{eq: special}). While left exact idempotent monads on $\Pre$ are known to correspond to Grothendieck topologies on $\C$ \cite{Borceux} and in this way to sheaf categories that are the EIlenberg-Moore categories $\Pre_\bimT$, monoidal monads in general do not have monoidal Eilenberg-Moore categories. Only the Kleisli category carries monoidal structure. We shall see that property (\ref{eq: special}) solves
this problem.

\subsection{The monoidal structure of $\T$-modules}

In this subsection $\T$ denotes a monoidal idempotent monad $\bra T,\mu,\nu\ket$ on a monoidal category $\bra\Pre,\odot,\YI\ket$. This means that $\bra T,\mu,\nu\ket$ is a monad with invertible multiplication $\mu:T^2\iso T$, $T$ is a monoidal functor with structure maps $T_{U,V}:TU\odot TV\to T(U\odot V)$, $T_0:\YI\to T\YI$ and
both $\mu$ and $\nu$ are monoidal natural transformations. The latter means that
\begin{align}
\label{mu mul}
T_{U,V}\ci(\mu_U\odot\mu_V)&=\mu_{U\odot V}\ci TT_{U,V}\ci T_{TU,TV}\\
\label{mu uni}
T_0&=\mu_\YI\ci TT_0\ci T_0\\
\label{nu mul}
T_{U,V}\ci(\nu_U\odot\nu_V)&=\nu_{U\odot V}\\
\label{nu uni}
T_0&=\nu_\YI\,.
\end{align}

Due to idempotency, i.e., invertibility of $\mu$, a pair $\bra U,\alpha\ket$ is an object of the Eilenberg-Moore category $\She$ associated to the monad $\T$ iff $\nu_U$ is invertible and in this case the action $\alpha:TU\to U$ is unique: $\alpha=\nu_U^{-1}$. The canonical adjunction $\nu,\tau:\F_\T\dashv\G_\T:\She\to\Pre$ associated to the monad $\T$ consists of a left adjoint $\F_\T:U\mapsto\bra TU,\mu_U\ket$ and a right adjoint $\G_\T:\bra U,\nu_U^{-1}\ket\mapsto U$, the forgetful functor, which, due to naturality of $\nu$, is fully faithful.
The unit of the adjunction $\F_\T\dashv\G_\T$ is the unit of the monad, i.e., $\nu$, and the counit is
\begin{equation} \label{eq: tau}
\tau:\F_\T\G_\T\to \She,\quad\tau_{\bra U,\nu_U^{-1}\ket}=\bra TU,\mu_U\ket\rarr{\nu_U^{-1}}\bra U,\nu_U^{-1}\ket.
\end{equation}

\begin{defi}\label{def: spec}
The monoidal idempotent monad $\T$ is called special if it satisfies the invertibility condition of (\ref{eq: special}).
\end{defi}
\begin{pro} \label{pro: She mon}
For a special monoidal idempotent monad $\T$ on a monoidal category $\Pre$ let $\She$ be its Eilenberg-Moore
category of $\T$-modules. Then $\She$ has a monoidal structure with
\begin{itemize}
\item monoidal product: $\Ush \ot\Vsh:=\bra T(U\odot V),\mu_{U\odot V}\ket$
\item monoidal unit: $\bra T\YI,\mu_\YI\ket$
\item and coherence isomorphisms:
\begin{trivlist}
\item $\asso_{\Ush,\Vsh,\Wsh}:=T(\nu_{U\odot V}\odot W)\ci T\asso_{U,V,W}\ci[T(U\odot\nu_{V\odot W})]^{-1}$
\item $\luni_{\Ush}:=\luni_U\ci\nu^{-1}_{\YI\odot U}\ci[T(\nu_\YI\odot U)]^{-1}$
\item $\runi_{\Ush}:=\runi_U\ci\nu^{-1}_{U\odot \YI}\ci[T(U\odot\nu_\YI)]^{-1}$
\end{trivlist}
\end{itemize}
where $\asso_{U,V,W}$, $\luni_U$, $\runi_U$ denote the coherence isomorphisms of $\Pre$.
\end{pro}
\begin{proof}
Note that by naturality of $\nu$, $\nu_U\ci\luni_U=T\luni_U\ci\nu_{\YI\odot U}$, therefore $\nu_{\YI\odot U}$ is invertible whenever $\nu_U$ is. This shows that $\luni_{\Ush}$ is well defined. By a similar argument $\nu_{U\odot \YI}$ is also invertible in the definition of $\runi_{\Ush}$.

Note also that the formulas for $\asso$, $\luni$ and $\runi$ in fact give the $\G_T$ of these arrows and our first task is to show that the given composite arrows can be lifted to yield arrows of $\C_\T$. This means that they have to
satisfy
\begin{align*}
\mu_{T(U\odot V)\odot W}\ci T\asso_{\Ush,\Vsh,\Wsh}&=\asso_{\Ush,\Vsh,\Wsh}\ci\mu_{U\odot T(V\odot W)}\\
\nu_U^{-1}\ci T\luni_{\Ush}&=\luni_{\Ush}\ci \mu_{T\YI\odot U}\\
\nu_U^{-1}\ci T\runi_{\Ush}&=\runi_{\Ush}\ci \mu_{U\odot T\YI}
\end{align*}
These relations follow easily from naturality of $\mu$, $\nu$ and from the monad axioms. The details are omitted.

As for the naturality of the resulting associator in $\Ush$, \dots, etc notice that the objects of $\She$ that enter are all tensor products and therefore their $\T$-actions are components of $\mu$. Thus naturality of $\mu$ guaranties
naturality of $\asso$. For $\luni$ this argument does not work but we can check it explicitely: For all
$\alpha:\Ush\to\Vsh$
\begin{align*}
\luni_{\Vsh}\ci T(T\YI\odot\alpha)&=\luni_V\ci\nu^{-1}_{\YI\odot V}\ci T(\YI\odot\alpha)\ci[T(\nu_\YI\odot U)]^{-1}=\\
&=\luni_V\ci(\YI\odot\alpha)\ci\nu^{-1}_{\YI\odot U}\ci[T(\nu_\YI\odot U)]^{-1}=\\
&=\alpha\ci\luni_{\Ush}\,.
\end{align*}

We turn to the proof of the coherence constraints. In the calculations below it is important to note that
for arrows $\alpha,\beta\in\She$ the ($\G_\T$ of their) monoidal product is $T(\alpha\odot \beta)$.

The pentagon relation:
\begin{gather*}
(\asso_{\Ush,\Vsh,\Wsh}\ot \Zsh)\ci\asso_{\Ush,\Vsh\ot\Wsh,\Zsh}\ci\\
\ci(\Ush\ot\asso_{\Vsh,\Wsh,\Zsh})=\\
=T(T(\nu_{U\odot V}\odot W)\odot Z)\ci T(T\asso_{U,V,W}\odot Z)\ci T(T(U\odot\nu_{V\odot W})\odot Z)^{-1}\ci\\
\ci T(\nu_{U\odot T(V\odot W)}\odot Z)\ci T\asso_{U,T(V\odot W),Z}\ci [T(U\odot\nu_{T(V\odot W)\odot Z})]^{-1}\ci\\
\ci T(U\odot T(\nu_{V\odot W}\odot Z))\ci T(U\odot T\asso_{V,W,Z})\ci T(U\odot T(V\odot\nu_{W\odot Z}))^{-1}=\\
=T(T(\nu_{U\odot V}\odot W)\odot Z)\ci T(T\asso_{U,V,W}\odot Z)\ci T(\nu_{U\odot(V\odot W)}\odot Z)\ci\\
   \ci[T((U\odot\nu_{V\odot W})\odot Z)]^{-1}
\ci T\asso_{U,T(V\odot W),Z} \ci T(U\odot(\nu_{V\odot W}\odot Z))\ci\\
\ci[T(U\odot\nu_{(V\odot W)\odot Z})]^{-1}
\ci T(U\odot T\asso_{V,W,Z})\ci T(U\odot T(V\odot\nu_{W\odot Z}))^{-1}=
\end{gather*}
\begin{gather*}
=T(T(\nu_{U\odot V}\odot W)\odot Z)\ci T(\nu_{(U\odot V)\odot W}\odot Z)\ci T(\asso_{U,V,W}\odot Z)
   \ci T\asso_{U,V\odot W,Z}\ci\\
\ci T(U\odot\asso_{V,W,Z})\ci[T(U\odot\nu_{V\odot(W\odot Z)})]^{-1}\ci T(U\odot T(V\odot\nu_{W\odot Z}))^{-1}=\\
=T(\nu_{T(U\odot V)\odot W}\odot Z)\ci T((\nu_{U\odot V}\odot W)\odot Z)\ci T\asso_{U\odot V,W,Z}
   \ci T\asso_{U,V,W\odot Z}\ci\\
\ci [T(U\odot(V\odot\nu_{W\odot Z}))]^{-1}\ci[T(U\odot\nu_{V\odot T(W\odot Z)})]^{-1}=\\
=T(\nu_{T(U\odot V)\odot W}\odot Z)\ci T\asso_{T(U\odot V),W,Z}\ci T(\nu_{U\odot V}\odot (W\odot Z))\ci\\
\ci[T((U\odot V)\odot\nu_{W\odot Z})]^{-1}\ci T\asso_{U,V,T(W\odot Z)}\ci[T(U\odot\nu_{V\odot T(W\odot Z)})]^{-1}=\\
=T(\nu_{T(U\odot V)\odot W}\odot Z)\ci T\asso_{T(U\odot V),W,Z}\ci [T(T(U\odot V)\odot\nu_{W\odot Z})]^{-1}\ci\\
\ci T(\nu_{U\odot V}\odot T(W\odot Z))\ci T\asso_{U,V,T(W\odot Z)}\ci[T(U\odot\nu_{V\odot T(W\odot Z)})]^{-1}=\\
=\asso_{\Ush\ot\Vsh,\Wsh,\Zsh}\ci\asso_{\Ush,\Vsh,\Wsh\ot\Zsh}\,.
\end{gather*}

The triangle relation:
\begin{gather*}
(\runi_{\Ush}\ot\Vsh)\ci\asso_{\Ush,\bra T\YI,\mu_\YI\ket,\Vsh}=\\
=T(\runi_U\odot V)\ci T(\nu^{-1}_{U\odot \YI}\odot V)\ci[T(T(U\odot \nu_\YI)\odot V)]^{-1}\ci T(\nu_{U\odot T\YI}\odot V)
\ci\\
\ci T\asso_{U,T\YI,V}\ci[T(U\odot\nu_{T\YI\odot V})]^{-1}=\\
=T(\runi_U\odot V)\ci T\asso_{U,\YI,V}\ci[T(U\odot(\nu_\YI\odot V))]^{-1}\ci [T(U\odot\nu_{T\YI\odot V})]^{-1}=\\
=T(U\odot\luni_V)\ci[T(U\odot\nu_{\YI\odot V})]^{-1}\ci[T(U\odot T(\nu_\YI\odot V))]^{-1}=T(U\odot \luni_{\Vsh})=\\
=\Ush\ot\luni_{\Vsh}\,.
\end{gather*}

Coincidence on the unit object:

\begin{gather*}
\luni_{\bra T\YI,\mu_\YI\ket}=\luni_{T\YI}\ci\nu^{-1}_{\YI\odot T\YI}\ci[T(\nu_\YI\odot T\YI)]^{-1}
=\nu_{T\YI}^{-1}\ci T\luni_{T\YI}\ci[T(\nu_\YI\odot T\YI)]^{-1}=\\
=[T(\luni^{-1}_{T\YI}\ci\nu_\YI)]^{-1}\ci[T(\nu_\YI\odot T\YI)]^{-1}=[T((\YI\odot\nu_\YI)\ci\luni^{-1}_\YI)]^{-1}
\ci[T(\nu_\YI\odot T\YI)]^{-1}=\\
=T\luni_\YI\ci[T(\nu_\YI\odot T\YI)\ci T(\YI\odot\nu_\YI)]^{-1}=T\runi_\YI\ci[T(T\YI\odot\nu_\YI)\ci T(\nu_\YI\odot \YI)]^{-1}=\\
=[T((\nu_\YI\odot \YI)\ci\runi_\YI^{-1})]^{-1}\ci[T(T\YI\odot\nu_\YI)]^{-1}=[T(\runi_{T\YI}^{-1}\ci\nu_\YI)]^{-1}
\ci[T(T\YI\odot\nu_\YI)]^{-1}=\\
=\nu_{T\YI}^{-1}\ci T\runi_{T\YI}\ci[T(T\YI\odot\nu_\YI)]^{-1}=\runi_{T\YI}\ci\nu^{-1}_{T\YI\odot \YI}\ci[T(T\YI\odot\nu_\YI)]^{-1}=\\
=\runi_{\bra T\YI,\mu_\YI\ket}\,.
\end{gather*}
\end{proof}

\begin{pro}
The monoidal structure defined above on the sheaf category $\She$ is such that the canonical adjunction $\F_\T\dashv\G_\T$ is a monoidal adjunction and $\T=\G_\T\F_\T$ as monoidal functors.
\end{pro}
\begin{proof}
It easy to check that
\begin{equation}
\begin{aligned}
(\F_\T)_{U,V} &: \begin{CD}\bra T(TU\odot TV),\mu_{TU\odot TV}\ket@>\mu_{U\odot V}\ci TT_{U,V}>>
\bra T(U\odot V),\mu_{U\odot V}\ket \end{CD}\\
(\F_\T)_0 &: \begin{CD}\bra T\YI,\mu_\YI\ket@=\F_\T \YI\end{CD}
\end{aligned}
\end{equation}
is a monoidal structure on $\F_\T$ and
\begin{equation} \label{eq: G_T mon}
\begin{aligned}
(\G_\T)_{\bra U,\nu_U^{-}\ket,\bra V,\nu_V^{-1}\ket}&:
\begin{CD}U\odot V@>\nu_{U\odot V}>> T(U\odot V)\end{CD}\\
(\G_\T)_0&:\begin{CD}\YI@>\nu_\YI>> T\YI\end{CD}
\end{aligned}
\end{equation}
is a monoidal structure on $\G_\T$. The $\F_\T$ is strong monoidal because
\begin{align*}
T(U\odot V)&=\mu_{U\odot V}\ci T\nu_{U\odot V}=\mu_{U\odot V}\ci TT_{U,V}\ci T(\nu_U\odot \nu_V)=\\
&=(\F_\T)_{U,V}\ci T(\nu_U\odot TV)\ci T(U\odot \nu_V)
\end{align*}
Moreover, the unit of the composite $\G_\T\F_\T$ is $\G_\T(\F_\T)_0$ which
is just $\nu_\YI=T_0$ and
\begin{align*}
(\G_\T\F_\T)_{U,V}&=\G_\T(\F_\T)_{U,V}\ci(\G_\T)_{\F_\T U,\F_\T V}=
\mu_{U\odot V}\ci TT_{U,V}\ci\nu_{TU\odot TV}=\\
&=\mu_{U\odot V}\ci \nu_{T(U\odot V)}\ci T_{U,V}=T_{U,V}.
\end{align*}
The unit of the adjunction $\F_\T\dashv\G_\T$ is $\nu$, therefore monoidal. In order to see monoidality of the
counit (\ref{eq: tau}) we compute
\begin{align*}
(\F_\T\G_\T)_{\bra U,\nu_U^{-1}\ket,\bra V,\nu_V^{-1}\ket}&=T\nu_{U\odot V}\ci \mu_{U\odot V}\ci TT_{U,V}
=TT_{U,V}\\
(\F_\T\G_\T)_0&=T\nu_\YI=TT_0
\end{align*}
and check up on the monoidality relations
\begin{align*}
\tau_{\bra U,\nu_U^{-}\ket\ot\bra V,\nu_V^{-1}\ket}\ci(\F_\T\G_\T)_{\bra U,\nu_U^{-}\ket,\bra V,\nu_V^{-1}\ket}&=
\nu_{T(U\odot V)}^{-1}\ci TT_{U,V}=\\
&=[T(\nu_U\odot \nu_V)]^{-1}=T(\nu_U^{-1}\odot\nu_V^{-1})=\\
&=\tau_{\bra U,\nu_U^{-1}\ket}\ot\tau_{\bra V,\nu_V^{-1}\ket}\\
\tau_{\bra T\YI,\mu_\YI\ket}\ci(\F_\T\G_\T)_0&=\nu_{T\YI}^{-1}\ci T\nu_\YI=\mu_\YI\ci T\nu_\YI=T\YI=(\She)_0\,.
\end{align*}
\end{proof}

\subsection{Monoidal Grothendieck topologies}

Recall that, in the additive setting, a \textit{Grothendieck topology} on $\C$ consists of families $\GroT(A)$ of (additive) sieves $S\into YA$ for each $A\in\ob\C$ such that the following
axioms hold:
\begin{trivlist}
\item (i) $YA$ belongs to $\GroT(A)$ for all $A\in\ob\C$.
\item (ii) If $S\in\GroT(B)$ and $f\in\C(A,B)$ then $f^{-1}(S)\in\GroT(A)$, where $f^{-1}(S)$ denotes the pullback
\[
\begin{CD}
f^{-1}(S)@>>> S\\
@VV{\ \ \ \ \ \text{p.b.}}V @VVV\\
YA@>Yf>>YB
\end{CD}\ .
\]
\item (iii) If $S\in\GroT(C)$ and $R\into YC$ is any sieve satisfying $s^{-1}(R)\in\GroT(\dom s)$ for all
$s\in S$ then $R\in\GroT(C)$.
\end{trivlist}

\begin{thm} \label{thm: GroT}
Let $\T$ be a left exact idempotent monad and define for each $A\in\ob\C$ the family $\GroT(A)$ as the family of subfunctors $i:S\into YA$ for which $Ti$ is invertible. Then $\GroT$ is a Grothendieck topology on $\C$ and
a presheaf $U$ is a $\GroT$-sheaf precisely when $\nu_U$ is invertible. Therefore the Eilenberg-Moore category
$\She$ of $\T$-modules can be identified with the category of $\GroT$-sheaves by restricting the codomain
of the forgetful functor $\G_\T:\She\to\Pre$.
\end{thm}

\begin{proof}
This is a standard result in (Grothendieck) topos theory \cite{Borceux,MM} so we only sketch the proof.
For $\iota\in\Pre$ and $U\in\ob\Pre$ let $\iota\perp U$ denote the situation that every natural transformation $\lambda:\dom\iota\to U$ has a unique extension $\bar\lambda$ along $\iota$, i.e., such that $\bar\lambda\ci\iota=\lambda$. Let $\E$ be the set of arrows in $\Pre$ inverted by the monad $\T$.
Then for a presheaf  $U$ the following conditions are equivalent:
\begin{enumerate}
\item $\iota\perp U$ for all $\iota\in\E$ which is a sieve, i.e., for all $\iota\in\GroT$.
\item $\iota\perp U$ for all $\iota\in\E$ which is monic.
\item $\iota\perp U$ for all $\iota\in\E$.
\item $\nu_U$ is invertible.
\end{enumerate}
Since condition (1) means precisely that $U$ is a $\GroT$-sheaf, the Theorem is proven.
\end{proof}

It is also well-known that left exact idempotent monads on the presheaf category $\Pre$ of a small category $\C$
are in bijection with certain factorization systems on $\Pre$ which in turn are in bijection with
Grothendieck topologies on $\C$. When $\C$ has a monoidal (and additive) structure we may ask for the conditions either on the factorization system or on the Grothendieck topology that correspond to the idempotent monad being
special monoidal in the sense of Definition \ref{def: spec}. Before elevating this to a definition it is worth observing
that the property of being "special" already implies special monoidality.
\begin{lem}
Let $\T=\bra T,\mu,\nu\ket$ be an idempotent monad on $\Pre$ such that $T(U\odot\nu_V)$ and $T(\nu_U\odot V)$ are invertible for all $U,V\in\ob\Pre$. Then there is a unique monoidal structure on the functor $T$ such that
$\T$ is a special monoidal idempotent monad.
\end{lem}
\begin{proof}
Since $\nu_U\odot\nu_V=(\nu_U\odot TV)\ci(U\odot\nu_V)$, it is inverted by $T$. Therefore
every natural transformation from $U\odot V$ to some $TW$ extends uniquely along $\nu_U\odot\nu_V$,
i.e., $\nu_U\odot\nu_V\perp TW$. Therefore the equations (\ref{nu mul}), (\ref{nu uni}) expressing monoidality
of $\nu$ have unique solutions for $T_{U,V}$ and $T_0$. This proves uniqueness and also constructs candidates for the monoidal structure. It remains to prove that the so defined triple $\bra T, T_2,T_0\ket$ is indeed a monoidal functor. Since $\nu_U\odot\nu_V\odot\nu_W$ is also inverted by $T$, the hexagon condition (associativity of $T_2$) follows from the calculations
\begin{align*}
T_{U,V\odot W}\ci(TU\odot T_{V,W})\ci(\nu_U\odot(\nu_V\odot\nu_W))&=
T_{U,V\odot W}\ci(\nu_U\odot\nu_{V\odot W})=\\
&=\nu_{U\odot(V\odot W)}\\
T_{U\odot V,W}\ci(T_{U,V}\odot TW)\ci((\nu_U\odot\nu_V)\odot\nu_W)&=
T_{U\odot V,W}\ci(\nu_{U\odot\nu_V}\odot\nu_W)=\\
&=\nu_{(U\odot V)\odot W}\ .
\end{align*}
The left unitality square follows from that $\YI\odot\nu_V $ is inverted by $T$. Indeed,
\begin{align*}
T\luni_V\ci T_{\YI,V}\ci(T_0\odot TV)\ci (\YI\odot\nu_V)&=T\luni_V\ci T_{\YI,V}\ci(\nu_\YI\odot\nu_V)=\\
&=T\luni_V\ci\nu_{\YI\odot V}=\nu_V\ci \luni_V=\\
&=\luni_{TV}\ci(\YI\odot\nu_V)\,.
\end{align*}
Right unitality can be shown similarly, using invertibility of $T(\nu_U\odot \YI)$.
\end{proof}

\begin{defi}
Let $\GroT$ be a Grothendieck topology on the underlying $\Ab$-category of the small monoidal $\Ab$-category $\C$ and let $\T=\bra T,\mu,\nu\ket$ be its left exact idempotent monad. Then we say that $\GroT$ is monoidal if
$T(U\odot\nu_V)$ and $T(\nu_U\odot V)$ are invertible for all $U,V\in\ob\Pre$. In this case the pair $\bra\C,\GroT\ket$ is called a monoidal site.
\end{defi}

\begin{lem}
Let $\GroT$ be a Grothendieck topology on the small monoidal $\Ab$-category $\C$ and let $\bra\E,\M\ket$
be the associated factorization system. Then $\GroT$ is monoidal if and only if $\E$ is closed under the monoidal product, i.e.,
\[
\alpha,\beta\in\E\quad\Rightarrow\quad\alpha\odot\beta\in\E\,.
\]
\end{lem}
\begin{proof}
Clearly, if $\E$ is closed under monoidal product then $\nu_U\odot V$ and $U\odot\nu_V$ belong to $\E$ whatever presheaves the $U$ and $V$ are since $\E$ contains all the identity arrows. Therefore $\GroT$ is monoidal. Assuming $\GroT$ is monoidal we have for all $U\rarr{\alpha}V$ in $\E$ and for all objects $W$ in $\Pre$ the commutative diagram
\[
\begin{CD}
T(U\odot W)@>T(\alpha\odot W)>>T(V\odot W)\\
@V{T(\nu_U\odot W)}VV @VV{T(\nu_V\odot W)}V\\
T(TU\odot W)@>T(T\alpha\odot W)>>T(TV\odot W)
\end{CD}\]
which contains 3 invertible arrows, hence $T(\alpha\odot W)$ is invertible, too. Similarly, one obtains also  $W\odot\alpha\in\E$. Since $\E$ is closed under composition, this implies that it is closed under monoidal product, too.
\end{proof}

It is easy to see that any flat additive functor $F:\C\to\Ab$ determines a Grothendieck topology $\GroT_F$ by
\begin{align} \label{F-topology}
\GroT_F(C)&:=\{S\text{ sieve on $C$}|\text{$S$ is a jointly $F$-epimorphic family of arrows}\}\\
&=\{S\text{ sieve on $C$}|\forall x\in FC\ \exists s\in S,y\in F(\dom s), Fsy=x\}\,.\notag
\end{align}

\begin{lem}\label{exa: F-topology}
If $F:\C\to\Ab$ is an essentially strong monoidal flat functor then the monoidal Grothendieck topology on $\C$
determined by the idempotent monad $\T$ of Proposition \ref{pro: T} is precisely the $F$-topology $\GroT_F$.
The sheaves for this topology are those presheaves $U$ for which $\etash_U$ is an equalizer in
\begin{equation}\label{diag: etash equalizer}
\parbox{250pt}{
\begin{picture}(200,35)
\put(0,10){$U$}
\put(8,10){$\longerrightarrowtail$} \put(34,18){$\etash_U$}
\put(70,10){$\Gsh\Fsh U$}
\put(112,11){\vector(1,0){60}} \put(120,-3){$\Gsh\Fsh\etash_U$}
\put(112,14){\vector(1,0){60}} \put(116,20){$\etash\Gsh\Fsh U$}
\put(180,10){$\Gsh\Fsh\Gsh\Fsh U $}
\end{picture}
}
\end{equation}
or, equivalently, $\bimeta_U:U\to\bimG\bimF U$ is an equalizer in the analogous diagram. 
\end{lem}
\begin{proof}
A subfunctor $S\rarr{i}YC$ belongs to $\GroT(C)$ iff $Ti=\bimL\bimK i$ is invertible and, by faithfulness of $\bimL$, iff $\bimK i$ is invertible. Since the forgetful functor $\M^\bimQ\to\,_R\M_R$ reflects isomorphisms, this happens precisely when $\bimF i$ is invertible and, by left exactness of $\bimF$, this is the same as $\bimF i$ being epimorphic, i.e., $\longF i=i\am{\C}\longf$ being epimorphic. Composing this arrow with the coequalizer (\ref{tensor}) we obtain that $\longF i$ is epimorphic iff the map
\[
\coprod_{A\in\ob\C}SA\ot \longf A\to \longf C,\quad \sum_i s_i\ot z_i\mapsto \sum_i \longf s_iz_i
\]
is epimorphic. Using flatness of $\longf$ the tensor product $S\am{\C}\longf$ can be computed set-theoretically, by replacing the coproduct of abelian groups with disjoint union, we see that the linear combination can always be chosen to consist of a single term. This proves that $\GroT=\GroT_F$.

$U$ is a $\T$-sheaf iff $\nu_U$ is invertible and by diagram (\ref{diag: nu}) this happens precisely when $\bimeta_U$
is an equalizer of the pair given there. By Lemma \ref{lem: m iso} this is also equivalent to $\nush_U$ being invertible,
i.e., $\etash_U$ being an equalizer.
\end{proof}
Unfortunately we cannot check monoidality of $\GroT$ directly in terms of its sieves; a characterization of monoidality of $\GroT$ without reference to its idempotent monad is still to be investigated.

>From the point of view of Tannaka duality the only interesting topologies are the \textit{subcanonical topologies}. These are the Grothendieck topologies for which every representable presheaf $YC$ is a sheaf. If $\GroT_F$ is subcanonical we shall say simply that $F$ is subcanonical. In the next subsection we shall find conditions for $F$
to be subcanonical.
As an extreme example consider the coarsest Grothendieck topology on $\C$ in which the only covering sieve on $C$ 
is the maximal sieve $YC$. This is obviously subcanonical: Every presheaf is a sheaf. If $\GroT_F$ is the coarsest topology we say that the flat functor $F$ is a \textit{coarse functor}. In Section \ref{sec: duality} 
we describe a wide class of categories $\C$ on which coarse fiber functors exist.

\subsection{The embedding theorem}

In this subsection we would like to find conditions on the fiber functor which ensure that the Yoneda embedding
factors through the monoidal embedding $\G_\T$ of $\T$-sheaves into presheaves. 

\begin{lem} \label{lem: *}
With the notations of Proposition \ref{pro: T} the composite functor $\bimK\G_\T$ is an equivalence of monoidal categories $\She\simeq\M^\bimQ$ and $\Ks\G_\T$ is an equivalence of categories $\She\simeq\M^Q$. Thus $\bimF\G_\T:\She\to\,_R\M_R$ is comonadic, left exact and strong monoidal and $\Fsh\G_\T:\She\to\M_R$ is comonadic and left exact.
\end{lem}
\begin{proof}
It suffices to show that $\bimK\G_\T$ is an equivalence of monoidal categories. This will follow from the fact that
the natural isomorphisms
\begin{align*}
\eps^*&:=\bimtheta\ci\bimK\bimL\bimtheta:\bimK\bimL\bimK\bimL\equiv\bimK\G_\T\F_\T\bimL\to\M^\bimQ\\
\eta^*&:=\F_\T\nu\G_\T\ci\tau^{-1}:\She\to\F_\T\G_\T\F_\T\G_\T\equiv\F_\T\bimL\bimK\G_\T
\end{align*}
provide the counit and unit of a monoidal adjoint equivalence
\begin{equation} \label{eq: *-adj}
\eta^*,\eps^*:\bimK\G_\T\dashv\F_\T\bimL.
\end{equation}
As a matter of fact, both $\eps^*$ and $\eta^*$ are built from vertical and horizontal composites of monoidal natural transformations, hence they are monoidal. Furthermore,
\begin{align*}
\F_\T\bimL\eps^*\ci\eta^*\F_\T\bimL&=\F_\T\bimL\bimtheta\ci\F_T\bimL\bimK\bimL\bimtheta\ci\F_\T\nu\bimL\bimK\bimL\ci\tau^{-1}\F_\T\bimL=\\
&=\F_\T(\bimL\bimtheta\ci\nu\bimL)\ci\F_\T\bimL\bimtheta\ci\tau^{-1}\F_\T\bimL=\F_\T\bimL\bimtheta\ci(\tau\F_\T\bimL)^{-1}=\\
&=\F_\T(\bimL\bimtheta\ci\nu\bimL)=\F_\T\bimL
\end{align*}
and
\begin{align*}
\eps^*\bimK\G_\T\ci\bimK\G_\T\eta^*&=\bimtheta\bimK\G_\T\ci\bimK\bimL\bimtheta\bimK\G_\T\ci\bimK\bimL\bimK\nu\G_\T\ci\bimK\G_\T\tau^{-1}=\\
&=\bimtheta\bimK\G_\T\ci\bimK\bimL(\bimtheta\bimK\ci\bimK\nu)\G_\T\ci\bimK\G_\T\tau^{-1}=
\bimtheta\bimK\G_\T\ci(\bimK\G_\T\tau)^{-1}=\\
&=(\bimtheta\bimK\ci\bimK\nu)\G_\T=\bimK\G_\T\,.
\end{align*}
This proves that (\ref{eq: *-adj}) is a monoidal adjoint equivalence, indeed.
Comonadicity of $\bimF\G_\T=\F^\bimQ\bimK\G_\T$ and $\Fsh\G_\T=\F^\Qsh\Ks\G_\T$ now follows from comonadicity of the canonical forgeful functors $\F^\bimQ$ and $\F^\Qsh$, respectively. Left exactness of both functors follow from right adjointness of $\G_\T$ and left exactness of $\bimF$ and $\Fsh$. Although $\G_\T$ is not strong monoidal, see (\ref{eq: G_T mon}), still the composite $\bimF\G_\T$ is because $\bimF$ inverts every arrow $\nu_U$, $U\in\ob\Pre$.
\end{proof}

\begin{pro} \label{pro: Gamma}
Let $\C$ be a small monoidal $\Ab$-category and $F:\C\to\Ab$ be an essentially strong monoidal flat functor.
Then the following conditions are equivalent:
\begin{enumerate}
\item $F$ is subcanonical, i.e., every representable presheaf on $\C$ is a $\GroT$-sheaf.
\item $\nu_{YC}$ is invertible for all objects $C\in\C$.
\item The Yoneda embedding $Y:\C\to\Pre$ factors through $\G_\T:\She\to\Pre$.
\item $\ksh:\C\to\M^\Qsh$ is fully faithful.
\item $\bimk:\C\to\M^\bimQ$ is fully faithful.
\item $\fsh$ is faithful and $\forall B,C\in\ob\C$ an element $f\in\M_R(\fsh B,\fsh C)$
belongs to the image of $\fsh$ if and only if $\forall x\in\fsh B$ $\exists A\in\ob\C$, $z\in\fsh A$, $s\in\C(A,B)$,
$t\in\C(A,C)$ such that $\fsh s z=x$ and $f\ci\fsh s=\fsh t$.
\item $\bimf$ is faithful and $\forall B,C\in\ob\C$ an element $f\in\,_R\M_R(\bimf B,\bimf C)$
belongs to the image of $\bimf$ if and only if $\forall x\in\bimf B$ $\exists A\in\ob\C$, $z\in\bimf A$,
$s\in\C(A,B)$, $t\in\C(A,C)$ such that $\bimf s z=x$ and $f\ci\bimf s=\bimf t$.
\end{enumerate}
If furthermore we assume that $\C$ is additive and $\fsh C$ is finitely generated for all objects $C\in\C$ then the above conditions are equivalent also to these ones:
\begin{enumerate}
\setcounter{enumi}{7}
\item $\fsh$ is faithful and $\forall B,C\in\ob\C$ an element $f\in\M_R(\fsh B,\fsh C)$
belongs to the image of $\fsh$ if and only if $\exists A\in\ob\C$, $s\in\C(A,B)$, and $t\in\C(A,C)$ such that
$\fsh s$ is epi and $f\ci\fsh s=\fsh t$.
\item $\fsh$ is faithful and for all $A,B,C\in\ob\C$ and for all $s\in\C(A,B)$ such that $\fsh s$ is epi
the square 
\begin{equation}
\begin{CD}
\C(B,C)@>\C(s,C)>>\C(A,C)\\
@V{\fsh_{B,C}}VV @VV{\fsh_{A,C}}V\\
\M_R(\fsh B,\fsh C)@>\M_R(\fsh s,\fsh C)>>\M_R(\fsh A,\fsh C)
\end{CD}
\end{equation}
is a pullback square in $\Ab$.
\item $\bimf$ is faithful and $\forall B,C\in\ob\C$ an element $f\in\,_R\M_R(\bimf B,\bimf C)$ 
belongs to the image of $\bimf$ if and only if $\exists A\in\ob\C$, $s\in\C(A,B)$, and $t\in\C(A,C)$ such that
$\bimf s$ is epi and $f\ci\bimf s=\bimf t$.
\item $\bimf$ is faithful and for all $A,B,C\in\ob\C$ and for all $s\in\C(A,B)$ such that $\bimf s$ is epi
the square
\begin{equation}
\begin{CD}
\C(B,C)@>\C(s,C)>>\C(A,C)\\
@V{\bimf_{B,C}}VV @VV{\bimf_{A,C}}V\\
_R\M_R(\bimf B,\bimf C)@>_R\M_R(\bimf s,\bimf C)>>_R\M_R(\bimf A,\bimf C)
\end{CD}
\end{equation}
is a pullback square in $\Ab$.
\end{enumerate}
\end{pro}

\begin{proof}
$(1)\Leftrightarrow(2)$ follows from Theorem \ref{thm: GroT}.

$(1)\Leftrightarrow(3)$ is obvious. We denote the embedding $\C\to\She$ by $Y_\T$.

$(4)\Leftrightarrow(5)$: $\ksh=\psh\bimk$ with $\psh$ an equivalence of categories (see Proposition
\ref{pro: Ls}).

$(3)\Rightarrow(5)$: By Lemma \ref{lem: *} $\bimk\cong\bimK Y=\bimK\G_\T Y_\T$ is the composite of an equivalence with a fully faithful functor.

$(5)\Rightarrow(2)$: Applying formula (\ref{eq: L}) we write $TYC=\bimL\bimK Y C=\M^\bimQ(\bimk\under,\bimK YC)$ which, due to the isomorphism $\bimK Y\iso \bimk$, can be identified with
 the presheaf $\M^\bimQ(\bimk\under, \bimk C)$. Upon this identification the $\nu_{YC}$ becomes the natural transformation with $B$-component equal to
\begin{equation*}
YCB=\C(B,C)\longrarr{}\M^\bimQ(\bimk B,\bimk C),\qquad t\mapsto \bimk t\,.
\end{equation*}
Therefore $\bimk$ being fully faithfull means that the $(\nu_{YC})_B$ are isomorphisms for all $B,C$.

$(2)\Leftrightarrow(7)$ Extending the previous argument for the inclusion $i:TYC\into\bimG\bimF YC$ we obtain
that $TYCB$ is the subgroup of $_R\M_R(\bimf B,\bimf C)$ the elements $f$ of which satisfy
$\bimF\bimG f\ci\delta_B=\delta_C\ci f$, where $\delta$ has been defined in (\ref{bimk}). Since
\[
\delta_B: x\mapsto 1_B\am{B} x\mapsto \{y\mapsto 1_B\am{B}y\}\am{B}x\mapsto 1_{\bimf B}\am{B}x\,,
\]
 the requirement on $f$ is that
\begin{equation} \label{eq: condition on f}
f\am{B}x=1_{\bimf C}\am{C}fx\quad\forall x\in\bimf B
\end{equation}
as elements of $_R\M_R(\bimf\under,\bimf C)\am{\C}\bimf=\bimF\bimG\bimf C$. Viewing $\bimF\bimG\bimf C$
as the filtered colimit of the functor $(\Elt\bimf)^\op\to\C^\op\rarr{\bimG\bimf C}\Ab$ equation
(\ref{eq: condition on f}) 
means precisely that $\forall x\in\bimf B$ $\exists A\in\ob\C$, $z\in\bimf A$, $s\in\C(A,B)$ and $t\in\C(A,C)$
such that $\bimf s z=x$, $\bimf t z =fx$ and $f\ci\bimf s=\bimf t$. Since the $(\nu_{YC})_B$ maps any $a\in\C(B,C)$
into $\bimf a\in\,_R\M_R(\bimf B,\bimf C)$ which obviously satisfies (\ref{eq: condition on f}), we see that invertibility
of $\nu_{YC}$ for all $C$ is equivalent to (7).

$(2)\Leftrightarrow(6)$:
By Lemma \ref{lem: m iso} the $\nush_{YC}$ is invertible iff $\bimnu_{YC}$ is invertible.
Using theformula for $\Ls$ given in Proposition \ref{pro: Ls} and the isomorphism $\Ks Y\iso\ksh$ we can identify $\tsh YC B$ with $\M^\Qsh(\ksh B,\ksh C)$ and consider $(\nush_{YC})_B$ as the unique factorization of the mapping $\C(B,C)\to\M_R(\fsh B,\fsh C)$, $t\mapsto\fsh t$ through $\M^\Qsh(\ksh B,\ksh C)\into\M_R(\fsh B,\fsh C)$. Then
we proceed as in the proof of $(2)\Leftrightarrow(7)$ by expanding what it means for an $f\in\M_R(\fsh B,\fsh C)$
to be a $\Qsh$-comodule map and arrive to the equivalence of condition (6) with invertibility of the $\nush Y$.

$(6)\Rightarrow(8)$: For fixed $B$ and $C$ let $\{x_i\}$ be a finite set of generators for $\fsh B$. Choose
$A_i,z_i,s_i,t_i$ according to the rule (6) for $x_i$ and then construct a direct sum diagram $A_i\dualpair{q_i}{p_i}A$
and the arrows $s:=\sum_i s_i\ci p_i\in\C(A,B)$, $t:=\sum_i t_i\ci p_i\in\C(A,C)$. Then
\[
f\ci\fsh s=\sum_if\ci\fsh s_i\ci\fsh p_i=\sum_i\fsh t_i\ci\fsh p_i=\fsh t
\]
and for a generic element $x=\sum_ix_i\cdot r_i\in\fsh B$ the $z:=\sum_i\fsh q_i z_i\cdot r_i$ satisfies
\[
\fsh s z=\sum_i\sum_j\fsh s_i\ci\fsh p_i\ci\fsh q_j (z_j\cdot r_j)=\sum_i\fsh s_i\ci z_i\cdot r_i=\sum_i x_i\cdot r_i=x
\]
thus $\fsh s$ is epi.

$(8)\Rightarrow(6)$: Since $\fsh$ is epi, for every $x\in\fsh B$ there is a $z\in \fsh A$ such that $\fsh s z=x$.
Therefore condition (6) is trivially satisfied.

Equivalence of (9), (10), (11) with (8) should now be clear.

\end{proof}

The above Proposition provides two embedding theorems at the same time: It allows for embedding $\C$
into the category of sheaves over $\C$ and embedding $\C$ into the category of comodules over a comonad.
In both cases the fiber functor can be written as the composite of the embedding functor with a comonadic functor.
\begin{thm} \label{thm: embedding}
Let $\C$ be a small monoidal $\Ab$-category and $\bimf:\C\to\,_R\M_R$ a faithful and flat strong monoidal functor
satisfying the following condition:
\begin{quote}
If $B,C\in\ob\C$ and $f\in\,_R\M_R(\bimf B,\bimf C)$ is such that for all $x\in\bimf B$ there are $A\in\ob\C$, $z\in\bimf A$, $s\in\C(A,B)$, $t\in\C(A,C)$ satisfying $\bimf s z=x$ and $f\ci\bimf s=\bimf t$
then there exists $a\in\C(B,C)$ such that $f=\bimf a$.
\end{quote}
Then there exist
\begin{itemize}
\item a special left exact monoidal idempotent monad $\T$ on $\Pre$,
\item a left exact monoidal comonad $\bimQ$ on $_R\M_R$,
\item a fully faithful strong monoidal functor $Y_\T:\C\to\She$
\item and a monoidal category equivalence $\She\to \M^\bimQ$
\end{itemize}
such that $\bimf$ can be written as the composite of monoidal functors
\begin{equation*}
\begin{CD}
\She@>\simeq>>\M^\bimQ\\
@A{Y_\T} AA  @VV{\F^\bimQ}V\\
\C@>\bimf>>_R\M_R
\end{CD}
\end{equation*}
where $\F^\bimQ$ is the canonical forgetful functor.
\end{thm}
\begin{proof}
The comonad $\bimQ$ has been constructed in Corollary \ref{cor: bimQ} and the monad $\T$ in Proposition \ref{pro: T} (3) for any strong monoidal flat functor $\bimf$. So it suffices to construct $Y_\T$. By the equivalence (3)$\Leftrightarrow$(7) of Proposition \ref{pro: Gamma} we can write the Yoneda embedding as $Y=\G_\T Y_\T$ with a uniquely determined $Y_\T:\C\to\She$ which is fully faithful, since both $Y$ and $\G_\T$ are fully faithful. Explicitly,
\[
Y_\T(C\rarr{t}D)=\bra YC,\nu_{YC}^{-1}\ket\rarr{Yt}\bra YD, \nu_{YD}^{-1}\ket.
\]
This functor has a strong monoidal structure
\begin{align*}
(Y_\T)_{C,D}&:\bra T(YC\odot YD),\mu_{YC\odot YD}\ket\longerrarr{\nu_{Y(C\ot D)}^{-1}\ci TY_{C,D}}
\bra Y(C\ot D),\nu_{Y(C\ot D)}^{-1}\ket\\
(Y_\T)_0&:\bra T\YI,\mu_\YI\ket\longrarr{\nu_\YI^{-1}}\bra \YI,\nu_\YI^{-1}\ket
\end{align*}
with which $Y=\G_\T Y_\T$ becomes a factorization of monoidal functors. Inserting this factorization into that of Lemma \ref{lem: bimK} and using the result of Lemma \ref{lem: *} that $\bimK\G_\T$ is a monoidal equivalence we are done.
\end{proof}

\subsection{The sheaf monoid $\G$ and bialgebroids} \label{ss: G}

Given a flat essentially strong monoidal functor $\longf:\C\to\Ab$ we can define a presheaf $\longg:\C^\op\to\Ab$
as the pointwise "dual" of $\fsh$. That is to say, $\longg C:=\M_R(\fsh C,R)$ as an abelian group. Of course,
$\longg C$ inherits $R$-$R$-bimodule structure from the left $R$-module structures given on $\fsh C$ and $R$ but for a while this will be ignored. We want to show that this presheaf is a sheaf with respect to the topology induced
by the idempotent monad $\T$.
\begin{lem} \label{lem: Gsh sheaf}
For a ring $R$, a small $\Ab$-category $\C$ and a flat additive functor $\fsh:\C\to\M_R$ let $\Fsh:\Pre\to\M_R$ be
the left Kan extension of $\fsh$ along $Y:\C\to\Pre$ and $\Gsh:\M_R\to\Pre$ be its right adjoint.
Then for all $R$-module $N$ the presheaf $\Gsh N$ is a sheaf with respect to the Grothendieck topology
induced by the left exact idempotent monad $\Tsh$ of Proposition \ref{pro: T} (4).
Especially, $\longg=\Gsh R$ is a sheaf.
\end{lem}
\begin{proof}
Due to the adjunction $\etash,\epsh:\Fsh\dashv\Gsh$ the $\etash\Gsh$ is a split equalizer of the pair
$\bra \etash\Gsh\Fsh\Gsh N,\Gsh\Fsh\etash\Gsh N\ket$. But the $\ish\Ks\Gsh N$ is also the equalizer of this pair,
hence $\nush\Gsh N$ is an isomorphism by Equation (\ref{nush}).
\end{proof}

In the presence of the monoidal structure we have two left exact idempotent monads: $\Tsh$ and the monoidal $\T$.
Since they are isomorphic by Proposition \ref{pro: T} (5), they have the same category of sheaves, as subcategories of $\Pre$. So $\longg$ is a sheaf in both senses but, of course, we are interested in the monoidal sheaf category
where $\longg$ turns out to be a monoid.


The monoidal structure of $\longg$ can be obtained by transposing that of $\longf$ w.r.t. the canonical pairing
$\bra f,x\ket=f(x)$, $f\in\longg C$, $x\in\longf C$ in the following sense: 
\begin{align}
\label{G2}
\longg_{C,D}&: \longg C\ot \longg D\to\longg(C\ot D),\quad f\ot g\mapsto 
\{z\mapsto \bra g,\bra f,z\oneT\ket\cdot z\twoT\ket\}\\
\label{G0}
\longg_0&:\ZZ\to \longg I,\qquad 1\mapsto 1_L
\end{align}
where we introduced, only for the sake of this formula, $z\oneT\oR z\twoT:=F_{C,D}^{-1}(z)$ for $z\in F(C\ot D)$.
The image of the identity monoid $I$ is isomorphic to the ring $L:=R^\op$ and therefore $\longg$ factors through
a normal monoidal functor $\bimg:\C^\op\to\,_L\M_L$. However, $\longg$ is not essentially strong unless
$FC$ is f.g. projective as right $R$-module.

Since monoidal structures on a presheaf are the same as monoid structures on the object in $\Pre$ by
Corollary \ref{cor: mon pre= mon in Pre}, we see that $\longg$ has a monoid structure in $\Pre$. Explicitly,
\begin{align}
\label{eq: m}
m:\longg\odot\longg\to \longg,&\quad [f,g,t]_{B,C}^A\mapsto Gt\ci G_{B,C}(f\ot g)\\
\label{eq: u}
u:YI\to \longg,&\quad (C\rarr{t} I)\mapsto Gt 1_L
\end{align}
But is $\longg$ a monoid also in the sheaf category $\She$?
\begin{lem} \label{lem: G_T reflects monoids}
The monoidal forgetful functor $\G_\T:\She\to\Pre$ reflects monoids in the following sense: If $\bra V,m,u\ket$ is a
monoid in $\Pre$ such that $\nu_V$ is invertible then there is a monoid $\bra\bra V,\nu_V^{-1}\ket,m',u'\ket$ in
$\She$ such that $m=\G_\T m'\ci(\G_\T)_{\bra V,\nu_V^{-1}\ket,\bra V,\nu_V^{-1}\ket}$ and $u=u'\ci(\G_\T)_0$.
Such a monoid can be given by $m'=\nu_V^{-1}\ci Tm$, $u'=\nu_V^{-1}\ci Tu$.
\end{lem}
\begin{proof}
Taking the monoidal structure of $\She$ and of $\G_\T$ into account, see Proposition \ref{pro: She mon} and Equation (\ref{eq: G_T mon}), the verification of the statement is straightforward.
\end{proof}

Giving a presheaf $\longg$ abstractly as an object of $\Pre$ allows to reconstruct $\longg$ as a functor provided the
Yoneda embedding $Y:\C\to\Pre$ is also given. Indeed, $\longg$ is the functor $C\mapsto \Pre(YC,\longg)$.
Similarly, if $\bra \longg,m,u\ket$ is a monoid in $\Pre$ then the reconstructed functor is monoidal by the mapping
$\alpha\ot\beta\mapsto m\ci(\alpha\ot\beta)$.
The image of the unit monoid (of $\C^\op$) under this monoidal functor is the convolution monoid
$L:=\Pre(YI,\longg)$.
In case of $\longg$ is obtained from the pointwise left dual of a strong monoidal functor $\bimf:\C\to\,_R\M_R$
then the convolution monoid $L$ becomes $R^\op$. This means that the base ring $R$ can be reconstructed from
the monoid in $\Pre$. If $\longg C$ is finitely generated and projective as right $L$-module then so is $\longf C$
as right $R$-module and in this case the whole functor $\longf$ can be recovered from the knowledge of the monoid $\bra\longg,m,u\ket$ in $\Pre$.

The f.g. projective case is interesting all the more because in this case the Eilenberg-Moore category $\M^\bimQ$ of the monoidal comonad $\bimQ$ becomes the comodule category $\M^H$ of a bialgebroid $H$ with underlying $R^\op\ot R$-ring $\bimF\bimg$.

\textbf{Notation:} With the appearence of the bialgebroid $H=G\am{\C}F$ a notational ambiguity arises when tensoring with $H$ since $R$ acts on $H$ in 4 different ways. We shall write $\under\obar{R}H$ if we want
tensoring w.r.t. the $R$-action $r\cdot h=ht_H(r)$ and write $\under\oR H$ if we mean $r\cdot h=s_H(r)h$.
The latter is the monoidal product with $H$ in $\M^H$. If tensoring from the left there is no ambiguity,
so $H\oR\under$ means tensoring w.r.t. $h\cdot r=hs_H(r)$. (For $s_H$, $t_H$ see the next proof.)

\begin{pro} \label{pro: finfun - H}
For a small monoidal $\Ab$-category $\C$ and a flat essentially strong monoidal functor $\longf:\C\to\Ab$ assume
that $\fsh C$ is finitely generated projective for all object $C$ of $\C$. Then
\begin{enumerate}
\item there is a left flat $R$-coring structure $H$ on the abelian group $\longg\am{\C}\longf$ and an
isomorphism $\Qsh N\cong N\obar{R} H$ of comonads on $\M_R$, 
\item the coring in (1) is underlying a right $R$-bialgebroid $H$ such that
the equivalence $\M^\bimQ\simeq\M^H$ induced by (1) and by Proposition \ref{pro: M^Qhat=M^Q}
is a monoidal equivalence. 
\end{enumerate}
\end{pro}
\begin{proof}
$(1)$ Noticing that the finiteness condition on $\fsh$ is equivalent to that
$\Gsh N\cong N\oR \gsh$ naturally in $N\in\M_R$ where $\gsh=\Gsh R$,
the statement
follows immediately: $\Fsh\Gsh N\cong (N\oR G)\am{\C}F\cong N\oR H$ naturally and this induces the following
coring structure on $H$
\begin{align*}
\text{underlying $R$-$R$-bimodule}&\quad \gsh\amo{\C}\fsh\\
\text{comultiplication}&\quad\cop_H(f\am{C}x)=\sum_i(f\am{C}x_C^i)\obar{R}(f_C^i\am{C}x)\\
\text{and counit}&\quad\eps_H(f\am{C}x)=fx
\end{align*}
where $\sum_i x_C^i\oR f_C^i\in \fsh C\oR\gsh C$ is the dual basis for the $R$-module $\fsh C$.
The module $_RH$ being flat is equivalent to that the comonad $\Fsh\Gsh$ is left exact.

$(2)$ Since $\bimG MC=\,_R\M_R(\bimf C,M)$ is the center of the bimodule
$\Gsh\phish MC=\M_R(\fsh C,\phish M)$, we have an isomorphism $\bimQ M\cong (M\xover{R}\bimg)\am{\C}\bimf$ where
$M\xover{R}\bimg C$ denotes the center of the bimodule $M\oR \bimg C$.
The coaction $M\to \bimQ M$ of a $\bimQ$-comodule therefore can be denoted by
$m\mapsto (m^{(0)}\xover{R} m^{(1G)})\am{\C}m^{(1F)}$. Embedded into $M\obar{R} H$ it becomes the coring
comultiplication $m\mapsto m^{(0)}\obar{R} m^{(1)}$ where $m^{(1)}\in H$ is the symbol for $m^{(1G)}\am{\C}m^{(1F)}$.
Computing explicitly what the natural map $\bimQ_{M,N}$ does for two bimodules $M$ and $N$,
\begin{equation*}
\begin{split}
(\bimF\bimG)_{M,N}\left(((m\xover{R}f)\am{B}x)\oR
((n\xover{R}g)\am{C}y)\right)=\\
=\left((m\oR n)\xover{R}\bimg_{B,C}(f\ot g)\right)\am{B\ot C}\bimf_{B,C}(x\oR y)\,,
\end{split}
\end{equation*}
we see that the monoidal product of two $\bimQ$-comodules is precisely the one of $H$-comodules when $H$ is
a bialgebroid with multiplication
\[
(f\am{B}x)(g\am{C}y):=\bimg_{B,C}(f\oL g)\amo{B\ot C}\bimf_{B,C}(x\oR y)\,.
\]
and $R^\op\ot R$-ring structure $t_H\ot s_H:R^\op\ot R\to H$ given by the ring homomorphisms
\begin{align*}
s_H(r)&=\bimg_0(1_L)\am{I}\bimf_0(r)\\
t_H(r)&=\bimg_0(r)\am{I}\bimf_0(1_R)
\end{align*}
called the source and target maps of $H$, respectively. Note that the $R$-$R$-bimodule structure of $H$ as a coring
is obtained by right multiplications with the source and target, $r'\cdot h\cdot r=ht_H(r')s_H(r)$, hence the name
right bialgebroid. The right bialgebroid axioms \cite{Ka-Sz} can now be easily verified. 
The details of the proof are omitted.  Since $H$ is the coend of the functor $\bimf$, its bialgebroid structure
has been already constructed in \cite{Phung} in a slightly different terminology.
\end{proof}

Finally we make a short observation on $\longg$ as a sheaf which, in the equivalent comodule category $\M^H$
is well-known \cite[18.9 (3)]{Brz-Wis}. 
The argument uses the following general fact of sheaves and presheaves: Although both $\Pre$ and $\She$ are complete and cocomplete the inclusion
$\G_\T:\She\to\Pre$ preserves only the limits. If $i\mapsto U_i$ is a sheaf-valued functor with colimit $U$ in $\Pre$ then $TU$ is its colimit in $\She$. 
\begin{lem}
Given a ring $R$ and a flat $\fsh:\C^\op\to\M_R$ such that $\fsh C$ is f.g. projective for $C\in\ob\C$
let $\She\subset\Pre$ be the corresponding subcategory of sheaves.
Then the pointwise dual $\gsh$ of $\fsh$ provides an object of $\She$ such that every sheaf
is subgenerated by $\gsh$.
\end{lem}
\begin{proof}
We have to show that every sheaf $U$ is the subsheaf of a $\longg$-generated sheaf.
Notice that the finiteness condition on $\fsh$ is equivalent to that its Kan extension $\Fsh$ is doubly left dual, i.e., $\Gsh:\M_R\to\Pre$ is both left and right dual. Since $R$ is a generator in $\M_R$, we can choose an epimorphism
$\coprod^\I R\longepi{e}\Fsh U$ and obtain the epimorphism $\Gsh e$ the domain of which is a coproduct of copies of $\longg=\Gsh R$, since $\Gsh$ is right adjoint, and it is a sheaf by Lemma \ref{lem: Gsh sheaf}. Thus we have an epimorphism of presheaves 
\[
T(\coprod^\I G)\cong\coprod^\I G\cong\Gsh(\coprod^\I R)\longepi{\Gsh e}\Gsh\Fsh U
\]
which is the image under $\G_\T$ of an epimorphism in $\She$ with domain a coproduct in $\She$ of copies of the sheaf $\longg$. This proves that the sheaf $\Gsh\Fsh U$ is generated by $\longg$. Finally notice that the sheaf $U$
itself is a subsheaf
\[
U\iso TU\longmono{\ish\Ks U}\Gsh\Fsh U
\]
of this $G$-generated sheaf.
\end{proof}

\section{The representation theorem} \label{sec: rep}

We have seen in the previous Section that certain fiber functors $F:\C\to\Ab$ factor through an embedding
$\bimk:\C\to\M^\bimQ$. Now we investigate the question if the image of this embedding is (equivalent to)
the subcategory of comodules that are finite projective as right $R$-modules.


\subsection{The case of corings}


Recall that an arrow $u$ in a category is called von Neumann regular if there exists an arrow $v$ for which
$u\ci v\ci u=u$. In particular, split epimorphisms and split monomorphisms, as well as idempotents, are von Neumann 
regular.
\begin{lem}\label{lem: von Neumann}
Let $\Qsh$ be a left exact comonad on $\M_R$. Let $\M^\Qsh_\fgp$ denote the full subcategory of the Eilenberg-Moore 
category $\M^\Qsh$ the objects of which are finitely generated and projective as $R$-modules.
Let $\F^\Qsh:\M^\Qsh\to\M_R$ be the canonical forgetful functor. 
If $t\in\M^\Qsh_\fgp$ is an arrow for which $u=\F^\Qsh t$ is von Neumann regular then $t$ has kernels and cokernels 
in $\M^\Qsh_\fgp$ and $\F^\Qsh$ preserves them.
\end{lem}
\begin{proof}
Since $\M_R$ is abelian and $\Qsh$ is left exact, $\M^\Qsh$ is abelian and the forgetful functor $\F^\Qsh$ is exact.
Let $0\to A\rarr{k} B\rarr{t}C\rarr{c}D\to 0$ be an exact sequence in $\M^\Qsh$. Then $\F^\Qsh k$ is a kernel and 
$\F^\Qsh c$ is a cokernel of $u$. But the kernels and
cokernels of von Neumann regular arrows in $\M_R$ are split since they are constructed by splitting the idempotents
$1-v\ci u$ and $1-u\ci v$, respectively. Therefore $\F^\Qsh A$ and $\F^\Qsh D$ are direct summands of finitely
generated projective $R$-modules, so themselves are finitely generated projective. This proves that 
$k,c\in\M^\Qsh_\fgp$. Since $\M^\Qsh_\fgp\subset\M^\Qsh$ is a full subcategory, $k$ is a kernel and $c$ is a cokernel 
of $t$ also in $\M^\Qsh_\fgp$.
\end{proof}

Now we try to reach this situation from the abstract setup of a `fiber functor'. 
At first we must find functors which can guarantie the embedding theorem, without the monoidal structure as yet.

Let $\C$ be a small $\Ab$-category and $R$ a ring. We consider
the following properties for an additive functor $\fsh:\C\to\M_R$.
\begin{enumerate}
\item $\fsh$ is faithful.
\item Every arrow $t:A\to B$ in $\C$ for which $\fsh t$ is split epi has a kernel.
\item $\fsh$ preserves the kernels of arrows of (2).
\item Every arrow $t:A\to B$ in $\C$ for which $\fsh t$ is split mono has a cokernel.
\item $\fsh$ preserves the cokernels of arrows of (4).
\item $\fsh$ reflects isomorphisms.
\end{enumerate}
\begin{lem} \label{lem: f=Fu}
Assume $\fsh:\C\to\M_R$ satisfies (1), (2), (3), (4), (5) and (6). Then for every $f:\fsh B\to\fsh C$ for which there exist
$A$, $s:A\to B$ and $t:A\to C$ such that $\fsh s$ is split epi and $f\ci\fsh s=\fsh t$ there exists a unique
$u:B\to C$ such that $\fsh u=f$.
\end{lem}
\begin{proof}
By (2) there is a kernel $k:K\to A$ of $s$ and by (3) $\fsh k$ is a kernel of $\fsh s$. Since $\fsh s$ is split epi, its kernel
$\fsh k$ is split as well, therefore $k$ has a cokernel $c:A\to B'$ by (4). Since $s\ci k=0$, there is a unique $d:B'\to B$
such that $s=d\ci c$.
By (5) $\fsh c$ is a cokernel of $\fsh k$. But $\fsh s$ is also a cokernel of $\fsh k$ (in the abelian category $\M_R$
$\coker\ker e=e$ for every epi $e$), hence $\fsh d$ is an isomorphism and so is $d$ by (6).
Therefore $s$  is a cokernel of (its kernel) $k$.

Now $t$ satisfies $\fsh t\ci \fsh k=f\ci\fsh s\ci\fsh k=0$ so, by (1), $t\ci k=0$. Therefore $t$ factorizes uniquely through $\coker k=s$, i.e., $t=u\ci s$ with a unique $u$. Since $\fsh s$ is epi, $f=\fsh u$. If $f=\fsh u'$ then 
faithfulness of $\fsh$ implies $u'=u$.
\end{proof}


\begin{pro} \label{pro: rep}
Assume $\fsh:\C\to\M_R$ is faithful, flat, reflects isomorphisms and $\fsh A$ is finitely generated projective for all $A\in\ob\C$. Assume
\begin{enumerate}
\item $\C$ is additive,
\item $\C$ has kernels of arrows $t$ for which $\fsh t$ is epi,
\item $\C$ has cokernels of arrows $t$ for which $\fsh t$ is von Neumann regular and $\fsh$ preserves such cokernels.
\end{enumerate}
Then $\ksh:\C\to\M^\Qsh$ corestricts to an equivalence of categories from $\C$ to the category $\M^\Qsh_\fgp$
of $\Qsh$-comodules that are finitely generated projective as right $R$-modules.
\end{pro}
\begin{proof}
Since flat functors preserve kernels, epimorphisms to projectives are split and split epimorphisms are von Neumann
regular, the conditions of Lemma \ref{lem: f=Fu} are satisfied. Thus the (4)$\Leftrightarrow$(8) part of Proposition \ref{pro: Gamma} implies that $\ksh$ is full and faithful. It remains to show that $\ksh$ is essentially surjective,
at least on the objects of $\M^\Qsh_\fgp$.

Every presheaf is the colimit of representables. Since Proposition \ref{pro: Gamma} applies, all representable
presheaves are $\T$-sheaves. Therefore if $U$ is a sheaf and $\tau_i:YA_i\to U$ is a colimiting cone in $\Pre$ then
$T\tau_i$, which is a colimiting cone in $\She$, is isomorphic to a lift of the original cone $\tau$ to $\She$ and presents the sheaf $U$ as a colimit of representables. Then, by the equivalence $\She\simeq\M^\Qsh$,
every $\Qsh$-comodule $M$ is the colimit of comodules $\ksh A$ with $A\in\ob\C$. Therefore every $M$ has a presentation
\[
\coprod_{j\in J} \ksh B_j\longrarr{\beta}\coprod_{i\in I}\ksh A_i\longrarr{\alpha}M\rarr{}0
\]
in the abelian category $\M^\Qsh$. The forgetful functor $\F^\Qsh:\M^\Qsh\to\M_R$, being left adjoint, preserves colimits therefore preserves the structure of this presentation.
If the underlying $R$-module $\F^\Qsh M$ is projective then $\F^\Qsh \alpha$ splits.
If it is also finitely generated then the splitting map factors through a finite subcoproduct $\coprod_{I_0}\fsh A_i
\into \coprod_{I}\fsh A_i$. Restricting $\alpha$ to the corresponding subcoproduct
$\coprod_{I_0} \ksh A_i \into \coprod_{I} \ksh A_i$ is still an epimorphism since $\F^\Qsh$ is faithful.
For all objects $M\in\M^\Qsh_\fgp$ therefore there is a presentation
\[
\coprod_{j\in J_0} \ksh B_j\longrarr{\beta}\coprod_{i\in I_0}\ksh A_i\longrarr{\alpha}M\rarr{}0
\]
in which $I_0$ is finite. Since splitness of $\F^\Qsh\alpha$ implies that the image of $\F^\Qsh\beta$ is a direct summand of a finitely generated projective module, $\F^\Qsh\beta$ is a split epi onto its image:
There exist $a$, $b$ in $\M_R$ such that $\F^\Qsh\alpha\ci a=1$ and $a\ci\F^\Qsh\alpha + \F^\Qsh\beta\ci b=1$.
Therefore $J_0$ can be chosen finite, too, and $\F^\Qsh\beta\ci b\ci\F^\Qsh\beta=\F^\Qsh\beta$.
Finally, using additivity of $\C$ and fullness of $\ksh$ we can present $M$ as the cokernel
\[
\ksh B\longrarr{\ksh b}\ksh A\longepi{} M
\]
of $\ksh b$ of an arrow $b\in\C$ for which $\F^\Qsh\ksh b=\fsh b$ is von Neumann regular.
By assumption such arrows have cokernels in $\C$ and $\fsh=\F^\Qsh\ksh$ preserves them. Since $\F^\Qsh$
reflects cokernels, as every comonadic functor does, $\ksh$ preserves the cokernels of $\fsh$-von Neumann regular arrows. This proves the existence of an object $C\in\C$ such that $M$ is isomorphic to $\ksh C$.
\end{proof}
By Proposition \ref{pro: finfun - H} (1) we know that the comonad $\Qsh$ of the above Proposition is the comonad of
a left flat coring $H$ over $R$ and therefore we actually proved the equivalence of $\C$ with the category $\M^H_\fgp$ of $H$-comodules that are f.g. projective as right $R$-modules. 

\begin{cor} \label{abelian coring rep thm}
Let $\C$ be a small abelian category and $\fsh:\C\to \M_R$ a faithful exact functor such that $\fsh C$ is f.g. projective $\forall C\in\ob\C$. Then there is a left flat $R$-coring $H$, such that $\M^H$ is generated by its subcategory
$\M^H_\fgp$, and an equivalence of monoidal categories
$\C\simeq  \M^H_\fgp$ through of which $\fsh$ factors as $\C\iso\M^H_\fgp\subset\M^H\to\M_R$.
\end{cor}

\subsection{Fiber functors}

Motivated by Proposition \ref{pro: rep} we can now distinguish a class of monoidal functors that can serve as input for
our Tannaka duality.

\begin{defi} \label{def: fiber}
The data $\bra\C,\longf\ket$ consisting of a small additive monoidal category $\C$ and an essentially strong
monoidal additive functor $\longf:\C\to \Ab$ is called a fiber functor if
in the (up to isomorphism unique) factorization $\C\rarr{\bimf}\,_R\M_R\rarr{\phish}\M_R\to\Ab$ of $\longf$, with
$\bimf$ strong monoidal, the functor $\fsh:=\phish\bimf:\C\to\M_R$
satisfies the following conditions:
\begin{itemize}
\item $\fsh$ is faithful, flat and reflects isomorphisms,
\item $\fsh C$ is f.g. projective for all $C\in\ob\C$,
\item $\C$ has kernels of arrows $t$ for which $\fsh t$ is epi,
\item $\C$ has cokernels of arrows $t$ for which $\fsh t$ is von Neumann regular and $\fsh$ preserves such
cokernels.
\end{itemize}
\end{defi}

The main result of this section is the following Representation Theorem.
\begin{thm} \label{thm: rep}
Let $\bra\C,\longf\ket$ be a fiber functor. Then there exists a ring $R$, a right $R$-bialgebroid $H$ and
an equivalence of monoidal categories $\C\simeq  \M^H_\fgp$ through of which
$\longf$ factors as the composite $\C\iso\M^H_\fgp\subset\M^H\to\Ab$ of monoidal functors.
\end{thm}
\begin{proof}


The ring $R$ is the base ring $\bra \longf I,\longf\luni_I\ci\longf_{I,I},\longf_0\ket$ of the fiber functor.
The bialgebroid $H$ is provided by Proposition \ref{pro: finfun - H}  as a right $R$-bialgebroid structure on $\bimg\am{\C}\bimf$ such that the equivalence $\M^H\simeq\M^\Qsh$ of categories induced by
the comonad isomorphism $\under\oR H\cong \Qsh$ becomes a monoidal equivalence when composed with
$\M^\Qsh\simeq\M^\bimQ$ of Proposition \ref{pro: M^Qhat=M^Q}.
This proves the monoidal factorization $\C\rarr{\bimk}\M^H\to\Ab$ of the fiber functor with $\bimk$ strong monoidal.
Then by Proposition \ref{pro: rep} $\bimk$ is an equivalence onto the full subcategory $\M^H_\fgp\subset\M^H$ of $H$-comodules that are f.g. projective as right $R$-modules.
\end{proof}

\begin{cor} \label{abelian bgd rep thm}
Let $\C$ be a small abelian monoidal category and $\longf:\C\to \Ab$ a faithful exact essentially strong monoidal functor such that $\fsh C$ is f.g. projective $\forall C\in\ob\C$. Then there is a ring $R$, a right $R$-bialgebroid $H$,
such that $_RH$ is flat and every object of $\M^H$ is the colimit of objects from $\M^H_\fgp$, and an equivalence of
monoidal categories $\C\simeq  \M^H_\fgp$ through of which $\longf$ factors as $\C\iso\M^H_\fgp\subset\M^H\to\Ab$.
\end{cor}
\begin{proof}
We have to show only that every faithful exact essentially strong monoidal functor $\longf$ on an abelian monoidal
category is a fiber functor in the sense of Definition \ref{def: fiber}.
Taking the factorization $\longf=\phish\fsh$ and using that $\phish:\M_R\to\Ab$ reflects epis and monos we see
that also $\fsh$ must be faithful and exact. Then $\fsh$ is flat since $\C$ has finite limits and $\fsh$ preserves them,
$\fsh$ reflects isomorphisms since it reflects epis and monos. The remaining properties are obvious since
$\C$ has all finite limits and colimits.
\end{proof}

The detailed Representation Theorem is this.
\begin{thm}
Let $\longf:\C\to\Ab$ be a fiber functor in the sense of Definition \ref{def: fiber}. Then there exist
\begin{itemize}
\item  a ring $R$,
\item a right $R$-bialgebroid $H$ such that
\begin{enumerate}
\item $H$ is flat as a left $R$-module,
\item every right $H$-comodule is generated by right $H$-comodules that are f.g. projective as right
$R$-modules
\end{enumerate}
\item a special left exact monoidal idempotent monad (see Definition \ref{def: spec}) $\T$ on the monoidal
category $\Pre$ of presheaves over $\C$
\item a strong monoidal left exact left adjoint functor $\bimK:\Pre\to\M^H$,
\item a strong monoidal fully faithful functor $Y_\T:\C\to\She$ embedding $\C$ into the category of $\T$-sheaves
\item a monoidal category equivalence $\She\simeq\M^H$ of the category of $\T$-sheaves with the category
of right $H$-comodules
\item and a monoidal category equivalence $\C\simeq\M^H_\fgp$ of $\C$ with the full subcategory of $H$-comodules that are f.g. projective as right $R$-modules
\end{itemize}
such that, together with the fully faithful strong monoidal embeddings $\She\into\Pre$ and $\M^H_\fgp\into\M^H$
and with the comonadic forgetful functor $\F^H$, the diagram
\begin{equation} \label{diag: all}
\parbox{150pt}{
\begin{picture}(150,150)(0,-30)
\put(0,10){$\C$}
\put(0,110){$\Pre$}
\put(37,80){$\She$}
\put(80,40){$\M^H_\fgp$}
\put(120,110){$\M^H$}
\put(112,10){$_R\M_R$}
\put(150,-30){$\Ab$}

\put(4,25){\vector(0,1){75}} \put(-10,60){$Y$}
\put(10,25){\vector(1,2){25}} \put(14,58){$Y_\T$}
\put(52,85){\vector(3,1){60}} \put(70,96){$\simeq$}
\put(17,17){\vector(3,1){60}} \put(47,32){$\simeq$}
\put(90,52){\vector(1,2){25}}
\put(20,113){\vector(1,0){90}} \put(60,116){$\bimK$}
\put(18,13){\vector(1,0){90}} \put(62,16){$\bimf$}
\put(124,100){\vector(0,-1){75}} \put(130,60){$\F^H$}
\put(20,8){\vector(4,-1){120}} \put(80,-5){$\longf$}
\put(30,90){\vector(-1,1){20}}
\put(130,2){\vector(1,-1){20}}
\end{picture}
}
\end{equation}
is commutative in the category of monoidal functors.
\end{thm}

\subsection{The reconstruction theorem}

Note that we use the name `reconstruction theorem' in a sense closer to the etymology of the word than
it is usually done. One has a fiber functor $F$ associated to each $H$ and one has the `Tannaka construction'
of a bialgebroid from $F$. If the composite of these two procedures is the identity up to isomorphism then we can say we have `reconstructed' $H$. 

\begin{pro} \label{pro: reco}
For a ring $R$ and a right $R$-bialgebroid $H$ let $\C\subset\M^H$ be the full subcategory of $H$-comodules that are finitely generated projective as right $R$-modules. Let $\longf:\C\to\Ab$ be the restriction to $\C$ of the
underlying abelian group functor $\M^H\to\Ab$.


If $H$ is flat as a left $R$-module and if $\M^H$ is generated by the subcategory $\C$ then the functor $\longf$
is a fiber functor.
\end{pro}
\begin{proof}
Flatness of $_RH$ means precisely that the comonad $\under\oR H$ is left exact on $\M_R$.
This implies by general arguments that the Eilenberg-Moore category $\M^H$ has finite limits. It also has colimits
since $\M_R$ does therefore $\M^H$ is cocomplete abelian.

Let $\bimf:\C\to\,_RM_R$ be the underlying bimodule functor and let $\fsh:=\phish\bimf:\C\to\M_R$.
Of course, $\fsh$ can be viewed also as the composite of the inclusion $\C\to\M^H$ and of the forgetful
$\F^H:\M^H\to\M_R$.
\begin{trivlist}
\item\textit{$\longf$ is essentially strong} Since the $\bimf$ is strong monoidal, by the very definition of monoidal
product in $\M^H$, the long forgetful functor $\longf$ is essentially strong monoidal.
\item\textit{$\fsh$ is faithful.} The inclusion $\C\into\M^H$ is faithful and the forgetful functor $\M^H\to\M_R$ is faithful. Then so is their composition $\fsh$.
\item\textit{$\C$ is additive.} $\M^H$ is additive, i.e., has binary direct sums and a zero object. This structure is
obviously inherited to $\C$.
\item\textit{$\fsh$ is flat.}
Since $\C$ is additive, the category $\Elt \fsh$ of elements of $\fsh$ has binary products. Thus the first axiom \ref{flat-1} for flatness of $\fsh$ is satisfied.
As for the second axiom consider an arrow $t:B\to C$ in $\C$ and an $y\in\fsh B$
such that $\fsh t y=0$. Since $\M^H$ is abelian, a kernel of $t$ exists in $\M^H$ and since the kernel is generated
by $\C$ we can find an $s:A\to B$ in $\C$ and an element $x\in\fsh A$ such that $\fsh s x=y$ and $t\ci s=0$.
Thus $\fsh$ is flat.
\item\textit{$\fsh$ reflects isomorphisms.}
The inclusion $\C\into\M^H$ is fully faithful therefore it reflects isomorphisms. The forgetful functor $\M^H\to\M_R$
is comonadic therefore, by an application of Beck's Theorem, it reflects isomorphisms. Then their composition $\fsh$
reflects them, either.
\item\textit{$\C$ has kernels of arrows $t$ for which $\fsh t$ is epi.}
Since $\fsh t$ is an epimorphism onto a projective module, it is split, hence von Neuman regular. 
Viewing $t$ as an arrow in $\M^H$ it satisfies that $\F^H t$ is von Neumann regular. Therefore Lemma
\ref{lem: von Neumann} implies that it has a kernel in $\C$.
\item\textit{$\C$ has cokernels of arrows $t$ for which $\fsh t$ is von Neumann regular and $\fsh$ preserves them.}
This also follows using Lemma \ref{lem: von Neumann}.
\end{trivlist}
\end{proof}

The next Theorem
shows that the bialgebroid, together with its base ring,
can be uniquely reconstructed from its long forgetful functor.

\begin{thm} \label{thm: reco}
Let $S$ be a ring and $J$ be a right $S$-bialgebroid such that $_SJ$ is flat and such that $\M^J_\fgp$ generates
$\M^J$. Let $\longf:\M^J_\fgp\to\Ab$ be the forgetful functor which is a fiber functor by Proposition 
\ref{pro: reco}. Then the $R$-bialgebroid $H$ constructed from $\longf$ by Theorem \ref{thm: rep}
is isomorphic to the $S$-bialgebroid $J$.
\end{thm}
\begin{proof}
The monoidal structure of $\longf$ allows to reconstruct the ring $S$ as the monoid $R=\bra\longf I,\longf\luni_I\ci\longf_{I,I},\longf_0\ket$ in $\Ab$ where $I$ is the monoidal unit of $\M^J_\fgp$.
The $J$, as a $J$-comodule, is a yet unknown object of $\M^J$ but we know that it is generated by the objects
$P$ of $\M^J_\fgp$. So we have an epimorphism
\[
\coprod_{i\in\mathcal{X}} P_i\longrarr{\varphi}J\qquad\in\M^J
\]
for some set $\mathcal{X}$ and for some $\mathcal{X}$-indexed family $\{P_i\}$ of objects in $\M^J_\fgp$. 
Now every such mapping factors through the evaluation
\[
\coprod_P\M^J(P,J)\ot P\longrarr{\ev}J\,,\quad \sum_i\alpha_i\ot p_i\mapsto\sum_i\alpha_i(p_i)
\]
in which $\M^J(P,J)\ot P$ is the $J$-comodule with coaction $1\ot\delta_P$ where 
$\delta_P:p\mapsto p^{(0)}\am{S}p^{(1)}$ is the coaction of $P$. Using the familiar isomorphism
\[
\M^J(P,J)\iso\M_S(P,S),\qquad \beta\mapsto\eps_J\ci\beta\,,
\]
with inverse $f\mapsto (f\am{S} J)\ci\delta_P$, and introducing the notation $P^\lef$ for the abelian group $\M_S(P,S)$
we see that the existence of an epimorphism $\varphi$ is equivalent to that the evaluation
\[
\coprod_P P^\lef\ot P\longrarr{\ev}J\,,\quad \sum_i f_i\ot p_i\mapsto \sum_i f_i(p_i^{(0)})\cdot p_i^{(1)}
\]
is an epimorphism in $\M^J$. 
Recognizing that the reconstructed bialgebroid $H$, as an abelian group at least, is the quotient $\longf^\lef\am{\C}\longf$ of $\coprod P^\lef\ot P$, see (\ref{tensor}),
we are going to show that $\ev$ annihilates the kernel of this quotient. Indeed,
if $t:P\to Q$ belongs to $\M^J_\fgp$ then for all $p\in P$, $g\in Q^\lef$
\[
\ev(g\ot t\cdot p)=g((t\cdot p)^{(0)})\cdot (t\cdot p)^{(1)}=g(t\cdot p^{(0)})\cdot p^{(1)}=\ev(g\cdot t\ot p)\,.
\]
Therefore $\ev$ factors uniquely through the quotient mapping $c$ that defines $H=\longf^\lef\am{\C}\longf$,
\begin{equation}\label{diag: n}
\parbox{100pt}{
\begin{picture}(100,60)(0,20)
\put(0,70){$\coprod_P P^\lef\ot P$}
\put(55,73){\vector(1,0){35}} \put(70,76){$c$}
\put(100,70){$H$}
\put(38,64){\vector(2,-1){55}} \put(50,40){$\ev$}
\put(100,20){$J$}
\dashline{3}(104,65)(104,33) \put(104,33){\vector(0,-1){0}} \put(107,50){$n$} 
\end{picture}
}
\end{equation}
where $n$ can be written on rank 1 tensors of $H$ as $n(f\am{P}x)=f(x^{(0)})\cdot x^{(1)}$. In fact $H$ has only rank 1 tensors since $\M^J_\fgp$ is additive, or, because $\longf$ is flat. This will be important in the next argument
showing that $n$ is monic. At first, $n\in\Ab$ is epi since $\ev$ is epi. Now assume $f\am{P}x\in\Ker n$. Then
$f\ot x\in\Ker\ev$, i.e., $x\in P$ belongs to the kernel of the $J$-comodule map 
$\alpha:=(f\am{S}J)\ci\delta_P:P\to J$.
But the kernel of $\alpha$, as a subcomodule of $P$ in $M^J$, is generated by $\M^J_\fgp$ therefore there exists
an object $Q$ in $\M^J_\fgp$ and a $J$-comodule map $\kappa:Q\to P$ such that $\alpha\ci\kappa=0$ and $x$ belongs to the image of $\kappa$. Therefore there is a $y\in Q$ such that
\[
f\ot x=f\ot \kappa y\quad\in\  f\ci\kappa\ot y\ +\ \Ker c\,.
\]
However, $f\ci\kappa=\eps_J\ci\alpha\ci\kappa=0$, thus proving that $f\ot x\in\Ker c$. Thus the epi $n$ is monic,
hence an isomorphism in $\Ab$.
Finally we have to show that $n$ is underlying a map of bialgebroids over the same base ring $R=S$. That is to say,
that $n$ is both a map of $(R^\op\ot R)$-rings and a map of $R$-corings. As for the latter, notice that if we equip
each $P^\lef\ot P$ term with the obvious comatrix coring structure then both $c$ and $\ev$ become coring
homomorphisms. Then $n$ is automatically a coring homomorphism since $c$ is epi in $_R\M_R$, or, what is the
same, in $\Ab$. As for the former, notice that the $(R^\op\ot R)$-ring structure of $J$ is completely encoded
in the structure of the monoidal functor $\bimf$ and its pointwise left dual $\bimg=\bimf^\lef$ and this is exactly
the $(R^\op\ot R)$-ring structure of $H$ we reconstruct. This can be traced in the calculation
\begin{align*}
n(f\am{P}x) n(g\am{Q}y)&=x^{(1)}t_J(f(x^{(0)}))y^{(1)}t_J(g(y^{(0)}))=\\
&=x^{(1)}y^{(1)}t_J(g(s_J(f(x^{(0)}))y^{(0)}))=\\
&=x^{(1)}y^{(1)}t_J(\bimf_{P,Q}^\lef(f\am{R^\op}g)(x^{(0)}\oR y^{(0)}))=\\
&=n(\bimf_{P,Q}^\lef(f\am{R^\op}g)\am{P\ot Q}\bimf_{P,Q}(x\oR y))=\\
&=n((f\am{P}x)(g\am{Q}y))
\end{align*}
which finishes the proof. Even if not each term $P^\lef\ot P$ individually but their coproduct can be given a bialgebroid
structure such that the whole diagram (\ref{diag: n}) lifts to the category of $R$-bialgebroids and their homomorphisms.
\end{proof}

\begin{cor}
If an $R$-bialgebroid $H$ satisfies that $_RH$ is flat and $\M^H_\fgp$ generates all $H$-comodules
then $\M^H$ is equivalent to the category of $\GroT$-sheaves for a subcanonical monoidal Grothendieck topology $\GroT$ on $\M^H_\fgp$.
\end{cor}

\subsection{Deligne's and Hai's Theorems}

By introducing the notion of semi-transi\-tive corings in \cite{Brugi} Brugui\`eres located an important piece of
Deligne's Tannaka duality theory \cite{Deligne} that is applicable also in the non-commutative setting. Semi-transitive
corings are precisely the corings that can be reconstructed from fiber functors on locally finite categories.
Let us recall the definitions.

Let $k$ be a field and $R$ be a $k$-algebra. An $R$-coring $H$ is called \textit{semi-transitive} if the following axioms hold:
\begin{enumerate}
\item Every $H$-comodule that is f.g. as $R$-module is also projective as $R$-module.
\item Every $H$-comodule is the filtered colimit of $H$-comodules that are f.g. as $R$-modules.
\item The category $\M^H_\fgp$ of $H$-comodules that are f.g. as $R$-modules is \textit{locally finite}, i.e.,
\begin{enumerate}
\item each hom-space of $\M^H_\fgp$ is finite dimensional over $k$,
\item and each object has finite composition length, i.e., $\M^H_\fgp$ is both noetherian and artinian.
\end{enumerate}
\end{enumerate}

Deligne's Theorem, as formulated in \cite[Theorem 5.2]{Brugi}, states that on a small locally finite abelian $k$-linear category every faithful exact $k$-linear functor $F:\C\to\M_R$ with values in f.g. projective $R$-modules
factors through the forgetful functor $\M^H\to\M_R$ of a semi-transitive $R$-coring $H$ by a category equivalence
$\C\simeq \M^H_\fgp$.
Vice versa, if $H$ is a semi-transitive $R$-coring then the restriction $\M^H_\fgp\to \M_R$
of the forgetful functor is a faithful exact $k$-linear functor with values in f.g. projective $R$-modules and its
domain is an essentially small locally finite abelian $k$-linear category.

In order to compare with our abelian representation theorem, formulated in Corollary \ref{abelian coring rep thm},
and its converse (in the monoidal setting) in Theorem \ref{thm: reco},
observe that both the input data $\bra\C,\fsh\ket$ and the output data $H$ of Deligne's Theorem is a subclass of
the respective data of Corollary \ref{abelian coring rep thm}. For the output data this is not obvious but Brugui\`eres
proves in \cite[Corollary 5.9]{Brugi} that semi-transitive corings are flat as left $R$-modules. Thus the two theorems
are compatible. It is very plausible that none of them implies the other.

Although Brugui\`eres also discusses them for commutative $R$ the monoidal version of Deligne's Theorem was
completed by Ph\`ung H{\^o} Hai at the time when the notion of bialgebroid was already available.
He proves in \cite[Corollary 2.2.5]{Phung} that on a locally finite abelian $k$-linear monoidal category every faithful
exact strong monoidal functor $\bimf:\C\to\,_R\M_R$ with image in the subcategory
of right dual bimodules induces a monoidal equivalence $\C\simeq\M^H_\fgp$ for a semi-transitive $R$-bialgebroid $H$. Vice versa, the restriction to $\M^H_\fgp$ of the forgetful functor $\M^H\to\,_R\M_R$ of any semi-transitive
$R$-bialgebroid is such a functor.

It should now be clear that the abelian representation theorem, Theorem \ref{abelian bgd rep thm}, is to
Hai's Theorem like Corollary \ref{abelian coring rep thm} is to Deligne's Theorem. It establishes Tannaka duality
for a larger class of fiber functors and for a larger class of bialgebroids.

\subsection{$k$-linear versions}

Let $k$ be a commutative ring, $\M_k$ the category of $k$-modules and $\Phi^k:\M_k\to\Ab$ the underlying abelian group functor.


A $k$-linear category $\C$ is nothing but an $\Ab$-category $\C$ together with a ring homomorphism
$k\ni\lambda\mapsto \{\lambda\cdot C\}_C$ from $k$ to the ring of self natural tranformations of the identity functor.

For any $k$-linear category $\C$ and a $k$-linear functor $\dot F:\C\to\M_k$ we can associate the additive functor $F=\Phi^k\dot F:\C\to\Ab$. This defines the functor (as an additive functor between $\Ab$-categories)
\begin{equation} \label{eq: forget k}
\Psi^k\ :\ k\text{-}\Fun(\C,\M_k)\to\Add(\C,\Ab),\qquad \dot F\mapsto \Phi^k\dot F\,.
\end{equation}
If $F:\C\to\Ab$ is additive then there is a unique $k$-linear functor $\dot F$ such that $F=\Phi^k\dot F$, namely
$\dot FC$ is $FC$ equipped with $k$-action $k\ni\lambda\mapsto F(\lambda\cdot C)\in\End FC$.
Since $\lambda\cdot C$ is natural in $C$, every $Ft$ becomes a $k$-module map $\dot Ft$.
If $\nu:F\to G:\C\to\Ab$ is a natural transformation then $\nu_C\ci F(\lambda\cdot C)=G(\lambda\cdot C)\ci\nu_C$,
$\lambda\in k$, $C\in\ob\C$, therefore there is a unique natural transformation $\dot \nu:\dot F\to\dot G:\C\to\M_k$ such that $\Phi^k\dot \nu=\nu$. In this way we constructed a strict inverse of (\ref{eq: forget k}).

Replacing $\C$ with $\C^\op$ we obtain that the category of $k$-module valued $k$-linear presheaves on $\C$,
$k$-$\Fun(\C^\op,\M_k)$ can be identified with the category $\Pre$ of $\Ab$-valued additive presheaves.

If $\C$ is $k$-linear monoidal then $\Psi^k$ maps monoidal functors $\dot F$ to monoidal functors $F$ using
the monoidal structure of $\Phi^k$. Vice versa, if $F:\C\to\Ab$ is a monoidal functor then the unique $\dot F$ such
that $F=\Phi^k\dot F$ has a unique monoidal structure by unique factorization through the canonical epimorphism $FC\ot FD \epi{} FC\am{k}FD$ and by setting $\dot F_0:k\to \dot FI$ to be the mapping
$\lambda\mapsto F(\lambda\cdot I)\ci F_0 (1)$. Thus $\Psi^k$ induces an isomorphism between categories of
monoidal functors.

However, since $\Phi^k$ is not strong monoidal, a strong monoidal $F$ need not be mapped to strong monoidal
$\dot F$. Fortunately, what we need to preserve by $\Psi^k$ is essential strong monoidality.
The base algebra of $\dot F$, i.e., the image under $\dot F$ of $I$, is a $k$-algebra $\dot R=\bra \dot FI,\dot F\luni_I\ci\dot F_{I,I},\dot F_0\ket$ the underlying ring of which is $R=\bra FI,F\luni_I\ci F_{I,I},F_0\ket$,
the base ring of $F$. Furthermore the $R$-bimodule structure of $FC$ reduces to a diagonal $k$-bimodule under the
ring homomorphism $\dot F_0:k\to R$, so becomes a $\dot R$-bimodule in $\M_k$. As a matter of fact,
\begin{align*}
\dot F_0(\lambda)\lact x&=F\luni_C\ci F_{I,C}\ci (F(\lambda\cdot I)\ot FC)(1_R\ot x)=F(\lambda\cdot C)x=\\
&=F\runi_C\ci F_{C,I}\ci(FC\ot F(\lambda\cdot I))(x\ot 1_R)=\\
&=x\ract\dot F_0(\lambda)\,.
\end{align*}
Therefore the normal factorizations of $F$ and $\dot F$ are related by the diagram
\begin{equation} \label{diag: normal fact k forget}
\begin{CD}
\C@>\hat{\dot F}>>_{\dot R}\M_{\dot R}@>>>\M_k\\
@| @VV{\hat{\Phi}^k}V @VV{\Phi^k}V\\
\C@>\hat F>>_R\M_R@>>>\Ab
\end{CD}
\end{equation}
in which $\hat{\Phi}^k$ is strong monoidal. Therefore $F$ is essentially strong iff $\dot F$ is essentially strong.

If flatness of a $k$-linear functor is defined, like for $\Ab$-functors, as cofilteredness of the category of elements,
which uses only the underlying set of $FC$ in its definition, then it is clear that $F$ is flat iff $\dot F$ is flat.

The left Kan extension along $\C\to\Pre$, as a tensor product over $\C$, is independent of whether we consider $\C$ as $k$-linear or as an $\Ab$-category. More precisely, the following diagram is commutative
\[
\begin{CD}
k\text{-}\Fun(\C^\op,\M_k)@>\under\am{\C}\dot F>>\M_k\\
@V{\simeq}VV @VVV\\
\Add(\C^\op,\Ab)@>>\under\am{\C}F>\Ab
\end{CD}
\]
for all $F\in\Add(\C,\Ab)$. 

If $\F:\Pre\to\Ab$ is the left Kan extension of the flat essential strong monoidal $F=\Phi^k\dot F$ then the strong
monoidal part $\hat\F$ factors as $\hat\Phi^k \hat{\dot\F}$, just as $\hat F$ does in (\ref{diag: normal fact k forget}).
The right adjoint of $\hat\F$ which is the functor
$\hat\G M =\,_R\M_R(\hat F\under,M)$, since each $FC$ is diagonal as $k$-bimodule, factors through the center
functor $\Z_k:M\mapsto \,_k\M_k(k,M)$ which maps $R$-bimodules to $\dot R$-bimodules. Therefore
$\hat\G=\hat{\dot \G} \Z_k$ where $\hat{\dot \G}$ is the right adjoint of $\hat{\dot\F}$. The left exact strong monoidal additive comonad $\hat Q=\hat\F\hat\G$ on $_R\M_R$ is therefore related to the left exact strong monoidal $k$-linear comonad $\hat{\dot Q}=\hat{\dot\F}\hat{\dot\G}$ by the comonad morphism
$\bra\hat\Phi^k,\hat\F\hat{\dot\G}\vartheta\ket$
\[
\hat\Phi^k\hat {\dot Q}\longrarr{\hat\F\hat{\dot\G}\vartheta}\hat Q \hat\Phi^k
\]
where $\vartheta$ is the unit of the adjunction $\hat\Phi^k\dashv\Z_k$ which is an isomorphism due to that
$\Phi^k$ is fully faithful. But the induced functor on the Eilenberg-Moore categories is not only a coreflective
subcategory $(\,_{\dot R}\M_{\dot R})^{\hat{\dot Q}} \into (\,_R\M_R)^{\hat Q}$ but also an isomorphism because
of the equivalences
\[
(\,_{\dot R}\M_{\dot R})^{\hat{\dot Q}} \simeq (\M_{\dot R})^{\dot Q}\simeq (\M_R)^Q\simeq (\,_R\M_R)^{\hat Q}.
\]
Or, putting in another way, an $R$-bimodule which is the underlying bimodule of a $\hat Q$-comodule is automatically $k$-diagonal. Therefore every $\hat Q$-comodule can be obtained by forgetting $k$ in a
$\dot Q$-comodule.

What we have shown is this. If a $k$-linear monoidal category and a $k$-linear fiber functor $\dot F:\C\to\M_k$
is given then the construction of the monoidal comonad and its Eilenberg-Moore category in the $k$-linear setting
leads to the same $\Ab$-category as what we have obtained by forgetting $k$ from the beginning.
The same conclusion can be said about the the construction of the monoidal idempotent monad $\T$
on the presheaf category: The $\Ab$-construction yields automatically a $k$-linear $\T$ provided we start from a
$k$-linear fiber functor.

One can look at $k$-linearity in another way. Every $\Ab$-category and every monoid\-al $\Ab$-category, too, has
a largest commutative ring $\Zee$ for which it is linear and $k$-linearity factors through this one via a ring homomorphism $k\to\Zee$. We formulate the precise statement for monoidal categories below.

Let $\C$ be a monoidal $\Ab$-category.
For endomorphisms $z$ of the monoidal unit $I$ of $\C$ we define $z\cdot C:=
\luni_C\ci(z\ot C)\ci\luni^{-1}_C:C\to C$ which is a endo-natural transformation of the identity functor of $\C$. Similarly, we define $C\cdot z$. Let
\[
\Zee=\Zee_\C=\{z\in\End I|z\cdot C=C\cdot z\ \forall C\in\ob\C\}
\]
which is a subring of the commutative ring $\End I$. Then for all $z\in\Zee$
\begin{align*}
z\cdot D\ci t&=t\ci z\cdot C\qquad \forall t\in\C(C,D)\\
z\cdot (C\ot D)&=z\cdot C\ot D=C\ot z\cdot D\qquad \forall C,D\in\ob\C
\end{align*}
therefore making $\C(C,D)$ a $\Zee$-module by $z\cdot t:=z\cdot D\ci t$ defines a monoidal $\Zee$-linear
category $\bar\C$ the underlying monoidal $\Ab$-category of which is $\C$.

Let $k$ be any commutative ring such that $\C$ is the underlying monoidal $\Ab$-category of a $k$-linear
monoidal category $\B$. Then there is a unique ring homomorphism $\alpha:k\to\Zee$ such that $\B$ is induced by $\alpha$ from the monoidal $\Zee$-linear category $\bar\C$.




\subsection{Weak bialgebras}

Let $R$ be a separable Frobenius algebra over the commutative ring $k$. That is to say, $R$ is a $k$-algebra with a $k$-linear functional $\varphi:R\to k$ and an $e=\sum_i e_i\am{k}f_i\in R\am{k}R$ such that
\[
\sum_i\varphi(re_i)f_i=r=\sum_i e_i\varphi(f_i r),\quad\forall r\in R\quad\text{and}\quad\sum_i e_i f_i=1_R\,.
\]
It follows that $e$ belongs to the center $(R\am{k}R)^R$ of the $R$-bimodule $R\am{k}R$ and
we can equip the forgetful functor $\Phi:\,_R\M_R\to\M_k$ with an opmonoidal structure
\begin{align*}
\Phi^{M,N}&:\Phi(M\oR N)\to\Phi M\am{k}\Phi N,\quad m\oR n\mapsto\sum_i m\cdot e_i\am{k}f_i\cdot n\,,\\
\Phi^0&:R\to k,\qquad r\mapsto\varphi(r)\,.
\end{align*}
This structure is compatible with the usual monoidal structure $\bra\Phi,\Phi_2,\Phi_0\ket$ of the forgetful functor:
For all $R$-bimodules $L$, $M$, $N$
\begin{align*}
(\Phi_{L,M}\am{k} \Phi N)\ci\asso_{\Phi L,\Phi M,\Phi N}\ci(\Phi L\am{k} \Phi^{M,N})
&=\Phi^{L\oR M,N}\ci \Phi\asso_{L,M,N}\ci \Phi_{L,M\am{k} N} \\
(\Phi L\am{k} \Phi_{M,N})\ci\asso^{-1}_{\Phi L,\Phi M,\Phi N}\ci(\Phi^{L,M}\am{k} \Phi N)
&=\Phi^{L,M\oR N}\ci \Phi\asso^{-1}_{L,M,N}\ci \Phi_{L\oR M,N} \\
\Phi_{M,N}\ci \Phi^{M,N}&= \Phi(M\oR N)\,.
\end{align*}
In the terminology of \cite{Sz: Brussels} $\Phi$ is equipped with a \textit{separable Frobenius structure}.
It follows that if $\bimf:\C\to\,_R\M_R$ is any strong monoidal functor then composition of monoidal and
composition of opmonoidal functors equip $F=\Phi \bimf:\C\to\M_k$ with a separable Frobenius structure
$\bra F,F_2,F_0,F^2,F^0\ket$.

We say that a monoidal $k$-linear functor $F:\C\to\M_k$ is \textit{split monoidal} if its monoidal structure
$\bra F,F_2,F_0\ket$ is part of a separable Frobenius structure on $F$. It follows from 
\cite[Lemmas 6.2, 6.3]{Sz: Brussels} that split monoidal $k$-linear functors to $\M_k$ are essentially strong monoidal
with base ring a separable Frobenius $k$-algebra $R$. On the other hand, weak bialgebras over $k$ can be identified
with bialgebroids equipped with a separable Frobenius $k$-algebra structure $\bra \varphi,e\ket$ on its base ring
$R$. Pfeiffer studying comodule categories $\M^W$ of weak bialgebras $W$ shows in \cite{Pfeif} that the long forgetful functor $\M^W\to\M_k$ is split monoidal. Combining these facts with Theorems \ref{thm: rep} and \ref{thm: reco}
we immediately obtain
\begin{thm} Let $k$ be a commutative ring. 
\begin{trivlist}
\item[$(A)$] Let $\C$ be a small additive $k$-linear monoidal category and $F:\C\to\M_k$ a $k$-linear split monoidal functor such that 
\begin{enumerate}
\item $F$ is faithful, flat and reflects isomorphisms,
\item $FC$ is f.g. projective $k$-module $\forall C\in\ob\C$,
\item $\C$ has kernels of arrows $t$ for which $Ft$ is split epi in $\M_k$,
\item $\C$ has cokernels of arrows $t$ for which $Ft$ is von Neumann regular in $\M_k$ and $F$ preserves these cokernels.
\end{enumerate}
Then there is a weak bialgebra $W$ over $k$ such that $F$ factors through the canonical forgetful functor $\M^W\to\M_k$
as $\C\simeq\M^W_\fgp\into\M^W\to\M_k$ with $\C\simeq\M^W_\fgp$ being a monoidal equivalence of $\C$ with the full subcategory of those $W$-comodules that are finitely generated projective $k$-modules.
This weak bialgebra is $k$-flat and such that $\M^W$ is generated by the subcategory $\M^W_\fgp$.
\item[$(B)$] Vice versa, if $J$ is a weak bialgebra over $k$ which is $k$-flat and such that $\M^J_\fgp$ generates $\M^J$ 
then the restriction to $\C=\M^J_\fgp$ of the forgetful functor $\M^J\to\M_k$ is a $k$-linear split monoidal functor
$F:\C\to\M_k$ satisfying conditions (1-4) above. Moreover, the weak bialgebra $W$ constructed in $(A)$ from $F$
is isomorphic to $J$ as weak bialgebras over $k$.
\end{trivlist}
\end{thm}
\begin{proof}
After what has been said above it suffices to show that the $F$ of $(A)$ is a fiber functor in 
the sense of Definition \ref{def: fiber}. Let $R$ be the base ring of $F$ which a separable Frobenius $k$-algebra and
let $\phish:\M_R\to\M_k$ be the forgetful functor. We have $F=\phish\fsh$ and the functor 
$\phish_*:\Fun(\C,\M_R)\to\Fun(\C,\M_k)$ preserves and reflects all properties listed in (1 - 4). For example, $F$ is conservative iff $\fsh$ is conservative, $FC$ is fgp iff $\fsh C$ is fgp, or, $Ft$ is von Neumann regular iff $\fsh t$
is von Neumann regular.
For showing e.g. that von Neumann regularity of $Ft$ implies von Neumann regularity of $\fsh t$ one takes a $\beta\in\M_k(FD,FC)$ satisfying $Ft\ci \beta\ci Ft=Ft$, constructs $\gamma\in \M_R(\fsh D,\fsh C)$ by 
$\phish \gamma (y):=\sum_i \beta(y\cdot e_i)\cdot f_i$ and verifies that $\phish(\fsh t\ci \gamma\ci\fsh t) x=
\sum_i Ft(\beta(Ftx\cdot e_i)\cdot f_i)=\sum_i Ft(\beta(Ft(x\cdot e_i)))\cdot f_i=\sum_iFt(x\cdot e_i)\cdot f_i=
Ftx\cdot\sum_i e_if_i= Ftx=\phish(\fsh t) x$, $\forall x\in FC$, and therefore $\fsh t\ci\gamma\ci\fsh t=\fsh t$.
\end{proof}

The above Theorem greatly simplifies if we assume that $k$ is a field. As a matter of fact, the only obstruction against $\C$ being abelian is that submodules and quotient modules of f.g. projective $k$-modules need not be  projective and, for submodules, not even finitely generated. But for $k$ a field these properties are automatic.

\begin{thm}
Let $k$ be a field. 
\begin{trivlist}
\item[$(A)$] Let $\C$ be a small $k$-linear abelian monoidal category equipped with a $k$-linear faithful exact 
split monoidal functor $F:\C\to\M_k$ such that $FC$ is a finite dimensional vector space for all object $C$ of $\C$.
Then there is a weak bialgebra $W$ over $k$ and a monoidal equivalence $\C\simeq\M^W_\fgp$ such that $F$ factors as
$\C\simeq\M^W_\fgp\into\M^W\to\M_k$. This weak bialgebra is such that $\M^W$ is generated by its finite dimensional comodules.
\item[$(B)$] Let $W$ be a weak bialgebra over $k$ such that $\M^W$ is generated by its finite dimensional comodules.
Then the restriction to $\C=\M^W_\fgp$ of the forgetful functor $\M^W\to\M_k$ is a $k$-linear faithful exact 
split monoidal functor $F:\C\to\M_k$ such that $FC$ is a finite dimensional vector space for all object $C$ of $\C$.
\end{trivlist}
\end{thm}
\begin{proof}
Only right exactness of $F$ in part $(B)$ requires some explanation.
The $W$ being flat over $k$ the category $\M^W$ of $W$-comodules is an abelian category.
Let $t:C\to D$ be an arrow in $\C$. Then $t$ has a cokernel $c:D\to E$ in $\M^W$. Since $FE$ is then a quotient of a finite dimensional $k$-space, it is also finite dimensional. Thus $c$ belongs to $\C$. Then $c$ is also the cokernel of $t$ in $\C$ since the embedding $\M^W_\fgp\into\M^W$ is a fully faithful functor between abelian categories. 
Since the forgetful functor $\F^W:\M^W\to\M_k$ is left adjoint, it preserves cokernels so $Fc=\F^Wc$ is a cokernel
of $Ft=\F^Wt$.
\end{proof}

\section{Existence of fiber functors} 
\label{sec: duality}

In this Section we further restrict the class of monoidal categories $\C$ in order to ensure the existence of fiber
functors and therefore establish an equivalence, by the Representation Theorem, of $\C$ with $\M^H_\fgp$ 
for some $R$-bialgebroid $H$. The class of categories $\C$ for which fiber functors can be constructed depends, of
course, on the Ansatz we take for the fiber functor, which is this. 

For a set $\I\subset \ob\C$ of objects we consider the presheaf $\Omega:=\coprod_{A\in\I} YA$ and the functor
\begin{equation} \label{longF Ansatz}
\longF:\Pre\to\Ab,\quad \longF U:=\Pre(\Omega,\Omega\odot U) 
\end{equation}
equipped with monoidal structure
\begin{align*}
\longF_{U,V}&: x\ot y\mapsto \asso^{-1}_{\Omega,U,V}\ci (x\odot V)\ci y\\
\longF_0&: 1\mapsto \runi^{-1}_\Omega\ .
\end{align*}
We want to impose conditions on the category $\C$ and on the subset $\I$ that imply that the composite monoidal
functor
\begin{equation} \label{longf Ansatz}
\longf:=\longF Y:\C\to\Ab
\end{equation}
is a fiber functor. Defining the ring $R:=\End \Omega$ and making each $\longf C$ an $R$-$R$-bimodule by
$r'\cdot x\cdot r=(r'\odot YC)\ci x\ci r$, $x\in \Pre(\Omega,\Omega\odot YC)$, $r,r'\in R$ we obtain
the factorization $\C\rarr{\bimf}\,_R\M_R\to\Ab$ of $\longf$ with $\bimf$ normal monoidal. As always $\fsh C$ denotes $\bimf C$ with the left $R$-module structure forgotten.

\subsection{Bounded fusion}

In a small monoidal $\Ab$-category $\C$ let $\I\subset\ob\C$ be such that $\{YA|A\in\I\}$ generates $\Pre$.
Such an $\I$ always exists, $\I=\ob\C$ for example. In particular, for every representable presheaf $YC$
there is an epimorphism $\coprod_{j\in\J}YB_j\epi{e} YC$ with $B_j\in\I$, $\forall j\in\J$. Since $YC$ is projective and small, $e$ is split and the splitting map factors through a finite subcoproduct. Thus, if $\C$ is additive, every $C\in\ob\C$ is a direct summand of a finite direct sum $B_1\oplus B_2\oplus\dots\oplus B_n$ of
objects $B_i\in\I$. 

It follows that for each $C\in\ob\C$ and each $A\in\I$ we can choose a direct summand diagram
\begin{equation} \label{diag: p-q}
A\ot C\longrarr{q^i_{A\ot C}}B^i_{A\ot C}\longrarr{p^i_{A\ot C}}A\ot C,\quad i=1,\dots,n_{A\ot C},
\end{equation}
i.e., such that $\sum_i p^i_{A\ot C}\ci q^i_{A\ot C}=A\ot C$, in which all the $B^i_{A\ot C}\in\I$.
Let $M^B_{A,C}$ denote the multiplicity of $B\in I$ in this diagram, i.e.,
\[
M^B_{A,C}:=\mash \{1\leq i\leq n_{A\ot C}|B^i_{A\ot C}=B\}\qquad A,B\in\I,\ C\in\ob\C\,.
\]

\begin{defi}
A set $\I\subset\ob\C$ of objects in a monoidal $\Ab$-category is called a bounded fusion system if
\begin{itemize}
\item $\{YA|A\in\I\}$ is a generator system in the category $\Pre$ of presheaves
\item and there is a choice of direct summand diagrams (\ref{diag: p-q}) the multiplicities $M^B_{A,C}$ of which
satisfy that there exist finite numbers $m_C$ for all $C\in\I$ (equivalently, for all $C\in\ob\C$) such that
\[
\sum_{A\in\I}M^B_{A,C}\ \leq\ m_C\,,\qquad \forall B\in\I\quad\text{and}\quad\forall C\in\I\ (\forall C\in\ob\C)\,.
\]
\end{itemize}
The category $\C$ itself is called to have bounded fusion if there exists a bounded fusion system $\I\subset\ob\C$.
\end{defi}

Let $O^B_{A,C}$, the occurence of $B\in\I$ in $A\ot C$, be defined by $O^B_{A,C}=1$ if $M^B_{A,C}>0$ and $O^B_{A,C}=0$ otherwise. Then the boundedness condition above is equivalent to boundedness of the
set of integers
\[
\{\sum_{A\in\I}O^B_{A,C}|B\in\I\}
\]
for all object $C$. That is to say, the number of $A\in\I$ for which a given $B\in\I$ occurs in the direct summand diagram for $A\ot C$, as one of the $B^i_{A\ot C}$, has a $B$-independent upper bound for all $C$.

Observe also that if $\I$ is a bounded fusion system and $D\in\ob\C\setminus \I$ then adjoining $D$ to $\I$ yields again a bounded fusion system $\I'$. This is because we can keep the old diagrams when expanding $A\ot C$ with $A\in\I$ and the only new diagram is that of $D\ot C$ which can weaken the bound $m_C$ only by a finite amount.
Therefore we can always assume without loss of generality that the unit object $I$ belongs to $\I$ and we do assume this in the sequel without warning.

\begin{lem} \label{lem: fib esmon}
Let $\I\subset\ob\C$ be a bounded fusion system in the monoidal $\Ab$-category $\C$ and define $\Omega=\coprod_{A\in\I}YA$. Then
\[
\longf:\C\to\Ab,\quad \longf C=\Pre(\Omega,\Omega\odot YC)
\]
with monoidal structure induced from that of $\longF$ in (\ref{longF Ansatz}) is essentially strong.
Moreover, $\fsh C$, which is $\longf C$ with right $R=\End\Omega$-module action
$x\cdot r:=x\ci r$, is finitely generated and projective for all $C\in\ob\C$.
\end{lem}
\begin{proof}
First we prove finite projectivity of $\fsh C$ for a fixed $C$ by constructing dual bases. Observe that the arrows in
(\ref{diag: p-q}) are labelled by pairs $\bra A,i\ket\in\I\x \NN$ and the number of such pairs for which $B^i_{A\ot C}$
is a given $B$ is bounded by $m_C$. Therefore we can choose injections
\[
\{\bra A,i\ket|A\in\I,\ i\in\{1,\dots,n_{A\ot C}\},\ B^i_{A\ot C}=B\}\longmono{j_B}\{1,\dots,m_C\}
\]
for each $B\in\I$. Taking disjoint union over $B\in\I$ this implies an injection
\[
\{\bra A,i\ket|A\in\I,\ i=1,\dots,n_{A\ot C}\}\longmono{\iota} \{\bra B,j\ket|B\in\I,\ j=1,\dots,m_C\}\,.
\]
This means that we can relabel the arrows in (\ref{diag: p-q}) by using the middle object $B$ and an integer $j$
having $B$-independent range:
\begin{equation} \label{diag: p-q relabelled}
A\ot C\longrarr{q^j_B}B\longrarr{p^j_B}A\ot C,\qquad \sum_{j=1}^{m_C}\sum_{B\in\I} p^j_B\ci q^j_B=A\ot C
\end{equation}
where we set $p^j_B=0$ and $q^j_B=0$ whenever the pair $\bra B,j\ket$ does not occur as the image
of some $\bra A,i\ket$.
Define matrices $P^j=P^j(C)$, $Q^j=Q^j(C)$ with rows and columns labelled by elements of $\I$ and with entries
$P^j(C)_{A,B}\in\C(B,A\ot C)$, $Q^j(C)_{B,A}\in\C(A\ot C,B)$ by
\[
P^j(C)_{A,B}=\begin{cases}p^j_B&\text{if $\bra B,j\ket=\iota(\bra A,i\ket)$ for some $i$}\\
                                                      0&\text{otherwise}\end{cases}
\]
\[
Q^j(C)_{B,A}=\begin{cases}q^j_B&\text{if $\bra B,j\ket=\iota(\bra A,i\ket)$ for some $i$}\\
                                                      0&\text{otherwise}\end{cases}
\]
The matrices $P^j$ have at most one nonzero entry in each column and the matrices $Q^j$ have at most one nonzero entry in each row. All the $P^j$, $Q^j$ are column finite and row finite. Moreover, for each $\bra B,j\ket$ the composite arrow $P^j_{A,B}\ci Q^j_{B,A'}$ is zero unless $A'=A$ (and only for exactly one $A$).
This leads to the matrix relation
\begin{equation}\label{eq: PQ=1}
\sum_{j=1}^{m_C} P^j(C) Q^j(C) =\mathbf{1}_{\ot C,\,\ot C}
\end{equation}
with the appropriate unit matrix on the right hand side. Note that the direct summand condition
(\ref{diag: p-q relabelled}) correspond only to the diagonal elements of this equation.

The functor $\longf$ can be rewritten up to isomorphism as follows.
\begin{align*}
\longf C&=\Pre(\Omega,\Omega\odot YC)\cong \prod_{B\in\I}\Pre(YB,\Omega\odot YC)\cong\prod_{B\in\I}
(\Omega\odot YC)B\cong\\
&\cong\prod_{B\in\I}\int^D\int^E\Omega D\ot\C(E,C)\ot\C(B,D\ot E)\cong\\
&\cong\prod_{B\in\I}\coprod_{A\in\I}\int^D\int^E \C(D,A)\ot\C(E,C)\ot\C(B,D\ot E)\cong\\
&\cong\prod_{B\in\I}\coprod_{A\in\I} \C(B,A\ot C)
\end{align*}
Thus $\longf C$ can be identified with the abelian group of column finite matrices $x$ with rows and columns labelled by $\I$ and with entries $x_{A,B}\in\C(B,A\ot C)$. In this language the monoidal structure utilizes the fact
that column finite matrices can be multiplied and the result is again column finite:
\[
\longf_{C,D}(x\ot y)_{A',A}=\asso^{-1}_{A',C,D}\ci\sum_B (x_{A',B}\ot D)\ci y_{B,A}\,.
\]
In the same manner the base ring $R$ can identified with the ring of column finite matrices with entries
$r_{A,B}\in\C(B,A)$, $A,B\in\I$ by
\[
R\cong\prod_{A\in\I}\coprod_{B\in\I}\C(B,A)\,.
\]
Defining the dual functor
\[
\longf^\lef:\C^\op\to\Ab,\qquad \longf^\lef C:=\Pre(\Omega\odot YC,\Omega)
\]
with the obvious monoidal structure we see that its base ring is $R^\op$. By the isomorphism
\[
\longf^\lef C\cong \Pre(\coprod_A Y(A\ot C),\Omega)\cong \prod_A\coprod_B\C(A\ot C,B)
\]
the elements of this functor can also be interpreted as column finite matrices.

We can now recognize the matrices $P^j(C)$ constructed above as elements of $\longf C$ and the matrices $Q^j(C)$
as elements of $\longf^\lef C$. Matrix multiplication yields right $R$-module maps $x\mapsto Q^j(C) x$ from $\longf C$ to $R$ that, together with the $P^j(C)$'s form a dual basis because of (\ref{eq: PQ=1}).
This proves finite projectivity of $\fsh C$.

In order to show strongness of the already normal monoidal functor $\bimf$
we have to show invertibility of the maps
\[
\bimf_{C,D}(x\oR y)=(\Omega\odot Y_{C,D})\ci\asso^{-1}_{\Omega,YC,YD}\ci (x\odot YD)\ci y\,.
\]
for all $C,D\in\ob\C$. For that purpose we consider the direct summand diagram
\begin{equation} \label{diag: pi-sigma}
\Omega\odot YC\longrarr{Q^j(C)}\Omega\longrarr{P^j(C)}\Omega\odot YC\,,\quad j=1,\dots, m_C
\end{equation}
corresponding to Equation (\ref{eq: PQ=1}) and verify by direct calculation that
\[
\bimf^{-1}_{C,D}(z)=\sum_j P^j(C) \ \oR\
(Q^j(C)\odot YD)\ci \asso_{\Omega,YC,YD}\ci(\Omega\odot Y^{-1}_{C,D})\ci z
\]
provides the inverse of $\bimf_{C,D}$.
\end{proof}

\begin{lem} \label{lem: F reflects split epis}
Any bounded fusion system $\I$ in the small monoidal $\Ab$-category $\C$ produces a functor $\longf$ by (\ref{longf Ansatz}) such that the associated functor $\fsh:\C\to\M_R$ is faithful, reflects isomorphisms, reflects split epimorphisms and reflects von Neumann regular morphisms.
\end{lem}
\begin{proof}
Recall that by definition of bounded fusion systems every object $C$ has a direct summand diagram
\begin{equation}\label{diag: C as a summand}
C\longrarr{q^i}B^i\longrarr{p^i}C
\end{equation}
with $B^i\in\I$. 

In order to prove faithfulness of $\longf$ suppose that $\longf t=0$ for some $t\in\C(C,D)$.
This means that for all $x:\Omega\to\Omega\odot YC$ the composite $(\Omega\odot Yt)\ci x$ vanishes.
In terms of column finite matrices this is equivalent to $(A\ot t)\ci x_{A,B}=0$ for all $A,B\in\I$.
Setting $A=I$ this implies $t\ci s=0$ for all $s:B\to C$ and for all $B\in\I$. Thus $t=\sum_i t\ci p^i\ci q^i=0$
where $p^i$, $q^i$ are the arrows of (\ref{diag: C as a summand}). Therefore $\longf$ is faithful.

Next consider a $t\in\C(C,D)$ such that $\longf t$ is an isomorphism.
Denoting by $\iota_A:YA\to\Omega$ the coproduct injections and by $\tau_A:\Omega\to YA$ the unique morphism
such that $\tau_A\ci\iota_B$ is the identity if $A=B$ and zero otherwise we obtain
\[
\longf_{A,B}C:=\C(B,A\ot C) \to \longf C, \qquad s\mapsto (\iota_A\odot YC)\ci Y^{-1}_{A,C}\ci Ys\ci \tau_B
\]
a split subfunctor, i.e., a direct summand of $\longf$ for all $A,B\in\I$. The splitting morphism is
\[
\longf C\to \longf_{A,B}C,\qquad x\mapsto \left(Y_{A,C}\ci(\tau_A\odot YC)\ci x\ci \iota_B\right)_B(B)\,.
\]
It follows that the arrow $\longf_{A,B}t$ is
mono since $\longf t$ is mono and $\longf_{A,B} t$ is epi since $\longf t$ is epi.
Thus $\longf_{A,B}t$ is an isomorphism for all $A,B\in\I$. Setting $A=I$
we obtain that $\C(B,t):\C(B,C)\to\C(B,D)$ is an isomorphism for all $B\in\I$ and it is natural in the $B$-argument
considered to run over the full subcategory of $\C$ with object set $\I$. It follows that $t$ is an isomorphism
with inverse
\[
s:=\sum_j\C(B'_j,t)^{-1}(p'_j)\ci q'_j
\]
where $D\rarr{q'_j}B'_j\rarr{p'_j}D$ is a direct summand diagram for $D$ with $B'_j\in\I$. Indeed,
\begin{align*}
t\ci s&=\sum_j\C(B'_j,t)(\C(B'_j,t)^{-1}(p'_j))\ci q'_j=\sum_jp'_j\ci q'_j=D\\
s\ci t&=\sum_i\sum_j\C(B'_j,t)^{-1}(p'_j)\ci q'_j\ci t\ci p^i\ci q^i=\\
&=\sum_i\sum_j\C(B^i,t)^{-1}(p'_j\ci q'_j\ci t\ci p^i)\ci q_i=\\
&=\sum_i p^i\ci q^i=C
\end{align*}
where $p^i$, $q^i$ are the arrows of (\ref{diag: C as a summand}).

Assume that $t:C\to D$ is such that $\fsh t$ is split epi. Then there is a $\varphi:\fsh D\to\fsh C$ satisfying
$\fsh t\ci\varphi=\fsh D$. Representing elements of $\fsh C$, $\fsh D$ by column finite matrices as before
we can write
\begin{align*}
\varphi(y)_{A,B}&=\varphi\left(\sum_j P^j(D)Q^j(D)y\right)_{A,B}=\sum_j\sum_{B'\in\I}\varphi(P^j(D))_{A,B'}
\ci(\underset{\in R}{\underbrace{Q^j(D)y}})_{B',B}=\\
&=\sum_{B'\in\I}\Phi_{A,B'}\ci y_{B',B}
\end{align*}
for some $\Phi_{A,B}\in\C(B\ot D,A\ot C)$. Then the splitting condition is equivalent to the equations
\[
(A\ot t)\ci\Phi_{A,B}=\left\{\begin{matrix}A\ot D&\text{if }A=B\\0&\text{otherwise}\end{matrix}\right.
\]
for $A,B\in\I$. Substituting $A=B=I$ we obtain
\[
t\ci\luni_C\ci\Phi_{I,I}\ci\luni^{-1}_D=D\,,
\]
hence $t$ is split epi in $\C$.

Now assume that $t:C\to D$ is such that $\fsh t$ is von Neumann regular, $\fsh t\ci\varphi\ci\fsh t=\fsh t$
for some $\varphi\in\M_R(\fsh D,\fsh C)$. The $\Phi_{A,B}$ constructed above satisfies
\[
(A\ot t)\ci\Phi_{A,B}\ci(B\ot t)=\left\{\begin{matrix}A\ot t&\text{if }A=B\\0&\text{otherwise}\end{matrix}\right.
\]
which, for $A=B=I$ yields
\[
t\ci\luni_C\ci\Phi_{I,I}\ci\luni^{-1}_D\ci t=t
\]
which proves von Neumann regularity of $t$ in $\C$.
\end{proof}

\subsection{Weak kernels}

In order to ensure flatness of the functor
\begin{equation} \label{Ansatz}
\longf C\cong \prod_{B\in\I}\coprod_{A\in\I} \C(B,A\ot C)
\end{equation}
we need flatness of the functors $\C(B,A\ot\under)$ and some coherence of the category in order to ensure
that product of flat functors will be flat. Therefore we require existence of weak kernels in $\C$.

\begin{defi}
In any $\Ab$-category an arrow $s$ is a weak kernel of the arrow $t$ if 
\begin{enumerate}
\item $t\ci s=0$
\item and if $t\ci q=0$ then there is an arrow $r$ such that $q=s\ci r$.
\end{enumerate}
\end{defi}
Equivalently, $s:B\to C$ is a weak kernel of $t:C\to D$ if $YB\rarr{Ys}YC\rarr{Yt}YD$ is exact in $\Pre$.

\begin{lem} \label{lem: weak kernel}
The following conditions on an arrow $t:C\to D$ in an $\Ab$-category $\C$ are equivalent.
\begin{enumerate}
\item $t$ has a weak kernel.
\item The kernel of $Yt$, as a subfunctor of $YC$, is a principal sieve on $C$.
\item If $V\to YC$ is a kernel of $Yt$ then there is an epimorphism $YB\epi{} V$ from a representable presheaf.
\end{enumerate}
\end{lem}
\begin{proof}
$(1)\Rightarrow(2)$ The equivalence class of kernels of $Yt$ is the sieve $S$ consisting of arrows $q$ to $C$ for
which $t\ci q=0$. If $s$ is a weak kernel of $t$ then all such $q$-s have the form $s\ci r$ for some $r\in\C$.
Thus $S$ is the sieve generated by a single arrow, $s$.

$(2)\Rightarrow(3)$ Let $V\into YC$ be a monic representing the kernel sieve $S\subset YC$ of $Yt$ and let
$S$ be generated by $s:B\to C$. Then there is a $v\in VB$ representing $s$ and the natural transformation
$\sigma:YB\to V$ corresponding to $v$ by the Yoneda Lemma is epi since the composite $YB\to V\iso S$ is epi.

$(3)\Rightarrow (1)$ If $\sigma: YB\epi{} V$ is an epi onto the kernel $V\into YC$ of $Yt$ then every $q:A\to C$
such that $t\ci q=0$ belongs to the image of $YBA\rarr{\sigma_A}VA\into YCA$, hence $q=s\ci r$ for some $r\in YBA$
where $s=\sigma_B(B)$ as an element of $YCB$. Thus $s:B\to C$ is a weak kernel of $t$.
\end{proof}

If $A\ot w$ is a weak kernel of $A\ot t$ whenever $w$ is a weak kernel of $t$ then we say $A\ot\under $ preserves weak kernels.

\begin{lem}
Let $\C$ be an additive monoidal category with weak kernels and $A\in\ob\C$.
Then $\C(B,A\ot\under):\C\to\Ab$ is flat $\forall B\in \ob\C$ iff $A\ot\under:\C\to\C$ preserves weak kernels.
\end{lem}
\begin{proof}
$(\Rightarrow)$ Let $t\in\C(C,D)$ and let $w$ be a weak kernel of $t$. If $(A\ot t)\ci x=0$ for some $x:B\to A\ot C$
then flatness of $\C(B,A\ot\under)$ provides an arow $s$ such that $t\ci s=0$ and a $y$ such that $x=(A\ot s)\ci y$.
>From the first relation we see that $s=w\ci q$ for some $q$ and therefore $x$ factors through $A\ot w$. Thus $A\ot w$ is a weak kernel of $A\ot t$.

$(\Leftarrow)$ If $t\in\C(C,D)$ and $w$ is a weak kernel of $t$ then for all $B$  the image of $\C(B,A\ot w)$ is the whole kernel of $\C(B,A\ot t)$, hence $\C(B,A\ot\under)$ is flat.
\end{proof}

\begin{lem} \label{lem: fib flat}
Let $\C$ be a small additive monoidal category, $\I$ be a set of objects of $\C$ and 
$\Omega:=\coprod_{A\in\I}YA$. If $\C$ has weak kernels and $A\ot\under$ preserves weak kernels for all object $A$ then the functor
\[
\longf:\C\to\Ab,\qquad\longf C:=\Pre(\Omega,\Omega\odot YC)
\]
is flat.
\end{lem}
\begin{proof}
Since $\C$ is additive, it suffices to show that axiom (flat-2) is satisfied.
Let $t\in\C(C,D)$ and $x\in \Ker Ft$. Then $x$ is a column finite matrix of arrows $x_{A,B}\in\C(B,A\ot C)$ such
that $(A\ot t)\ci x_{A,B}=0$. Let $w:E\to C$ be a weak kernel of $t$. Then $A\ot w$ is a weak kernel of $A\ot t$ therefore
there exist $y_{A,B}:B\to A\ot E$ such that $x_{A,B}=(A\ot w)\ci y_{A,B}$. This means that we could find an arrow $w$ such that $t\ci w=0$ and such that $x=Fw y$ for some $y$, hence $F$ is flat.
\end{proof}

Observe that even in the presence of weak kernels the functor $\longF$ of (\ref{longF Ansatz}) need not be the left
Kan extension of $\longf=\longF Y$. What we can show is this. Let $\gamma$ denote the canonical mapping
\begin{alignat*}{2}
V\am{\C}\longf\cong\int^B\Pre(YB,V)&\ot\Pre(\Omega,\Omega\ot YB)&\quad\longrarr{\gamma_V}\quad
&\Pre(\Omega,\Omega\odot V)\\
\sigma&\ot\xi\quad&\mapsto\quad&(\Omega\odot\sigma)\ci\xi
\end{alignat*}
which is defined for all presheaves $V$ and is natural in $V$. It is an epimorphism if $V$ is finitely generated, i.e., if there is an epi $\sigma:YE\epi{}V$. As a matter of fact, in this case $\Omega\odot\sigma$ is epimorph and
$\Omega$, being a coproduct of representables, is projective. Therefore for every $x:\Omega\to\Omega\odot V$
there is a $y:\Omega\to \Omega\odot YE$ such that $x=(\Omega\odot\sigma)\ci y$. Thus $\gamma_V$ is epi. On
the other hand, if $\kappa:V\into YC$ is a sieve then $\kappa\am{\C}\longf$ is monic, since $\longf$ is flat, and
therefore naturality
$\longF\kappa\ci\gamma_V=\gamma_{YC}\ci(\kappa\am{\C}\longf)$ implies that $\gamma_V$ is monic. Thus
$\gamma_V$ is an isomorphism for all principal sieves $V$.

Before proceeding we would like to argue that in a certain sense assuming the existence of weak kernels is necessary.

Insisting to the Ansatz (\ref{Ansatz}) to give a fiber functor let us drop the assumption that weak kernels exist and
imagine that some other conditions on $\C$ ensure that $\longf$ is flat. Any flat functor $\longf$ defines a Grothendieck topology $\GT_\longf$ in which $\GT_\longf(C)$ consists of those sieves $S$ on $C$ which are jointly $\longf$-epimorphic, see (\ref{F-topology}). 
We have seen in Lemma \ref{exa: F-topology} that the Grothendieck topology constructed from a fiber functor is precisely this construction. Since our functor $\longf$ factors through a functor $\fsh$ of finite type, every $\GT_\longf$-covering sieve $S$ will contain an arrow $e$ for which $\longf e$ is an epimorphism. But then $\fsh e$ is a split epimorphism since $\fsh C$ is projective. By Lemma \ref{lem: F reflects split epis}
then also $e$ is split epi, hence the covering sieve $S$ is the maximal sieve. This means that if (\ref{Ansatz})
is flat then the topology $\GT_\longf$ is the coarsest topology in which every presheaf is a sheaf.

On the other hand, if $\GT$ is any subcanonical Grothendieck topology, i.e., such that all representable presheaves
are sheaves, then we may want not only the underlying $R$-module $\fsh C$ of the sheaf $YC$ to be finite but the
sheaf $YC$ itself to obey some finiteness condition. Assume all representable presheaves are noetherian as $\GT$-sheaves. Since the inclusion of sheaves into presheaves is fully faithful, so reflects isomorphisms,
this is equivalent to the statement that ACC holds for the set of all subfunctors $V\into YC$ that are sheaves. 
It is now easy to show, using additivity of $\C$, that this implies that every sieve $V$ which is a sheaf is the $\GT$-closure of a principal sieve. In particular, for every $t\in\C$ the kernel sieve $V_t$ of $t$, being closed, is
such a sieve. Applying this argument to the coarsest topology $\GT_\longf$ we see that the closure operation is trivial, hence the kernel sieves are principal, i.e., weak kernels exist in $\C$.

We do not claim, however, that the existence of weak kernels would imply that all the representable presheaves are
noetherian. For later usage we record what we obtained above:
\begin{cor} \label{cor: F coarsest}
Let $\C$ be a small additive monoidal category with weak kernels, with bounded fusion and such that
$A\ot\under$ preserves weak kernels for all object $A$. Then for any bounded fusion system $\I$ the functor $\longf$ of (\ref{longf Ansatz}), which is an essentially strong monoidal and flat by Lemmas \ref{lem: fib esmon} and
\ref{lem: fib flat}, is such that the monoidal Grothendieck topology associated to $\longf$ via the idempotent monad $\T$ of Proposition \ref{pro: T} is the coarsest Grothendieck topology on $\C$. Thus every presheaf is a $\T$-sheaf.
\end{cor}

\subsection{An Almost Duality Theorem}

The assumptions we made in the previous Subsections are sufficient to prove the following result.

\begin{thm} \label{thm: duality}
Assume the small additive monoidal category $\C$ has weak kernels and bounded fusion and let
$A\ot\under $ preserve weak kernels for all $A\in\ob\C$.
\begin{enumerate}
\item Then there is a ring $R$, a right $R$-bialgebroid $H$ and a monoidal equivalence $\Pre\simeq\M^H$
of the category of presheaves over $\C$ with the category of right $H$-comodules.
\item If $\C$ also has splittings of idempotents then, and only then, $\C$ is equivalent, by restriction of the equivalence in (1),
to the full subcategory $\M^H_\fgp$ of $H$-comodules that are f.g. projective as right $R$-modules.
\end{enumerate}
\begin{equation} \label{diag: duality}
\begin{CD}
\Pre@>\simeq>>\M^H\\
@AYAA @AA{\subset}A\\
\C@>\simeq>>\M^H_\fgp
\end{CD}
\end{equation}
\end{thm}
\begin{proof}
Let $\I$ be a bounded fusion system and $\longf:\C\to\Ab$ the associated monoidal functor (\ref{longf Ansatz}).
Let $R$ be the ring which is the image under $\longf$ of the unit monoid $I$ of $\C$.
Then Lemma \ref{lem: fib flat} implies that $\longf$ is flat and Lemma \ref{lem: fib esmon} implies that it is essentially
strong monoidal with base ring $R$ and with $\fsh:\C\to\M_R$ having image in the subcategory of f.g. projective $R$-modules. 
By Proposition \ref{pro: finfun - H} the category $\She$ of $\T$-modules is equivalent to the category of $H$-comodules
for a left flat right bialgebroid $H$. 
By Corollary \ref{cor: F coarsest}, every presheaf is a sheaf therefore $\She=\Pre$, as subcategories of $\Pre$, thus the comparison functor $\bimK:\Pre\to\M^H$ is a monoidal equivalence. This proves (1).

In Lemma \ref{lem: F reflects split epis} we have shown that $\longf$ is faithful and reflects isomorphisms.
$\C$ has kernels for $\fsh$-split epis since by the same Lemma such arrows are split epis in $\C$ where idempotents
split.
Similarly, $\C$ has cokernels of arrows $t$ for which $\fsh t$ is von Neumann regular. Indeed, by Lemma
\ref{lem: F reflects split epis} such a $t$ itself is von Neumann regular and therefore have cokernels (and kernels)
by splitting the appropriate idempotent. This cokernel is then also split hence $\fsh$ preserves it.
This proves that $\longf$ is a fiber functor in the sense of Definition \ref{def: fiber}.
Therefore, by Theorem \ref{thm: rep},
$\C$ is equivalent to the full subcategory $\M^H_\fgp$ of $\M^H$ for a left flat right bialgebroid $H$ over $R$
in such a way that the rhomboid middle of diagram (\ref{diag: all}) reduces to diagram (\ref{diag: duality}).
Vice versa, if the equivalence (1) restricts to an equivalence $\C\simeq \M^H_\fgp$ then idempotents split in $\C$
since this property is trivially satisfied in $\M^H_\fgp$.
This proves (2).
\end{proof}

\begin{cor}
For $\C$ as in Theorem \ref{thm: duality} (1) there exists a ring $R$ and a left exact, strong monoidal, comonadic functor $\bimF:\Pre\to\,_R\M_R$ with image in bimodules that are f.g. projective as right $R$-modules.
\end{cor}
\begin{proof}
Take $\bimF$ to be the strong part of the left Kan extension of the fiber functor $\longf$ constructed in Theorem
\ref{thm: duality}.
\end{proof}
The comonad associated to this comonadic functor is, of course, that of an $R$-bialgebroid.

If we assume that $\C$ has kernels in addition to finite (co)products then splitting of idempotents follows automatically
and flatness of the functors $\C(B,A\ot\under)$ is equivalent to that $A\ot\under$ preserve finite limits. Therefore
we obtain the following
\begin{cor}
Let the small monoidal $\Ab$-category $\C$ have finite limits and bounded fusion and assume $A\ot\under$
preserves kernels for all $A\in\ob\C$. Then $\C$ is monoidally equivalent to $\M^H_\fgp$ for some coarse bialgebroid $H$.
\end{cor}

\section{Comparison of fiber functors and Hopf algebroids}

Let $F$ and $F'$ be fiber functors in the sense of Definition \ref{def: fiber}. 
Then the construction described in Sections \ref{s: FG}, \ref{s: comp} and \ref{s: T} yields functors
$\bimF$, $\bimG$, $\bimK$, \dots, $\T$ for the first and $\bimF'$, $\bimG'$, $\bimK'$, \dots, $\T'$ for the second functor. Combining the two data sets determine new objects such as the left exact monoidal functors
\[
\bimK\bimL':\M^{H'}\to\M^H,\qquad \bimK'\bimL:\M^H\to\M^{H'}\,,
\]
the bicomodule algebras
\[
A:=G\am{\C}F'\ \in\ ^H\M^{H'}\,,\qquad B:=G'\am{\C}F\ \in\,^{H'}\M^H\,,
\]
the monoidal natural isomorphisms
\[
F\coten{H}A\iso F'\,,\qquad F'\coten{H'}B\iso F\,,
\]
\dots, etc. 

In this section we would like to study these objects under the additional assumption of existence of left and/or right duals in $\C$ and the existence of coarse fiber functors on $\C$.

\subsection{Some general observations}

Let us start on the general level of Theorem \ref{thm: rep} with $\C$ a small additive monoidal category.

\begin{lem} 
Let $F$ and $F'$ be fiber functors on $\C$ with associated bialgebroids $H$ and $H'$. Then
there exists a (monoidal) equivalence $\E:\M^H\to\M^{H'}$ and a (monoidal) natural isomorphism $\E\bimK Y\iso\bimK' Y$ if and only if there is a (monoidal) monad morphism $\T\iso\T'$, implying that a presheaf $U$
is a $\T$-sheaf precisely when it is a $\T'$-sheaf.
\end{lem}
\begin{proof}
Assume $\E$ exists.
Since every object $M$ of $\M^H$ can be written as a  colimit  $\colim_j \bimK Y C_j$ and both $\E$ and $\bimK$ preserve colimits, we immediately get monoidal natural isomorphisms $\E\bimK\cong\bimK'$, $\bimL'\cong \bimL\E_*$, where $\E_*$ is some right adjoint of $\E$, and therefore $\T=\bimL\bimK\iso\bimL'\bimK'=\T'$.

Assuming $\T$ and $\T'$ are isomorphic monoidal monads we obtain isomorphic sheaf categories and therefore a
monoidal equivalence $\M^H\simeq\Pre_\T\simeq\Pre_{\T'}\simeq\M^{H'}$.
\end{proof} 

But $H$ and $H'$ can have isomorphic comodule categories without their Grothen\-dieck topologies being the "same".
As an experimentation consider the following 
\begin{deflem}
Let two Grothendieck topologies be given on the same category $\C$ described by two left exact idempotent monads
$\T$ and $\T'$. We say that the two Grothendieck topologies are equivalent and write $\T\sim\T'$ if $\T\nu'_U$ is invertible for all $\T$-sheaves $U$ and $\T'\nu_{U'}$ is invertible for all $\T'$ sheaves $U'$.
\end{deflem}
\begin{proof}
This relation is evidently symmetric and reflexive. As for transitivity consider three left exact idempotent monads such that $\T\sim\T'$, $\T'\sim\T''$ and let $U$ be a $\T$-sheaf. Naturality of $\nu$, $\nu'$ and $\nu''$ leads 
to the three commutative squares of the diagram
\begin{equation}
\begin{CD}
U@>\nu_U>>\T U@>\T\nu'_U>>\T\T' U@>\T\T'\nu'_U>>\T{\T'}^2U\\
@V{\nu''_U}VV @VV{\T\nu''_U}V @VV{\T\T'\nu''_U}V @VV{\T\T'\nu''_{\T' U}}V\\
\T'' U@>\nu_{\T'' U}>>\T\T'' U@>\T\nu'_{\T'' U}>>\T\T'\T'' U@>\T\T'\T''\nu'_U>>\T\T'\T''\T' U
\end{CD}
\end{equation}
By assumption all the four items in the top-right composite are invertible. Therefore $\nu''_U$ is split monic and
$\T\T'\T''\nu'_U$ is split epi.
Since the left hand side of the naturality relation $\T\nu'_U\ci\nu_U=\nu_{\T' U}\ci\nu'_U$ is invertible, we obtain that $\nu'_U$ is split monic. Therefore $\T\T'\T''\nu'_U$ is an isomorphism and the diagram implies
that $\T\nu'_{\T'' U}$ is a split epipimorphism. Since $\T'' U$ is a $\T''$-sheaf, the left hand side of the naturality
relation $\T''\nu'_{\T'' U}\ci\nu''_{\T'' U}=\nu''_{\T'\T'' U}\ci\nu'_{\T'' U}$ is invertible, implying that
$\nu'_{\T'' U}$ is split monic. But then $\T\nu'_{\T'' U}$ is invertible and the above diagram implies $\T\nu''_U$ is invertible. 

Repeating the same argument with $\T$, $\T''$ interchanged we conclude that $\T''\nu_{U''}$ is also invertible for all $\T''$-sheaves $U''$. Therefore $\T\sim\T''$.
\end{proof}


\begin{lem}
If two fiber functors $F$ and $F'$ on $\C$ have equivalent (monoidal) Grothendieck topologies, $\T\sim\T'$,
then the functors $\bimK\bimL'$, $\bimK'\bimL$ establish a monoidal adjoint equivalence $\M^H\simeq\M^{H'}$.
\end{lem}
\begin{proof}
$\T=\bimL\bimK$ inverts an arrow iff $\bimK$ does and every presheaf in the image of $\bimL$ is a $\T$-sheaf. Applying these two observations for $\T$ and $\T'$ one can easily verify that the composites
\[
\M^Q\longrarr{\theta^{-1}}\bimK\bimL\longrarr{\bimK\nu'\bimL}\bimK\bimL'\,\bimK'\bimL
\]
and
\[
\bimK'\bimL\,\bimK\bimL'\longrarr{\bimK'\nu^{-1}\bimL'}\bimK'\bimL'\longrarr{\theta'}\M^{H'}
\]
are isomorphisms and provide the unit and counit of an adjunction.
\end{proof}

\begin{cor}
If $F$ and $F'$ are coarse fiber functors then their bialgebroids $H$ and $H'$ are monoidally Morita-Takeuchi equivalent.
\end{cor}

\subsection{$G$ as a Galois object and right autonomy}

For any essentially strong monoidal and flat functor $F:\C\to\Ab$ we can define the monoidal (pre)sheaf $G=\Ssh\Ipre\in\Pre$ (Subsection \ref{ss: G}) where $\Ssh=\Gsh\Fsh$ and the following natural transformation
\[
\gamma:=\left(G\odot G\longrarr{G\odot\bimeta}G\odot\bimS G\longrarr{\Ssh_{\Ipre,G}}\Ssh(\Ipre\odot G)
\longrarr{\sim}\Ssh G\right)
\]
Looking at $\Ssh$ as a semicomonad on $\Pre$ with comultiplication $\etash\Ssh:\Ssh\to\Ssh\Ssh$ and defining comodules of $\Ssh$ as arrows $V\rarr\alpha \Ssh V$ for which 
\[
V\longrarr{\alpha}\Ssh V\pair{\etash\Ssh V}{\Ssh\alpha}\Ssh\Ssh V
\]
is an equalizer we see that $G\rarr{\etash_G}\Ssh G$, just like every $\T$-sheaf, is an $\Ssh$-comodule.
Moreover, $\Ssh$ has a monoidal counterpart $\bimS=\bra \bimG\bimF,\bimeta\bimG\bimF\ket$ which is related to $\Ssh$ as $\bimQ$ is to $\Qsh$. Now the $\Ssh$-comodule $\etash$ is underlying the $\bimS$-comodule
$\bimeta$ and this latter is a monoidal natural transformation. This means precisely that the coaction $G\rarr{\bimeta}\bimS G$ is compatible with the monoid structure $\bra G,m,u\ket$ of $G$. This suggests that
we should consider $\gamma$ as the Galois map of the $\bimS$-comodule monoid $G$ in $\Pre$. 

Assuming finite projectivity of the right $R$-modules $\fsh C$ we are dealing with a bialgebroid $H$ and
every strong monoidal $\bimF'$ applied to $\gamma$ gives rise to a Galois mapping for the left 
$H=\bimF G$-comodule algebra $A:=\bimF' G$ 
\[
\gamma_A:=\left(A\am{R'} A\longrarr{\bimF'_{G,G}} \bimF'(G\odot G)\longrarr{\bimF'\gamma}\bimF'\Gsh\Fsh G\cong H\obar{R}A\right)
\]
The monoidal category $\C$ is called right autonomous if every $C\in\ob\C$ has a right dual object $^*C$  with
evaluation $\tilde{\ev}_C: C\ot\,^*C\to I$ and dual basis $\tilde{\db}_C:I\to\,^*C\ot C$. 
\Phung proves in  \cite[Theorem 2.2.4]{Phung}, among others, that for right autonomous $\C$ the Galois mapping
$\gamma_H:H\oR H\to H\obar{R}H$ is invertible. 
In fact right autonomy implies somewhat more as we are going to show next.
\begin{pro} \label{pro: A H Gal}
Let $\bra\C,F\ket$ be a fiber functor and assume that $\C$ is right autonomous. Then $\gamma:G\odot G\to 
H\obar{R}G$ is an isomorphism and for all fiber functor $F'$ the left $H$-comodule algebra $A=G\am{\C}F'$ 
is a Galois extension of the base ring $R'$ of $F'$ via the canonical $s_A:R'\to A$, $r'\mapsto 1_L\am{I}r'$.
\end{pro}
\begin{proof}
At first we show that the $H$-coinvariant subalgebra is indeed $s_A$ and then construct an inverse for $\gamma$
which will  then will imply invertibility of $\gamma_A$ for all $F'$.

The left $H$-coaction on $A$ is the mapping $\F'\etash_G:A\to \F'\Gsh\Fsh G\cong H\obar{R}A$, $g\am{C}x'\mapsto \sum_i (g\am{C}x_C^i)\obar{R}(f_C^i\am{C}x')$ where $\sum_i x_C^i\oR f_C^i\in FC\oR GC$
denotes the dual basis associated to the fact that $GC$ is the left dual bimodule of $FC$. The object of coinvariants
$A^{\co H}$ is defined as the coequalizer of $\F'\etash_G$ and the trivial coaction $\F'\Gsh\Fsh\etash_\YI$ which sends $a$ to $1_H\obar{R} a$, up to the isomorphism $\F'\Gsh\Fsh G\cong H\obar{R}A$. But the equalizer 
\begin{equation} \label{eq: A co H}
A^{\co H}\longmono{}A=\F' G\pair{\F'\etash_G}{\F'\Gsh\Fsh\etash_\YI} \F'\Gsh\Fsh G
\end{equation}
is nothing but the image under $\F'$ of the equalizer defining $\Ls\Ks\YI\cong\T\YI$,
\begin{equation} \label{eq: G co H}
\Ls\Ks\YI\longmono{\ish_{\Ks\YI}}\Gsh\Fsh\YI\pair{}{}\Gsh\Fsh\Gsh\Fsh\YI\,.
\end{equation}
Since $F$ is subcanonical, $\YI$ is a $\T$-sheaf and we obtain that $A^{\co H}\rightarrowtail A$ can be obtained
as the composite
\[
R'\rarr{{N'}^{-1}_I}\F'\YI\rarr{\nush_\YI}\F'\Ls\Ks\YI\rarr{\ish_{\Ks\YI}}\F' G=A
\]
which, by (\ref{nush}), is $\F'\etash_\YI\ci {N'}^{-1}_I:r'\mapsto I\am{I}r'\mapsto 1_L\am{I}r'=s_A(r')$.
This finishes the proof of that $s_A:R\to A$ is the coinvariant subalgebra.

The value of $\gamma$ on the generators (\ref{Z-generators of conv.prod.}) is
\[
\gamma([f,g,t]^C_{A,B})=\sum_i(g\am{B}x_B^i)\ \obar{R}\ Gt\ci G_{A,B}(f\am{L} f_B^i)\,.
\]
In order to write up the inverse we choose right duality data $\bra C,\,^*C,\tilde{\ev}_C,\tilde{\db}_C\ket$
in $\C$ and left duality data $\bra FC,GC\equiv(FC)^*,\ev_{FC},\db_{FC}\ket$ in $_R\M_R$ for all object $C$. Notice that $\sum_i x_C^i\oR f_C^i=\db_{FC}(1_R)$. We also use the isomorphism $v_C:FC\iso(F\,^*C)^*$ given by
\[
FC\longerrarr{FC\oR\db_{F\,^*C}}FC\oR F\,^*C\oR (F\,^*C)^*\longerrarr{\tilde{\ev}_{FC}\oR(F\,^*C)^*}(F\,^*C)^*\,,
\]
up to coherence isomorphisms,
where $\bra FC,F\,^*C,\tilde{\ev}_{FC},\tilde{\db}_{FC}\ket$ are right duality data in $_R\M_R$ induced by
the chosen right duality data in $\C$. Explicitly,
\begin{align}
\label{eq: evtilde FC}
\tilde{\ev}_{FC}&=FC\oR F\,^*C\longrarr{F_{C,\,^*C}}F(C\ot\,^*C)\longrarr{F \tilde{\ev}_C}FI\rarr{=} R\\
\label{eq: dbtilde FC}
\tilde{\db}_{FC}&=R\rarr{=} FI\longrarr{F\tilde{\db}_C}F(\,^*C\ot C)\longrarr{F_{\,^*C,C}^{-1}}F\,^*C\oR FC\,.
\end{align}
We claim that the inverse of $\gamma$ is
\[
\gamma^{-1}_C((f\am{A}x)\oR g)=\left[ G_{C,\,^*A}(g\oL v_A(x)),\,f,\,
\asso_{C,\,^*A,A}\ci(C\ot\tilde{\db}_A)\ci\runi^{-1}_C\right]^C_{C\ot\,^*A,A}
\]
As a matter of fact,
\begin{align*}
&\gamma_C\gamma_C^{-1}((f\am{A}x)\oR g)=\\
&={\sst\sum_i(f\am{A}x_A^i)\obar{R}G\runi^{-1}_C\ci G(C\ot\tilde{\db}_A)\ci G\asso_{C,\,^*A,A}\ci G_{C\ot\,^*A,A}
\ci(G_{C,\,^*A}\oL GA)\ ((g\oL v_A(x))\oL f_A^i)}\\
&={\sst\sum_i(f\am{A}x_A^i)\obar{R}G\runi^{-1}_C\ci G(C\ot\tilde{\db}_A)\ci G_{C,\,^*A\ot A}\ci(GC\oL G_{^*A,A})
\ (g\oL(v_A(x)\oL f_A^i)}\\
&=\sst{\sum_i(f\am{A}x_A^i)\obar{R}G\runi^{-1}_C\ci G_{C,I}\ (g\oL\underset{\bra f_A^i,x\ket}
{\underbrace{\sst G\tilde{\db}_C\ci G_{^*A,A}(v_A(x)\oL f_A^i)}})}\\
&=(f\am{A} x)\obar{R} g
\end{align*}
and
\begin{align*}
&\gamma^{-1}_C\gamma_C([f,g,t]^C_{A,B})=\\
&={\sst\sum_i\left[ G_{C,\,^*B}(Gt\ci G_{A,B}(f\oL f_B^i)\oL v_B(x_B^i)),\,g,\,
\asso_{C,\,^*B,B}\ci(C\ot\tilde{\db}_B)\ci\ci\runi^{-1}_C\right]^C_{C\ot\,^*B,B}}\\
&={\sst\sum_i\left[G_{A\ot B,\,^*B}(G_{A,B}(f\oL f_B^i)\oL v_B(x_B^i)),\,g,\,\asso_{A\ot B,\,^*B,B}\ci
((A\ot B)\ot\tilde{\db}_B)\ci\runi^{-1}_{A\ot B}\ci t\right]^C_{(A\ot B)\ot\,^*B,B}}\\
&=\left[{\sst G_{A,B\ot\,^*B}(f\oL }\right.\underset{\sst G\tilde{\ev}_B(1_L)}
{\underbrace{\sst\sum_i G_{B,\,^*B}(f_B^i\oL v_B(x_B^i))}}\left.{\sst ),\,g,\,
\asso^{-1}_{A,B\ot\,^*B,B}\ci\asso_{A\ot B,\,^*B,B}\ci((A\ot B)\ot\tilde{\db}_B)\ci\runi^{-1}_{A\ot B}\ci t}
\right]{\sst ^C_{A\ot I,B}}\\
&={\sst\left[G_{A,I}(f\oL 1_L),\,g,\,\asso_{A,I,B}\ci(A\ot(\tilde{\ev}_B\ot B)\ci\asso_{B,\,^*B,B}\ci(B\ot\tilde{\db}_B)\ci\runi^{-1}_B)\ci t\right]^C_{A\ot I,B}}\\
&=[f,g,t]^C_{A,B}\,.
\end{align*}
\end{proof}



\subsection{Left autonomy}

Assume that the small additive monoidal category $\C$ is left autonomous, i.e., every object $C$ has a left dual
$C^*$ with evaluation morphisms $\ev_C:C^*\ot C\to I$ and coevaluation or `dual basis' $\db_C:I\to C\ot C^*$
satisfying the usual rigidity or adjunction relations. We denote by $(\under)^*$ the corresponding left dual object
functor $\C\to\C^{\op,\rev}$ which is fully faithful but need not be an equivalence.

Since strong monoidal functors preserve left dual objects, the obvious advantage of existence of left duals
in $\C$ is that for any strong monoidal functor $F:\C\to\,_R\M_R$ the bimodules $FC$ are f.g. projective as 
right $R$-modules. Therefore the finiteness property of fiber functors in Definition \ref{def: fiber} holds automatically.

The following result is a well-known generalization of Saavedra's \cite[Proposition 5.2.3]{Saavedra-Rivano}.
\begin{lem} \label{lem: Saavedra}
Let $\M$ be a monoidal category and $\C$ a left (or right) autonomous monoidal category. If $F,G:\C\to\M$ are strong monoidal functors and $u:F\to G$ is a monoidal natural transformation then $u$ is invertible.
\end{lem}
\begin{proof}
Although $\M$ is not left autonomous the objects $FC$, $C\in\ob\C$ have left duals namely $FC^*$ with
\begin{align}
\ev_{FC}&:FC^*\ot FC\rarr{F_{C^*,C}}F(C^*\ot C)\rarr{F\ev_C}FI\rarr{=} R\label{eq: ev FC}\\
\db_{FC}&:R\rarr{=} FI\rarr{F\db_C}F(C\ot C^*)\rarr{F_{C,C^*}^{-1}}FC\ot FC^*\label{eq: db FC}
\end{align}
where $R$ is the unit of $\M$. Similarly, we define $\ev_{GC}$ and $\db_{GC}$ for $C\in\ob\C$.
Let $v:G\to F$ be defined (with coherence isomorphisms suppressed) by
\[
v_C:=(FC\ot\ev_{GC})\ci(FC\ot u_{C^*}\ot GC)\ci(\db_{FC}\ot GC).
\]
Then using monoidality of $u$ it is easy to show that $u_C\ci v_C=GC$ and $v_C\ci u_C=FC$.
\end{proof}

Therefore if $F,F'$ are fiber functors with isomorphic base rings, $R\cong R'$, and $\C$ is left autonomous then
any monoidal natural transformation $F\to F'$ is an isomorphism. Unfortunately, this conclusion fails for
arbitrary $R$ and $R'$ and does not say anything about $F$ and $F'$ if they are not connected by a monoidal
natural transformation. We can instead consider for arbitrary fiber functors $F$ and $F'$ the $H$-$H'$-bicomodule algebra
\[
A:=G\am{\C}F\ \in\ ^H\M^{H'}\,.
\]
%
In order to study the properties of $A$ we need to know more about $G$ when $C$ has left duals.
\begin{pro} \label{pro: G fiber}
Let $\bra\C,F\ket$ be a fiber functor. If $\C$ is left autonomous then the pointwise left dual $GC=(FC)^*\cong FC^*$ is a fiber functor $G$ on $\C^\op$. 

If $\C$ is also right autonomous then $F$ is coarse on $\C$ iff $G$ is coarse on $\C^\op$.
\end{pro}
\begin{proof}
We have already seen in Subsection \ref{ss: G} that $G$ is a strong monoidal functor $\C^\op\to\,_L\M_L\equiv\,_R\M_R^\rev$. 

\textit{$G$ is flat}: Axiom (flat-1) for flatness of $G:\C^\op\to\Ab$ is automatically true since $\C^\op$ is additive.
Axiom (flat-2) requires 
\begin{quote}
$\forall$ $t\in\C(A,B)$ and $\forall$ $g\in\M_R(FB,R)$ such that $g\ci Ft=0$\\
$\exists$ $s\in\C(B,C)$ and $\exists$ $f\in\M_R(FC,R)$ such that $s\ci t=0$ and $f\ci Fs=g$.
\end{quote}
Using that $FC^*$ is a left dual bimodule of $FC$, so $GC\cong FC^*$,  we can reformulate (flat-2) for $G$ as follows:
\begin{quote}
$\forall$ $t\in\C(B^*,A^*)$ and $\forall$ $g\in FB^*$ such that $Ftg=0$\\
$\exists$ $s\in \C(C^*,B^*)$ and $\exists$ $f\in FC^*$ such that $t\ci s$=0 and $Fsf=g$.
\end{quote}
Now this is clearly a consequence of flatness of $F$.

\textit{$G$ is faithful and reflects isomorphisms}: Since $G\cong F(\under)^*$ and $(\under)^*$ is fully faithful,
the statement follows from the respective properties of $F$.

\textit{If $G^\pisharp t$ is epi then $t$ has a kernel in $\C^\op$}: 
Let $t:A\to B$ and $\sigma:G^\pisharp A\to G^\pisharp B\in\M_L\equiv\,_R\M$ be such that 
$G^\pisharp t\ci \sigma=G^\pisharp A$. Construct $^*\sigma:\fsh B\to\fsh A$ by $^*\sigma(y):=\sum_i x_A^i\cdot\bra \sigma(f_A^i),y\ket$ which is a right $R$-module map and obeys $^*\sigma\ci\fsh t=\fsh A$,
i.e., $\fsh t$ is split monic. But then $\fsh t$ is von Neumann regular and therefore $t$ has a cokernel in $\C$. 
But this means precisely that $t$ has a kernel in $\C^\op$.

\textit{If $G^\pisharp t$ is von Neumann regular then $t$ has a cokernel in $\C^\op$ and $G^\pisharp$ 
preserves it}: 
Let $t:A\to B$ and $\sigma:G^\pisharp A\to G^\pisharp B\in\M_L\equiv\,_R\M$ be such that 
$G^\pisharp t\ci \sigma\ci G^\pisharp t=G^\pisharp A$. Then the $^*\sigma$ constructed above satisfies
$\fsh t\ci\,^*\sigma\ci \fsh t=\fsh t$, so $\fsh t$ is von Neumann regular. By Lemma \ref{lem: von Neumann} 
$K t$ has a kernel in $\M^H_\fgp$ and by the Representation Theorem $t$ has one in $\C$. This means precisely 
that $t$ has a cokernel in $\C^\op$. 
If $k$ is a kernel of $t$ in $\C$ then $F k$ is a kernel of $F t$ and using the $^*(\under)$ functor on $_R\M_R$ it is easy to see that $G k$ is the cokernel of $G t$. Therefore $G$, and $G^\pisharp$, too, preserves such cokernels. 

This finishes the proof of that $G$ is a fiber functor.

In the Grothendieck topology $\GroT_G$ the covering sieves on $C\in\C^\op$ are the `cosieves' $S$ from $C$
such that $\{Gs|s\in S\}$ are jointly epimorphic on $GC$. This is equivalent to that $\{Fs|s\in S^*\}$ is jointly
epimorphic on $FC^*$. Denoting by $S^*\C$ the sieve generated by $S^*$ we obtain 
\[
\GroT_G(C)=\{S\text{ cosieve from }C\,|\,S^*\C\in\GroT_F(C^*)\}\,.
\]
If $\C$ is autonomous then $(\under)^*$ is an equivalence therefore $S^*\C=S^*$. So, $F$ is coarse iff
every sieve in $\GroT_F(C^*)$ contains the identity $C*$ iff every cosieve in $\GroT_G(C)$ contains $C$
iff $G$ is a coarse fiber functor on $\C^\op$.
\end{proof}

\begin{cor} \label{cor: dual representables}
Under the assumtions of the above Proposition all the representable presheaves $\C(C,\under)$ on $\C^\op$
are sheaves for the monoidal Grothendieck topology induced by $G$ on $\C^\op$.
\end{cor}

Since the pointwise left dual of the $G^\pisharp:\C^\op\to\M_L$ is $_R\M(GC,R)\cong FC$ the $L$-bialgebroid one constructs from the fiber functor $G$ is the `coopposite' right bialgebroid $H^{\text{coop}}$ of $H$. 
This means that $H^{\text{coop}}$ has the same underlying ring $G\am{\C}F$ as $H$ but the source and target maps are interchanged and the comultiplication is the opposite of $H$ (with $\obar{R}$ replaced by $\obar{L}$).
Therefore the right $H^{\text{coop}}$-comodule category can be identified with the category $^H\M$ of left
$H$-comodules. By the Representation Theorem we obtain the following
\begin{cor}
If $\bra\C,F\ket$ is a fiber functor with $\C$ left autonomous and Cauchy complete then the full subcategory $^H_\fgp\M\subset\,^H\M$
of left $H$-comodules that are f.g. projective as left $R$-modules is monoidally equivalent to $\C^\op$ via 
$C\mapsto \bra GC, \lambda_C\ket$ where $\lambda_C:GC\to H\obar{R} GC$, $g\mapsto\sum_i(g\am{C}x_C^i)\obar{R} f_C^i$.
\end{cor}

The bialgebroid $H=G\am{\C}F$ is a proper extension of $L=R^\op$ via the target map $t_H$ simply because
$t_H$ is a section of $\eps_H$. For arbitrary fiber functors $F$, $F'$ the ring $A=G\am{C}F'$ can still be a proper
extension of $L$ via $t_A$ without, however, any mapping analogous to $\eps_H$. The following Lemma is therefore
a nice application of the fact that $G$ is a fiber functor on $\C^\op$.

\begin{lem} \label{lem: A/L proper}
If $\C$ is a small additive left autonomous monoidal category and $F$, $F'$ are fiber functors on $\C$ then
$t_A:L\to A=G\am{\C}F'$, $l\mapsto l\am{I} 1_{R'}$ is a monic, so $A/L$ is a proper ring extension.
\end{lem}
\begin{proof}
If $F'$ is a fiber functor then so is its pointwise left dual $G'$. Therefore we have an adjunction $\check{\F'}\dashv \check{\G'}$ where $\check{\F'}=G'\am{\C}\under$ and, by the finitess condition on $G'$, $\check{\G'}X\cong X
\am{L'} F'$ for any right $L'$-module $X$. 
Let $\check{\eta'}$ denote the unit of this adjunction. Evaluated on the unit presheaf,
\[
\check{\eta'}_{\check{I}}:\check{I}\to\check{\G'}\check{\F'}\check{I}=
\check{\G'}(G'\am{\C}\check{I})\cong\check{\G'}(G'I) \cong G'I\am{L'} F'\cong F'
\]
is a monomorphism since $\check{I}$ is a $\T'$-sheaf (cf. Corollary \ref{cor: dual representables} and Lemma \ref{exa: F-topology}).
Now applying the exact $\check{\F}$ associated to the fiber functor $F$ we obtain the monic
\[
L\iso G\am{\C}\check{I}=\check{\F}\check{I}\longrarr{\check{\F}\check{\eta'}_{\check{I}}} 
\check{\F}F'\iso G\am{\C}F'=A
\]
which maps $l\mapsto l\am{I}I\mapsto l\am{I} 1_{R'}$ so it is equal to $t_A$.
\end{proof}

Next we study the Galois properties of $A$ as a right $H'$-comodule. Consider the mapping
\[
\beta_A:A\am{L}A\to A\obar{R'}H',\quad a\am{L}a'\mapsto a{a'}\nulT\obar{R'}{a'}\oneT\,.
\]
If $F'=F$, hence $A$ is the bialgebroid $H=G\am{\C}F$, we obtain the Galois map $\beta_H:H\oL H\to H\obar{R}H$
the invertibility of which is the defining condition for the right bialgebroid $H$ to be a right 
$\xover{R}$-Hopf algebra  \cite{Schauenburg: ddqg}, also called  Hopf algebroid in \cite{Phung}.
 
\Phung proves in \cite[Theorem 2.2.4]{Phung}, among others, the following statement: If $\C$ is left autonomous
then the $R$-bialgebroid $H$ is a Hopf algebroid, i.e., $\beta_H:H\oL H\to H\obar{R}H$ is invertible.
We can generalize this to all the $H'$-comodule algebras $A$ as follows.
\begin{pro} \label{pro: A H' Gal}
Let $F$, $F'$ be fiber functors on the small additive left autonomous category $\C$. Then $A=G\am{\C}F'$
is a right $H'$-Galois extension of the subalgebra $t_A:L\to A$, $l\mapsto l\am{I}1_{R'}$.
\end{pro}
\begin{proof}
Define a subfunctor ${F'}^{\co H'}$ of $F'$ by the equalizer
\begin{equation}\label{eq: F' co H'}
{F'}^{\co H'} \longmono{} F'\pair{\delta'}{F'\am{R'}1_{H'}} F'\obar{R'}H'
\end{equation}
the elements of which are the elements $\bra x',C\ket$ of $F'$ satisfying the equation
\[
\sum_i {x'}^i_C\am{R'}(f^i_C\am{C} x')\ =\ x'\am{R'}(1_L\am{I} 1_{R'})\,.
\]
A comparison with Equations (\ref{eq: A co H}) and (\ref{eq: G co H}) helps to recognize this equalizer as the one defining $\check{\T}\check{I}$ where $\check{\T}$ is the 
monoidal idempotent monad on the presheaf category $\check{\C}=\Add(\C,\Ab)$ associated to $G$ and $\check{I}=\C(I,\under)$ is the unit object of $\check{\C}$. By Corollary \ref{cor: dual representables} $\check{I}$
is a $\check{\T}$-sheaf therefore ${F'}^{\co H'}\cong \check{I}$. 
Now apply the left Kan extension $G\am{\C}\under$ of the flat functor $G$ to this equalizer and notice that
we obtain the equalizer
\[
A^{\co H'}\longmono{} A\parallelpair A\obar{R'}H'
\]
defining the coinvariant subalgebra of $A$. Therefore $A^{\co H'}\cong G\am{\C}\check{I}\cong GI\cong L$
proving that $_LA^{H'}$ is an $H'$-extension. 

Since $L$ is the coinvariant subalgebra of $A^{H'}$, the mapping $\beta_A$ is indeed the Galois map associated 
to the $H'$ comodule algebra $A$. We are left to show that $\beta_A$ is invertible.
At first we observe that $\beta_A$ has a factorization
\[
A\oL A\longrarr{\sim}G\am{\C}(F'\odot F')\longrarr{G\am{\C}\beta} G\am{\C}(F'\am{R'}H')\cong A\am{R'}H'
\]
where the first arrow is invertible since $G\am{\C}\under$ is a strong monoidal functor on $\check{\C}$ and
$\beta:F'\odot F'\to F'\am{R'}H'$, which no longer depends on the fiber functor $F$, provides the Galois map of $F'$
as a monoid object in the presheaf category $\check{\C}$. Like for presheaves over $\C$ we denote the elements
of the convolution product $F'\odot F'=\int^A\int^B F'A\ot F'B \ot\C(A\ot B,\under)$ by $[x,y,t]^C_{A,B}$.
We obtain the following formula 
\[
\beta([x,y,t]^C_{A,B})=\sum_i F't\ci F'_{A,B}(x\am{R'}x_B^i)\am{R'}(f_B^i\am{B}y)
\]
where $\sum_i x_B^i\am{R'}f_B^i\in F'B\am{R'}G'B$ is the dual basis for $G'B=(F'C)^*$.
The left autonomous structure allows to write down the inverse of $\beta$ as
\[
\beta_C^{-1}(z\am{R'}(h\am{B} y))=[F'_{C,B^*}(z\am{R'}w_B(h)),y,\runi_C\ci(C\ot\ev_B)\ci\asso^{-1}_{C,B^*,B}
]^C_{C\ot B^*,B}
\]
where $w$ denotes the natural isomorphism $w_B:G'B=(F'B)^*\iso F'B^*$. Indeed,
\begin{align*}
&\beta_C\ci\beta_C^{-1}(z\am{R'}(h\am{B} y))=\\
&={\sst\sum_iF'\runi_C\ci F'(C\ot\ev_B)\ci F'\asso^{-1}_{C,B^*,B}\ci F'_{C\ot B^*,B}\ci(F'_{C,B^*}\am{R'}F'B)
((z\am{R'}w_B(h))\am{R'}x_B^i)\obar{R'}\ (f_B^i\am{B}y)}=\\
&={\sst\sum_iF'\runi_C\ci F'(C\ot\ev_B)\ci F'_{C,B^*\ot B}\ci(F'C\am{R'}F'_{B^*,B})(z\am{R'}(w_B(h)\am{R'}x_B^i))
\ \obar{R'}\ (f_B^i\am{B}y)}=\\
&={\sst\sum_iz\cdot\bra h,x_B^i\ket\ \obar{R'}\ (f_B^i\am{B}y)}=\\
&=z\am{R'}(h\am{B} y)
\end{align*}
and 
\begin{align*}
&\beta_C^{-1}\ci\beta_C([x,y,t]^C_{A,B})=\\
&={\sst\sum_i\left[F'_{C,B^*}(F't\ci F'_{A,B}(x\am{R'}x_B^i)\am{R'}
w_B(f_B^i)),y,\runi_C\ci(C\ot\ev_B)\ci\asso^{-1}_{C,B^*,B}\right]^C_{C\ot B^*,B}}\\
&={\left[{\sst F'(t\ot B^*)\ci F'\asso_{A,B,B^*}\ci F'_{A,B\ot B^*}(x\am{R'}}\right.}
\underset{\sst F'\db_B(1_{R'})}{\underbrace{\sst\sum_iF'_{B,B^*}(x_B^i\am{R'}w_B(f_B^i))}}
{\sst\left. ),\, y,\, \runi_C\ci(C\ot\ev_B)\ci\asso^{-1}_{C,B^*,B}\right]^C_{C\ot B^*,B}}\\
&={\sst\left[F'((t\ot B^*)\ci\asso_{A,B,B^*}\ci(A\ot\db_B)\ci\runi_A^{-1})x,\ y,\ \runi_C\ci(C\ot\ev_B)
\ci\asso^{-1}_{C,B^*,B}\right]^C_{C\ot B^*,B}}\\
&={\sst\left[x,y,\runi_C\ci(C\ot\ev_B)\ci\asso^{-1}_{C,B^*,B}\ci((t\ot B^*)\ot B)\ci(\asso_{A,B,B^*}\ot B)\ci
((A\ot\db_B)\ot B)\ci(\runi^{-1}\ot B)\right]^C_{A,B}}\\
&={\sst\left[x,y,t\ci\runi_{A\ot B}\ci((A\ot B)\ot\ev_B)\ci\asso^{-1}_{A\ot B,B^*,B}\ci(\asso_{A,B,B^*}\ot B)\ci
((A\ot\db_B)\ot B)\ci(\runi^{-1}\ot B)\right]^C_{A,B}}\\
&=[x,y,t]^C_{A,B}
\end{align*}
show that $\beta$ is invertible and therefore so is $\beta_A$ for all fiber functors $F$ and $F'$.
\end{proof}

\subsection{Coarse fiber functors of corings} \label{ss: coarse}

Recall \cite{Mitchell} that an object $X$ in a cocomplete $\Ab$-category $\M$ is called \textit{small} if every morphism
$X\to\coprod_{i\in\I}Z_i$ into a coproduct factors through a finite subcoproduct $\coprod_{i\in\I_0}Z_i\into\coprod_{i\in\I}Z_i$. This is equivalent to that the hom-functor $\M(X,\under):\M\to\Ab$
preserves coproducts. Let $\proj(\M)$ denote the full subcategory of a cocomplete $\Ab$-category $\M$ the objects
of which are the small projective objects.
\begin{lem} \label{lem: small proj}
\begin{enumerate}
\item
If $\C$ is an essentially small additive category with splittings of idempotents then the small projective objects of the presheaf category $\Pre=\Add(\C^\op,\Ab)$ are precisely the representable presheaves.
\item
For any ring $R'$ the category $\proj(\M_{R'})$ coincides with the full subcategory $\M_{R'}^\fgp$ of f.g. projective modules.
\item
For an $R'$-coring $H'$ we have an inclusion $\proj(\M^{H'})\subset\M^{H'}_\fgp$, i.e., the forgetful functor
$\F^{H'}:\M^{H'}\to\M_{R'}$ preserves small projective objects.
\item
If $H$ is a coarse $R$-coring then $\proj(\M^H)=\M^H_\fgp$, i.e, the forgetful functor preserves and reflects
small projective objects.
\end{enumerate}
\end{lem} 
\begin{proof}
(1) It is a simple consequence of the Yoneda Lemma that representable presheaves are projective. That they are also
small follows from the usual structure of coproducts of abelian groups.
If $P$ is a projective presheaf then it is a summand of a coproduct of representables. If $P$ is also small then
it is a summand in a finite coproduct of representables. If $\C$ is additive then $P$ is a summand of a single representable and if idempotents split in $\C$ then $P$ is representable.

(2) follows from (1) since $\M_{R'}$ is a presheaf category.

(3) Observe that $\F^{H'}$ is doubly left adjoint for any coring hence its right adjoint
$\G^{H'}=\under\am{R'}H$ preserves epimorphisms and coproducts. It follows then from the adjunction isomorphism
$\M_{R'}(\F^{H'}P,X)\cong\M^{H'}(P,\G^{H'}X)$ that if $P$ is small or projective then so is $\F^{H'}P$.

(4) $H$ being coarse means that $\M^H$ is equivalent to the presheaf category $\Pre$ over $\C=\M^H_\fgp$
via $M\mapsto \M^H(\under,M)$. If $M\in\M^H_\fgp$ then this presheaf is representable, i.e., a small projective
object by (1). Therefore $\F^H$ reflects small projectives. Preservation follows from (3).
\end{proof}

Lacking of monoidal structure in this subsection we are using the term `fiber functor' to mean functors which
obey what $\fsh:\C\to\M_R$ does in Definition \ref{def: fiber}.

\begin{pro} \label{pro: coarse}
Let $\C$ be a small additive category with splittings of idempotents and let $F:\C\to\M_R$ be a fiber functor.
Associated to these data we have the left Kan extension functor $\F:\Pre\to\M_R$, the $R$-coring $H=G\am{\C}F$,
the comparison functor $\K:\Pre\to\M^H$, the idempotent monad $\T$ and the Grothendieck topology $\GroT_F$.
The following conditions are equivalent:
\begin{enumerate}
\item $F$ reflects split epimorphisms.
\item $\GroT_F$ is the coarsest Grothendieck topology on $\C$, i.e., $F$ is a coarse fiber functor.
\item $\K$ is an equivalence.
\item $\F$ is comonadic.
\item $\F$ is faithful.
\item $\T$ is isomorphic to the idempotent monad $\mathbf{1}$ on $\Pre$.
\item $M\mapsto\M^H(\under,M):{\M^H_\fgp}^\op\to\Ab$ is an equivalence of $\M^H$ with the category of presheaves over $\M^H_\fgp$, i.e., $H$ is a coarse coring.
\end{enumerate}
\end{pro}
\begin{proof}
$(1)\Rightarrow (2)$ Assume $S$ is a covering sieve on $C$ w.r.t. the topology $\GroT_F$, i.e., $\{Fs|s\in S\}$ is a
jointly epimorphic family of arrows to $FC$. Since $FC$ is finitely generated, there exists a finite subset$\{s_1,\dots, s_n\}\subset S$ and elements $x_i\in F(\dom s_i)$ such that $\{Fs_i x_i\}$ is a system of $R$-generators for $FC$. Since $\C$ is additive, we can construct a direct sum $B$ of the $\dom s_i$ and 
an arrow $t\in SB$ such that $Ft$ is an epimorphism. Using that $FC$ is projective we conclude that $Ft$, hence also $t$, splits therefore $S$ contains the identity $C$.

$(2)\Rightarrow (1)$ Let $t\in\C(B,C)$ be such that $Ft$ is split epi. Let $S$ be the sieve on $C$ generated by the
arrow $t$. Since $Ft$ alone is already epic, $S$ is a covering sieve on $C$. By assumption $S$ must contain the identity arrow therefore $t$ is split epi.

$(2)\Rightarrow(3)$ If every presheaf is a sheaf then $\nu$, together with $\theta$, is an isomorphism, hence $\K\dashv\Ll$ is an adjoint equivalence.

$(3)\Leftrightarrow(4)$ By definition of comonadicity.

$(3)\Rightarrow(6)$ The unit $\nu$ of $\T$ is easily seen to provide a monad isomorphism $\mathbf{1}\to\T:\Pre\to\Pre$.

$(6)\Rightarrow(2)$ If $\varphi:\mathbf{1}\to\T$ is a monad isomorphism then the unit preserving property
of monad morphisms implies that $\varphi_U=\nu_U$. Thus $\nu$ is invertible and every presheaf is a sheaf.

$(3)\Leftrightarrow(7)$ By the Representation Theorem $\K$ restricts to an equivalence of  the representables with
$\M^H_\fgp$ and $M\mapsto\M^H(\under,M)$ is the right adjoint $\Ll$ of $\K$.


$(4)\Rightarrow(5)$ Every comonadic functor is faithful.

$(5)\Rightarrow (1)$ Let $t\in\C(B,C)$ be such that $Ft=\F Yt$ is split epi. By faithfulness of $\F$ the arrow $Yt$ is epi and its target $YC$ being projective it is also a split epi. By the Yoneda Lemma the splitting morphism $YC\to YB$
must be $Ys$ for a unique $s\in\C(C,B)$ which is then a splitting morphism for $t$. 
\end{proof}

\subsection{An Ulbrich Theorem for Hopf algebroids}

If $\C$ is autonomous, i.e., it has both left and right dual objects, then for any pair of fiber functors $F$, $F'$
the bicomodule algebra $A=A_{F,F'}=G\am{\C}F'$ is a left $H$-Galois extension of $R'$ and a right $H'$-Galois extension of $L$ by Propositions \ref{pro: A H Gal} and \ref{pro: A H' Gal}. 
Also, by flatness of $F'$, $A$ is the filtered colimit of f.g. projective, hence flat $R$-modules $_RG$, hence $A_L\equiv\,_RA$ is flat. Similarly, since $G$ is flat as a functor on $\C^\op$ by Proposition \ref{pro: G fiber}, 
$A_{R'}$ is flat. But $F'C$ is f.g. projective also as left $R'$-module therefore $_{R'}A$ is flat, too. 

If $F$ is a coarse fiber functor then $G\am{\C}\under:\Add(\C,\Ab)\to \Ab$ is faithful by Proposition
\ref{pro: coarse} (5). Since $A\am{R'}\under:\,_{R'}\M\to \Ab$ is the composite of $G\am{\C}\under$ with the
faithful functor $X\mapsto F'\am{R'}X$, in this case $A_{R'}$ is faithfully flat.

\begin{cor}\label{cor: half Ulbrich}
If $F$ is a coarse and $F'$ is an arbitrary fiber functor on the small additive Cauchy complete autonomous 
monoidal category $\C$ then $A_{R'}$ is a faithfully flat left $H$-Galois extension 
such that $_RA$ is flat.
\end{cor}

This looks like as one half of an Ulbrich's \cite[Theorem 1.2]{Ulbrich89} since $A$ is the value at $H$ of the monoidal functor 
\[
\A:=\under\coten{H}A\cong (\under\coten{H}G)\am{C}F'=\F'\bimL:\ \M^H\to\,_{R'}\M_{R'}
\]
where $\bimL$ is an equivalence since $F$ is coarse. This functor is always exact but cannot be expected 
to be faithful unless $F'$ is also coarse. 
Certainly there are other properties of $A$ that can be shown
to hold for all fiber functor $F'$ but we don't know yet which of them imply the converse of the above Corollary.
Let us proceed gradually. 

The proof of the following Proposition is an adaptation of the proofs of  \cite[Theorem 1.2]{Ulbrich89} and
\cite[Theorem 5.6]{BB} .
\begin{pro} \label{pro: A}
Let $H$ be a right $R$-bialgebroid
and $A_{R'}$ a faithfully flat left $H$-Galois extension of a ring $R'$ such that $_RA$ is flat.  
Then the functor $\A=\under\coten{H}A:\M^H\to\,_{R'}\M_{R'}$ is a colimit preserving left exact strong
monoidal functor.
\end{pro}
\begin{proof}
Since $\A$ is the limit of a finite diagram of colimit preserving functors $\under\oR A$, $\under\obar{R}H\obar{R}A$, $\A$ preserves filtered colimits and therefore coproducts, too.

Consider the following isomorphism 
\begin{equation} \label{Gamma}
\Gamma_M:= \left(A\am{R'}(M\coten{H}A)\iso M\coten{H}(A\am{R'}A)\longrarr{M\coten{H}\gamma_A}
M\coten{H}(H\obar{R}A)\iso M\oR A\right)
\end{equation}
in $\_A\M_{R'}$ which is
natural in $M\in\M^H$ and maps $a\am{R'}(m\coten{} b)\mapsto m\oR ab$. Since $\under\oR A$ preserves and $A\am{R'}\under$ reflects both epimorphisms and monomorphisms, $\A$ is exact. 
This proves that $\A$ preserves all colimits and it is left exact.

In order to show strongness of the monoidal structure
\begin{align*}
\A_{M,N}:(M\coten{H}A)\am{R'}(N\coten{H}A)&\to (M\oR N)\coten{H} A\\
(m\coten{} a)\am{R'}(n\coten{}b)&\mapsto(m\oR n)\coten{}ab\\
\A_0:R'&\to R\coten{H}A\\
r'&\mapsto 1_R\coten{} s_A(r')
\end{align*}
look at the commutative diagram
\[
\begin{CD}
A\am{R'}(M\coten{H}A)\am{R'}(N\coten{H}A)@>\Gamma_M\am{R'}(N\coten{H}A)>>M\oR A\am{R'}(N\coten{H}A)\\
@V{A\am{R'}\A_{M,N}}VV @VV{M\oR\Gamma_N}V\\
A\am{R'}((M\oR N)\coten{H}A)@>\Gamma_{M\oR N}>>M\oR N\oR A
\end{CD}
\]
and use that $A\am{R'}\under$ reflects isomorphisms. $\A_0$ is invertible because the defining equalizer of
$R\coten{H}A$ exhibits it as the subobject $A^{\co H}\subset A$ which is represented by $s_A:R'\into A$ by assumption. 
\end{proof}

\begin{lem} \label{lem: AA left adj}
If $H$ is coarse then $\A:\M^H\to\,_{R'}\M_{R'}$ is left adjoint.
\end{lem}
 \begin{proof}
Since $\M^H$ is the presheaf category on $\C=\M^H_\fgp$, $\A$ will be shown to be left adjoint by Lemma \ref{lem: FF left adj} once we can show that $\A$ is the left Kan extension of its restriction to $\C$. 
Using that $H$ is coarse, so $M\mapsto\M^H(\under,M)$ is an equivalence $\M^H\to\Pre$, and that $\C$ is dense in $\M^H$ we can write any object $M$ as a colimit $\colim_i C_i$ of $C_i\in\C$ and compute the left Kan extension 
on $M$ as 
\[
\M^H(\under,\colim_i C_i)\am{\C}(\under\coten{H}A)\cong\colim_i \left(\M^H(\under,C_i)\am{\C}(\under\coten{H}A)\right)
\cong \colim_i(C_i\coten{H}A)
\]
which shows that $\A$ is the left Kan extension of its restriction to $\C$ iff $\A$ is cocontinuous.
But we have seen in the Proposition above that $\A$ is cocontinuous so the proof is complete.
\end{proof}

\begin{lem} \label{lem: F' finite}
In addition to the assumption made in Proposition \ref{pro: A} and Lemma \ref{lem: AA left adj} we assume that the
Galois map $\beta_H:H\oL H\to H\obar{R}H$ is invertible.
Then the functor $\A:\M^H\to\,_{R'}\M_{R'}$ satisfies the following finiteness condition: $\A M$ is f.g. projective as right $R'$-module provided the object $M\in\M^H$ is f.g. projective as right $R$-module.
\end{lem}
\begin{proof}
By Proposition \ref{pro: A} $\A$ is strong monoidal, in particular it preserves right dual objects. 
By a non-trivial result of \Phung \cite[Proposition 1.8.2]{Phung} the forgetful functor $\M^H\to \,_R\M_R$
reflects right dual objects. 
\end{proof}

\begin{cor}
Let $H$ be a coarse $R$-bialgebroid such that $\beta_H$ is invertible and let $A_{R'}$ be a faithfully flat left
$H$-Galois extension of a ring $R'$ such that $_RA$ is flat. Then the functor $\A:=\under\coten{H}A:\M^H\to \,_{R'}\M_{R'}$ factors through the forgetful
functor $\F^{H'}:\M^{R'}\to\,_{R'}\M_{R'}$ of an $R'$-bialgebroid $H'$ via the reflection $\K'$
of a monoidal localization $\Ll':\M^{H'}\to\M^H$.
\end{cor}

We shall denote by $F'$ the restriction of $\A$ to $\C=\M^H_\fgp$ and our purpose is to show that $F'$ is a fiber 
functor. 
In order to show that it is faithful we need $A/L$ to be proper, see Lemma \ref{lem: A/L proper}.

\begin{lem} \label{lem: F' faithful}
Let $H$ be a coarse $R$-bialgebroid such that $\beta_H$ is invertible and let $A_{R'}$ be a faithfully flat left
$H$-Galois extension of a ring $R'$ such that $_RA\equiv A_L$ is flat and assume that $A/L$ is a proper ring extension. Then $F':\C\to\,_{R'}\M_{R'}$ is faithful.
\end{lem}
\begin{proof}
For a proper ring extension $L\to A$ and for any left $L$-module map $t:M\to N$ if $N$ is f.g. projective then
$A\am{L}t=0$ implies $t=0$. Applying this on the right hand side of the isomorphism $\Gamma$ in (\ref{Gamma})
we obtain that $\A t=0$ implies $t=0$. 

(This argument can be used to show that if $\A t$ is an isomorphism then $t$ is monic. However, invertibility
of $t$ does not seem to follow because $\coker t$ can be completely torsion. )
\end{proof}

After these preparations we can formulate the main result of this section.

\begin{thm}\label{thm: Ulbrich}
Let $\C$ be a small additive Cauchy complete autonomous monoidal category and $F$ a coarse fiber functor on $\C$ with base ring $R$. Let $H$ denote the coarse biagebroid associated to $F$.
Then for each ring $R'$ there is a bijection between the following two sets of data:
\begin{description}
\item[Fib] isomorphy classes of fiber functors $F'$ on $\C$ with base ring $R'$
\item[Gal] isomorphy classes of left $H$-Galois extensions $A_{R'}$ satisfying the following properties:
\begin{enumerate}
\item $A_{R'}$ is faithfully flat
\item $_RA$ is flat and $R^\op\into A$ is a proper ring extension
\item for all $M\rarr{t}N$ in $\M^H_\fgp$ 
\begin{enumerate}
\item $t\coten{H}A$ invertible $\Rightarrow$ $t$ is invertible
\item $t\coten{H}A$ split epi in $\M_{R'}$ $\Rightarrow$ $\ker t\in\M^H_\fgp$
\item $t\coten{H}A$ von Neumann regular in $\M_{R'}$ $\Rightarrow$ $\coker t\in\M^H_\fgp$.
\end{enumerate}
\end{enumerate}
\end{description}
The bijection is induced on the equivalence classes by the following mappings: If $F'$ is a fiber functor then
$A=G\am{\C}F'$ is a Galois extension where $G$ is the pointwise left dual of $F$. If $A$ is a Galois extension
then $F'C:=FC\coten{H}A$ defines a fiber functor.
\end{thm}
\begin{proof}
By the Representation Theorem for $F$ we can identify $\C$ with $\M^H_\fgp$ and $\Pre$ with $\M^H$.
We note that the given assumptions imply that $\M^H_\fgp$ coincides with the full subcategories 
of the small projective, of the right dual and also of the left dual objects in $\M^H$.

(\textbf{Fib}$\to$\textbf{Gal}) If $F'$ is a fiber functor then Corollary \ref{cor: half Ulbrich} shows that $A=G\am{\C}F'$ is a left $H$-Galois extension of the subalgebra $R'\rarr{s_A}A$ satisfying (1) and using also
Lemma \ref{lem: A/L proper} it satisfies (2). The functor $\A=\under\coten{H}A:\M^H\to\,_{R'}\M_{R'}$ 
which maps an object $M$ to $M\coten{H}(G\am{\C}F')\cong
(M\coten{H}G)\am{\C}F'\cong\M^H(F\under,M)\am{\C}F'$
is isomorphic to $\F'\Ll$ where the localization $\Ll$ is an equivalence by the coarseness assumption.
Therefore the restriction of $\A$ to $\M^H_\fgp$ is isomorphic to $F'$ itself. Hence the properties
(3.a-b-c) all follow from respective properties of $F'$ listed in Definition \ref{def: fiber}.

(\textbf{Gal}$\to$\textbf{Fib}) If $A$ is a Galois extension with the given properties then we can construct 
the functor $F'$ as the restriction of $\A=\under\coten{H}A:\M^H\to\,_{R'}\M_{R'}$ to $\M^H_\fgp$.
By Proposition \ref{pro: A} and Lemma \ref{lem: AA left adj} $\A$ is strong monoidal left exact left adjoint and
and it is the left Kan extension of its restriction $F'$ which is then strong monoidal and flat by Lemma 
\ref{lem: FF left adj} and by Proposition \ref{pro: esss}.
Faithfulness of $F'$ follows from assumption (2) by Lemma \ref{lem: F' faithful}.
$F'$ satisfies the finiteness condition by Lemma \ref{lem: F' finite}.
$F'$ reflects isomorphisms by assumption (3.a). 
That $\M^H_\fgp$ has kernels of arrows for which $F' t$ is split epi in $\M_{R'}$ follows from (3.b) and
that it has cokernels of arrows for which $F' t$ is von Neumann regular in $\M_{R'}$ follows from (3.c).
Also $F'$ preserves the latter cokernels since $\coker$ (and of course also $\ker$) in (3) is understood in the abelian category $\M^H$ on which the left Kan extension $\F'\equiv\A$ is left adjoint.
Thus all requirements of Definition \ref{def: fiber} are satisfied by $F'$.

(\textbf{Fib}$\to$\textbf{Gal}$\to$\textbf{Fib}) We have already seen that $\A=\under\coten{H}(G\am{\C}F')$
is isomorphic to $\F'\Ll$ therefore composing it with the comparison functor $\K:\Pre\to\M^H$ we have
the isomorphism $\F'\nu^{-1}:\A\K\iso\F'$ and then also
\[
\A F\iso\A\K Y\longrarr{\F'\nu^{-1}Y}\F' Y\iso F'
\]
is an isomorphism. (In the more paranoid notation used in earlier sections we should write here $K$ instead of $F$).
 
(\textbf{Gal}$\to$\textbf{Fib}$\to$\textbf{Gal}) 
If $F'C=FC\coten{H}A$ for a Galois extension $A$ then we have the isomorphisms $G\am{\C}F'\iso
(G\am{\C}F)\coten{H}A=H\coten{H}A\iso A$ where the first isomorphism follows as in the proof of Lemma 
\ref{lem: AA left adj}.
\end{proof}

\begin{rmk}
In the language of corings the left $H$-comodule $A$ of the above Theorem can be thought of as
$h_H(B,H)$ of a quasi-finite injector $R'$-object $B$ in the category $\M^H$. As a matter of fact,
by the finiteness condition $\A=\under\coten{H}A:\M^H\to\M_{R'}$ is doubly left adjoint so its right adjoint
$\B$ preserves colimits therefore $\B X\cong\B(X\am{R'}R')\cong X\am{R'}B$ with $B$ quasi-finite. 
In this context the left adjoint of $\B$ is called the cohom functor and is denoted by $h_H(B,\under)$. 
$B$ is an injector since $\A$ is left exact \cite[23.7]{Brz-Wis}. The left exact comonad $\A\B$ on $\M_{R'}$ is left adjoint therefore an $R'$-coring, called the coendomorphism coring $h_H(B,B)$ of $B$. This is the underlying 
coring of the bialgebroid $H'$ associated to the fiber functor $F'$ of the Theorem. 
\end{rmk}




\subsection{Invertible antipodes}

If we assume that $\C$ is not only autonomous but a monoidal natural isomorphism $(\under)^*\iso\,^*(\under)$
between left and right duals exists as well then we can show that an invertible antipode exists on the bialgebroid
$H$, so $H$ is a Hopf algebroid in the sense of \cite{B-Sz}.

Since we are working with right bialgebroids, we need the opposite co-opposite version of the axioms \cite{B-Sz}
which are these:

An antipode for a right bialgebroid $H$ over $R$ is an isomorphism $S:H\to H$ of abelian groups such that
\renewcommand{\labelenumi}{(S-\arabic{enumi})}
\begin{enumerate}
\item 
$S\ci t_H=s_H$
\item 
$S(h h')=S(h')S(h)$
\item 
$S(h\twoT)\oneT\obar{R} h\oneT S(h\twoT)\twoT=S(h)\obar{R} 1_H$
\item
$h\twoT S^{-1}(h\oneT)\oneT\obar{R} S^{-1}(h\oneT)\twoT=1_H\obar{R} S^{-1}(h)$
\end{enumerate}
\renewcommand{\labelenumi}{(\arabic{enumi})}
for all $h,h'\in H$.

Choosing left duality data $C^*$, $\ev_C:C^*\ot C\to I$, $\db_C:I\to C\ot C^*$ for each object $C$ we have the left 
dual object functor $(\under)^*:\C\to\C^{\op,\rev}$ with strong monoidal structure
\[
u:I\iso I^*,\qquad v_{C,D}:D^*\ot C^*\iso (C\ot D)^*
\]
with all arrows in the sense of $\C$ (then this is actually the opmonoidal data).
For the fiber functor $F:\C\to\,_R\M_R$ we can define its left dual as the functor $GC:=FC^*\cong(FC)^*$
which is then also strong monoidal as a functor $G:\C^\op\to \,_R\M_R^\rev\equiv\,_L\M_L$ with structure maps
\begin{equation} \label{G_0 G_2}
G_0=F u\ci F_0\,,\qquad G_{C,D}=Fv_{C,D}\ci F_{D^*,C^*}\,.
\end{equation}
Since right duality data $^*C$, $\tilde{\ev}_C$, $\tilde{\db}_C$ also exist we have the right dual object functor
$^*(\under)$ and the composite functors $^*((\under)^*)$ and $(^*(\under))^*$ are monoidally isomorphic to the identity functor. Existence of a monoidal natural isomorphism $(\under)^*\iso\,^*(\under)$ is equivalent to the
existence of another one, $\vartheta_C:C\to C^{**}$. Since the double dual is a strong monoidal endofunctor, actually this is equivalent to the existence of any monoidal natural transformation $\vartheta_C:C\to C^{**}$
by Lemma \ref{lem: Saavedra}.

We suppose we have given only left duality data and $\vartheta$ and we introduce right duality data by setting $^*C:=C^*$ and
\[
\tilde{\ev}_C:=\ev_{C^*}\ci(\vartheta_C\ot C^*),\qquad \tilde{\db}_C:=(C^*\ot\vartheta_C^{-1})\ci\db_{C^*}
\]
Monoidality of the natural isomorphism $\vartheta$ is expressed by the relations
\begin{align}
\label{vartheta multip}
\vartheta_{B\ot C}&=v^{-1 *}_{B,C}\ci v_{C^*,B^*}\ci(\vartheta_B\ot \vartheta_C)\\
\label{vartheta unit}
\vartheta_I&=u^{-1 *}\ci u\,.
\end{align}
We shall also need left and right duality data for $FC$, $C\in\ob\C$ which are chosen as in (\ref{eq: ev FC}),
(\ref{eq: db FC}) and (\ref{eq: evtilde FC}), (\ref{eq: dbtilde FC}) and this entails that $FC^*$ is the common left and right dual of the bimodule $FC$.
We shall use the notation 
\[
\bra y,x\ket:=\ev_{FC}(y\oR x)\,,\qquad \sum_i x_C^i\oR y_C^i:=\db_{FC}(1_R) \,.
\]
With this choice of left duality data the bialgebroid structure of $H$ given in the proof of Proposition 
\ref{pro: finfun - H} take the form
\begin{align*}
H&=\int^C FC^*\ot FC\\
s_H(r)&=G_0(1_L)\am{I} F_0(r)\\
t_H(r)&=G_0(r)\am{I}F_0(1_R)\\
(y\am{A}x)(y'\am{B}x')&=G_{A,B}(y\oL y')\am{A\ot B}F_{A,B}(x\oR x')\\
\cop_H(y\am{B}x)&=\sum_i(y\am{B}x_B^i)\obar{R}(y_B^i\am{B}x)\\
\eps_H(y\am{B}x)&=\bra y,x\ket\,.
\end{align*}

\begin{pro}
Let $\C$ be a small additive monoidal category with left duals and with a monoidal natural isomorphism
$\vartheta_C:C\to C^{**}$. If $F:\C\to\,_R\M_R$ is a strong monoidal functor with image in the subcategory of
right dual bimodules then the bialgebroid associated to $F$ by \cite[Theorem 2.2.4]{Phung}, see also
Proposition \ref{pro: finfun - H}, has an invertible antipode
\begin{equation} \label{eq: S}
S(y\am{B}x)=F\vartheta_B(x)\am{B^*} y\,.
\end{equation}
\end{pro}
\begin{proof}
$S$ is well-defined since for $x'=Ft x$, $t\in\C(B,C)$, $x\in FB$, $y=Ft^* y'$, $y'\in FC^*$
\[
\vartheta_B(x)\am{B^*}y=Ft^{**}\ci\vartheta_B(x)\am{C^*}y'=\vartheta_C(x')\am{C^*}y'\,.
\]
Verifying axiom (S-1) is easy using (\ref{vartheta unit}) and (\ref{G_0 G_2}),
\begin{align*}
S(t_H(l))&=S(G_0(l)\am{I} F_0(1_R))=F\vartheta_I\ci F_0(1_R)\am{I^*} G_0(l)=\\
&=F(u^*\ci\vartheta_I)\ci F_0(1_R)\am{I}F(u^{-1})\ci G_0(l)=\\
&=Fu\ci F_0(1_R)\am{I} Fu^{-1}\ci G_0(l)=s_H(l),\qquad l\in R\,.
\end{align*}
The antimultiplicativity axiom (S-2) follows from the calculation
\begin{align*}
&S(y\am{C}x)S(y'\am{B}x')=\\
&=G_{C^*,B^*}(\vartheta_C(x)\oL\vartheta_B(x'))\am{C^*\ot B^*}F_{C^*,B^*}(y\oR y')=\\
&=G_{C^*,B^*}(\vartheta_C(x)\oL\vartheta_B(x'))\am{C^*\ot B^*}Fv^{-1}_{B,C}\ci G_{B,C}(y'\oL y)=\\
&=Gv^{-1}_{B,C}\ci G_{C^*,B^*}(\vartheta_C(x)\oL\vartheta_B(x'))\am{(B\ot C)^*}G_{B,C}(y'\oL y)=\\
&=F(v^{-1 *}_{B,C}\ci v_{C^*,B^*})\ci F_{B^{**},C^{**}}(\vartheta_B(x')\oR\vartheta_C(x))
\am{(B\ot C)^*}G_{B,C}(y'\oL y)=\\
&=F(v^{-1 *}_{B,C}\ci v_{C^*,B^*}\ci (\vartheta_B\ot\vartheta_C))\ci F_{B,C}(x'\oR x)\am{(B\ot C)^*}G_{B,C}(y'\oL y)=\\
&=S((y'\am{B}x')(y\am{C}x))
\end{align*}
where we used (\ref{vartheta multip}) in the last line.

In order to verify (S-3) on $h=y\am{C}x$ we proceed as follows. 
\begin{align*}
&\sum_iS(y_C^i\am{C}x)\oneT\oR(y\am{C}x_C^i)S(y_C^i\am{C}x)\twoT=\\
&=\sum_i\sum_j(F\vartheta_C(x)\am{C^*}x_{C^*}^j)\oR(y\am{C}x_C^i)(y_{C^*}^j\am{C^*} y_C^i)=\\
&=\sum_j(F\vartheta_C(x)\am{C^*}x_{C^*}^j)\oR(G_{C,C^*}(y\oL y_{C^*}^j)\am{C\ot C^*}
\underset{F\db_C(1_R)}{\underbrace{\sum_iF_{C,C^*}(x_C^i\oR y_C^i)}})=\\
&=\sum_j(F\vartheta_C(x)\am{C^*}x_{C^*}^j)\oR(
\underset{\bra y_{C^*}^j,y\ket\cdot G_0(1_L)}{\underbrace{G\db_C\ci G_{C,C^*}(y\oL y_{C^*}^j)}}
\am{I}F_0(1_R))=\\
&=(F\vartheta_C(x)\am{C^*}\sum_jx_{C^*}^j\cdot\bra y_{C^*}^j,y\ket)\oR 1_H=
(F\vartheta_C(x)\am{C^*}y)\oR 1_H=\\
&=S(h)\oR 1_H
\end{align*}

Before proving axiom (S-4) the reader should check the following formula for the inverse antipode:
\begin{equation} \label{eq: S^-1}
S^{-1}(y\am{B}x)=F\vartheta_B(x)\am{B^*}F(\vartheta^{-1}_{C^*}\ci\vartheta^{* -1}_C)(y).
\end{equation}
Then, putting again $h=y\am{C}x$, the calculation
\begin{align*}
&\sum_i(y_C^i\am{C}x)S^{-1}(y\am{C}x_C^i)\oneT\oR S^{-1}(y\am{C}x_C^i)\twoT=\\
&=\sum_{i,j}(y_C^i\am{C}x)(F\vartheta_C(x_C^i)\am{C^*}x_{C^*}^j)\oR(y_{C^*}^j\am{C^*}F(\vartheta_{C^*}^{-1}\ci
\vartheta_C^{* -1})(y))=\\
&=\sum_{i,j}(G_{C,C^*}(
\underset{F\vartheta_C^*(y_{C^{**}}^i)\oL x_{C^{**}}^i}{\underbrace{y_C^i\oL F\vartheta_C(x_C^i)}})
\am{C\ot C^*}F_{C.C^*}(x\oR x_{C^*}^j))\oR(y_{C^*}^j\am{C^*}F(\vartheta_{C^*}^{-1}\ci\vartheta_C^{* -1})(y))\\
&=\sum_j(
\underset{G\ev_{C^*}\ci G_0(1_L)}{\underbrace{\sum_i G_{C^{**},C^*}(y_{C^{**}}^i\oL x_{C^{**}}^i)}}
\am{C^{**}\ot C^*}F_{C^{**},C^*}(F\vartheta_C(x)\oR x_{C^*}^j))\\
&\hskip 7cm\oR(y_{C^*}^j\am{C^*}F(\vartheta_{C^*}^{-1}\ci\vartheta_C^{* -1})(y))=\\
&=\sum_j(G_0(1_L)\am{I}F_0(1_R)\cdot\bra F\vartheta_C(x),x_{C^*}^j\ket)
\oR(y_{C^*}^j\am{C^*}F(\vartheta_{C^*}^{-1}\ci\vartheta_C^{* -1})(y))=\\
&=1_H\oR S^{-1}(h)
\end{align*}
proves axiom (S-4).
\end{proof}
If the category $\C$ is pivotal \cite[5.1]{Freyd-Yetter} then a comparison of (\ref{eq: S}) and (\ref{eq: S^-1}) 
immediately implies that the antipode is involutive.



\end{document}